\documentclass[12pt]{article}
\usepackage{amsthm}
\usepackage{amsfonts}
\usepackage{epsfig}
\usepackage{hyperref}
\newcommand{\KK}{{\cal K}}
\newcommand{\CC}{{\cal C}}
\newcommand{\FF}{{\cal F}}
\newcommand{\SSS}{{\cal S}}
\newcommand{\JJ}{{\cal J}}
\newcommand{\TT}{{\cal T}}
\newcommand{\oldK}{{\ensuremath{10^{28}}}}
\newcommand{\newK}{{\ensuremath{10^{21}}}}
\newtheorem{theorem}{Theorem}
\newtheorem{corollary}[theorem]{Corollary}
\newtheorem{lemma}[theorem]{Lemma}
\newcommand{\claim}[2]{\begin{equation}\mbox{\parbox{\linewidth}{{\em #2}}}\label{#1}\end{equation}}
\newcommand{\oururl}{\url{http://arxiv.org/abs/1305.2670}}

\title{$4$-critical graphs on surfaces without contractible $(\le\!4)$-cycles}
\author{Zden\v{e}k Dvo\v{r}\'ak\thanks{Charles University, Prague, Czech
Republic. E-mail: \protect\href{mailto:rakdver@iuuk.mff.cuni.cz}{\protect\nolinkurl{rakdver@iuuk.mff.cuni.cz}}.
Supported by
the Center of Excellence -- Inst. for Theor. Comp. Sci., Prague (project P202/12/G061 of Czech Science Foundation), and
by project LH12095 (New combinatorial algorithms - decompositions, parameterization, efficient solutions) of Czech Ministry of Education.}\and
Bernard Lidick\'y\thanks{
Charles University, Prague, Czech Republic and University of Illinois at Urbana-Champaign, Urbana, USA. E-mail:
\protect\href{mailto:lidicky@illinois.edu}{\protect\nolinkurl{lidicky@illinois.edu}}.
Supported by NSF grant DMS-1266016.}
}
\date{\today}

\begin{document}
\maketitle

\begin{abstract}
We show that if $G$ is a $4$-critical graph embedded in a fixed surface $\Sigma$ so that every contractible cycle has length at
least $5$, then $G$ can be expressed as $G=G'\cup G_1\cup G_2\cup\ldots\cup G_k$, where $|V(G')|$ and $k$ are bounded by a constant
(depending linearly on the genus of $\Sigma$)
and $G_1$, \ldots, $G_k$ are graphs (of unbounded size) whose structure we describe exactly.  The proof is computer-assisted---we
use computer to enumerate all plane $4$-critical graphs of girth $5$ with a precolored cycle of length at most $16$, that
are used in the basic case of the inductive proof of the statement.
\end{abstract}

\section{Introduction}
The problem of $3$-coloring triangle-free graphs embedded in a fixed surface,
motivated by the celebrated Gr\"otzsch theorem~\cite{grotzsch}, has drawn much
attention. 
Thomassen~\cite{thom-torus} showed that if a graph $G$ is embedded in
the torus or the projective plane so that every contractible
cycle has length at least $5$, then $G$ is $3$-colorable.
Thomas and Walls~\cite{thomwalls} showed that graphs of girth at least $5$ embedded in the Klein bottle are $3$-colorable
and gave a description of all $4$-critical graph on the Klein bottle without contractible cycles of length
at most $4$.
Gimbel and Thomassen~\cite{gimbthom} showed that graphs of girth 6 embedded in the double torus are 3-colorable
and described triangle-free projective plane graphs that are not 3-colorable.

Recently, Dvo\v{r}\'ak, Kr\'al' and Thomas~\cite{proof-lincrit} gave a structural description of $4$-critical (i.e.,
minimal non-$3$-colorable) triangle-free graphs embedded in a fixed surface,
and used this result to give a linear-time algorithm to decide $3$-colorability
of such graphs.  In particular, this description implies the following.

\begin{theorem}[Dvo\v{r}\'ak, Kr\'al' and Thomas~\cite{proof-lincrit}]\label{thm-dktsim}
There exists an absolute constant $K$ such that every $4$-critical graph of girth $5$
embedded in a surface of genus $g$ has at most $Kg$ vertices.
\end{theorem}

This improves a doubly-exponential bound by Thomassen~\cite{thom-surf}.  
Let us note that the linear bound was proved by Postle~\cite{luke-hyperbolic}
also for girth 5 and 3-list coloring.
Somewhat
unsatisfactorily, the bound on $K$ given by Dvo\v{r}\'ak et al.~\cite{proof-lincrit}
is rather weak, proving that $K<\oldK$ (we are not aware of any non-trivial lower bound, and suspect that $K\approx 100$ should suffice).
One of the reasons why this bound is so large
is hidden in the handling of the basic case of the induction, where they prove
that if $G$ is a plane graph with exactly two faces $C_1$ and $C_2$ of length at most $4$,
all other cycles have length at least $5$ and the distance between $C_1$ and $C_2$ is at least $1500000$,
then any precoloring of $C_1$ and $C_2$ extends to a proper coloring of $G$ by three colors.
In this paper, we give a computer-assisted proof showing that it suffices to assume
that the distance between $C_1$ and $C_2$ is at least $4$, which can be used to show that $K<\newK$.
We were originally hoping in a bigger improvement on $K$.

\begin{theorem}\label{thm-dist}
Let $G$ be a plane graph and let $C_1$ and $C_2$ be faces of $G$ of length at most $4$, such that every cycle in $G$ distinct
from $C_1$ and $C_2$ has length at least $5$.  If the distance between $C_1$ and $C_2$ is at least $4$, then every precoloring of $C_1\cup C_2$
extends to a proper $3$-coloring of $G$.
\end{theorem}

Combining these results, we give a more precise description of the structure
of the $4$-critical graphs without contractible cycles of length at most $4$
(a cycle is \emph{contractible} if it separates the surface to two parts and at least one of them is homeomorphic to the open disc).
Thomassen~\cite{thom-torus} showed that every graph of girth at least $5$ embedded in the projective plane or in the torus
is $3$-colorable.  Actually, he proved a stronger claim that enables him to apply induction:
every graph embedded in the projective plane or in the torus so that all contractible
cycles have length at least $5$ (but there may be non-contractible triangles or $4$-cycles)
is $3$-colorable.   Thus, it might seem possible to strengthen Theorem~\ref{thm-dktsim} by allowing non-contractible triangles or $4$-cycles.
However, Thomas and Walls~\cite{thomwalls} exactly characterized $4$-critical graphs embedded
in the Klein bottle so that no contractible cycle has length at most $4$,
showing that there are infinitely many such graphs.

Let  $\CC$ be the class of plane graphs that can be obtained from a
cycle of length $4$ by a finite number of repetitions of the following
operation: given a graph with the outer face $v_1v_2v_3v_4$ of length $4$ such
that $v_1$ and $v_3$ have degree two, add new vertices $v'_2$, $v'_3$ and
$v'_4$ and edges $v_1v'_2$, $v'_2v'_3$, $v'_3v'_4$, $v'_4v_1$ and $v_3v'_3$,
and let $v_1v'_2v'_3v'_4$ be the outer face of the resulting graph. Examples
of elements of $\CC$ are the $4$-cycle, and the graphs $Z_3$ in Figure~\ref{fig-44}, $A_{11}$ in
Figure~\ref{fig-join2b} and $Z_3A_{11}a$ and $Z_3A_{11}b$ in
Figure~\ref{fig-join3c}.  Let $\CC'$ be the class of graphs
obtained from those in $\CC$ 
by adding a chord joining the pair of vertices of degree two in both of the $4$-faces
 (thus introducing $4$ triangles).  Each graph in $\CC'$ can be embedded in the Klein bottle
by putting crosscaps on both newly added chords; the $4$-faces of a graph in $\CC$ thus become $6$-faces
in such an embedding of the corresponding graph in $\CC'$.

Thomas and Walls~\cite{thomwalls} proved that a graph embedded in the Klein bottle without contractible $(\le\!4)$-cycles
is $4$-critical if and only if it belongs to $\CC'$.  We extend this result to other surfaces.

\begin{theorem}\label{thm-main}
There exists a function $f(g)=O(g)$ with the following property.  Let $G$ be a $4$-critical graph embedded in a surface $\Sigma$ of genus $g$ so that every
contractible cycle has length at least $5$.  Then $G$ contains a subgraph $H$ such that
\begin{itemize}
\item $|V(H)|\le f(g)$, and
\item if $F$ is a face of $H$ that is not equal to a face of $G$, then $F$ has exactly two boundary walks, each of the walks has length  $4$,
and the subgraph of $G$ drawn in the closed region corresponding to $F$ belongs to $\CC$.
\end{itemize}
\end{theorem}

\section{Preliminaries}

In order to state more technical results necessary to prove Theorems~\ref{thm-dist} and \ref{thm-main}, we need a few definitions.
The graphs considered in this paper are undirected and without loops and parallel edges.
By a {\em coloring} of a graph we always mean a proper $3$-coloring.
By the {\em genus} $g(\Sigma)$ of a surface $\Sigma$ we mean the Euler genus, i.e., $2h+c$, where $h$ is the number
of handles and $c$ is the number of crosscaps attached to the sphere in order to create $\Sigma$.  If $G$ is a graph embedded
in $\Sigma$, a {\em face} $F$ of $G$ is a maximal connected open subset of $\Sigma-G$.  Sometimes, we also let $F$ stand
for the subgraph of $G$ consisting of the edges of $G$ contained in the closure of $F$.  We let $\ell(F)$ be the sum of
the lengths of the boundary walks of $F$ in $G$.

A graph $G$ is {\em $k$-critical} if $G$ is not $(k-1)$-colorable, but every proper subgraph
of $G$ is $(k-1)$-colorable.
A well-known result of Gr\"otzsch~\cite{grotzsch} states that all triangle-free planar graphs are $3$-colorable,
i.e., there are no planar triangle-free $4$-critical graphs.
Since the cycles of length $4$ can be easily eliminated, the main part of the proof of Gr\"otzsch's theorem
concerns graphs of girth $5$.  Generalizing this result,  Thomassen~\cite{thom-surf} proved that there exists a function $f$
such that every $4$-critical graph of girth $5$ and genus $g$ has at most $f(g)$ vertices (where $f$ is double-exponential in $g$),
and thus the number of such graphs is finite.  This was later improved by Dvo\v{r}\'ak et al.~\cite{proof-lincrit},
by showing that the number of vertices of such a graph is at most linear in $g$ (Theorem~\ref{thm-dktsim}).  Both the original result of Thomassen
and its improvement allow a bounded number of vertices to be precolored.  To state this generalization, we need to extend the
notion of a $4$-critical graph.

There are two natural ways one can define a critical graph with precolored vertices.
Consider a graph $G$ and a subgraph (not necessarily induced) $S\subseteq G$.
We call $G$ {\em strongly $S$-critical} if there exists a coloring of $S$ that does not extend to a coloring
of $G$, but extends to a coloring of every proper subgraph $G'\subset G$ such that $S\subseteq G'$.
We say that $G$ is {\em $S$-critical} if for every proper subgraph $G'\subset G$ such that $S\subseteq G'$,
there exists a coloring of $S$ that does not extend to a coloring of $G$, but extends to a coloring of $G'$.
We call a (strongly) $S$-critical graph $G$ {\em nontrivial} if $G\neq S$.
Note that every strongly $S$-critical graph is also $S$-critical, but the
converse is false (for example, if $G$ is a cycle $S$ with two chords, then
$G$ is $S$-critical, but not strongly $S$-critical).  Also, $G$ is $\emptyset$-critical (or
strongly $\emptyset$-critical) if and only if $G$ is $4$-critical.

Dvo\v{r}\'ak et al.~\cite{proof-lincrit} bounded the size of critical graphs as follows:
\begin{theorem}[Dvo\v{r}\'ak et al.~\cite{proof-lincrit}]\label{thm-dkt}
Let $K=\oldK$.  Let $G$ be a graph
embedded in a surface $\Sigma$ of genus $g$ and let $\{F_1, F_2, \ldots, F_k\}$
be a set of faces of $G$ such that the open region corresponding to $F_i$ is
homeomorphic to the open disk for $1\le i\le k$.  If $G$ is $(F_1\cup
F_2\cup\ldots\cup F_k)$-critical and every cycle of length of at most $4$ in $G$ is
equal to $F_i$ for some $1\le i\le k$, then
$$|V(G)|\le \ell(F_1)+\ldots+\ell(F_k)+K(g+k).$$
\end{theorem}

Let us note that such a claim does not hold without the restriction on the
cycles of length $4$, since Youngs~\cite{Youngs} gave a construction of an infinite
family of $4$-critical triangle-free graphs that can be embedded in any surface distinct from
the sphere.

Analogously, we will prove a generalization of Theorem~\ref{thm-main}
allowing a bounded number of precolored vertices (Theorem~\ref{thm-maingen}).
It is easy to reduce the proof to the case that $\Sigma$
is the sphere and exactly two cycles are precolored.  In this case, we say
that the graph is embedded in the {\em cylinder}, and we call the precolored
cycles the {\em boundaries} of the cylinder.  In fact, it suffices to consider
the case that both boundaries have length at most $4$.  By cutting along cycles
of length at most $4$, such a graph decomposes to a possibly large number of
graphs embedded in the cylinder such that the only cycles of length at most $4$
are the boundaries.  The main part of our proof is based on an enumeration of
such graphs:

\begin{theorem}\label{thm-cyl}
Let $G$ be a connected graph embedded on the cylinder with distinct boundaries $C_1$ and $C_2$ such that $\ell(C_1), \ell(C_2)\le 4$ and every cycle in $G$ distinct
from $C_1$ and $C_2$ has length at least $5$.  If $G$ is $(C_1\cup C_2)$-critical, then $G$ is isomorphic to one of the graphs drawn in
Figures~\ref{fig-44} and \ref{fig-34}.
\end{theorem}

It is straightforward to check that the distance between the boundaries in the described critical graphs is at most three;
hence, Theorem~\ref{thm-cyl} implies Theorem~\ref{thm-dist}.

\begin{figure}
\begin{center}
\newcommand{\sze}{35mm}
\begin{tabular}{cccc}
\includegraphics[width=\sze]{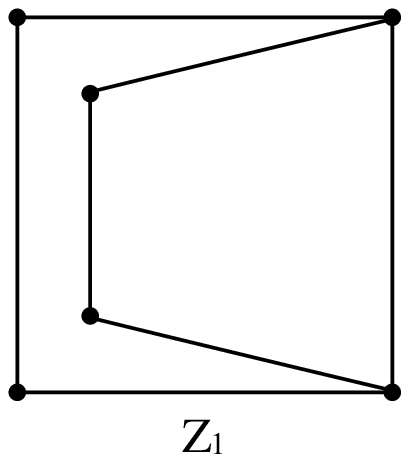}&
\includegraphics[width=\sze]{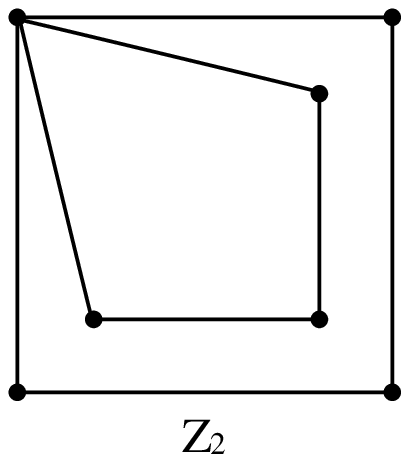}&
\includegraphics[width=\sze]{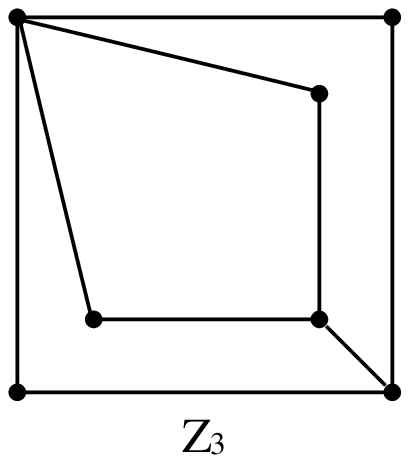}&
\includegraphics[width=\sze]{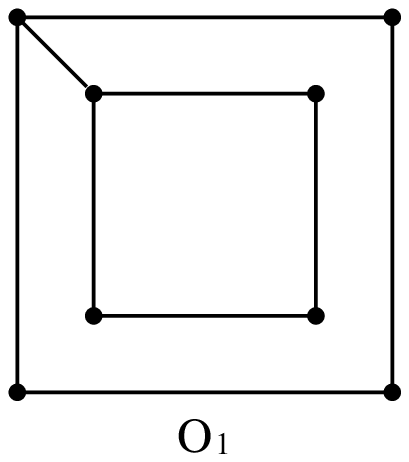}\\
\includegraphics[width=\sze]{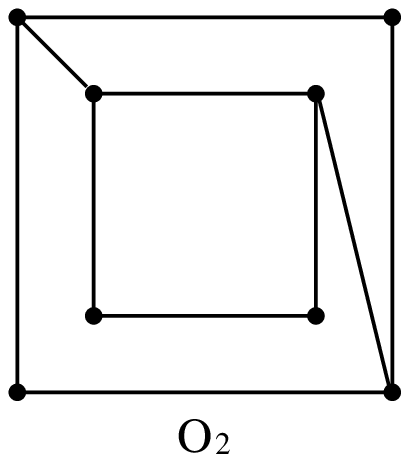}&
\includegraphics[width=\sze]{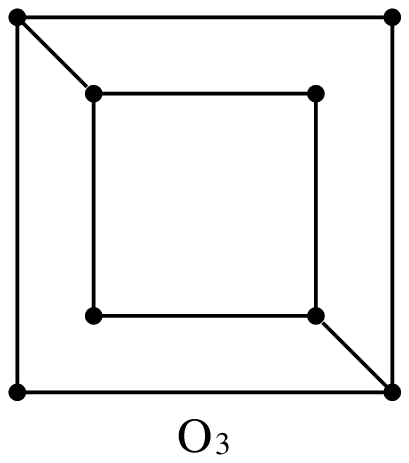}&
\includegraphics[width=\sze]{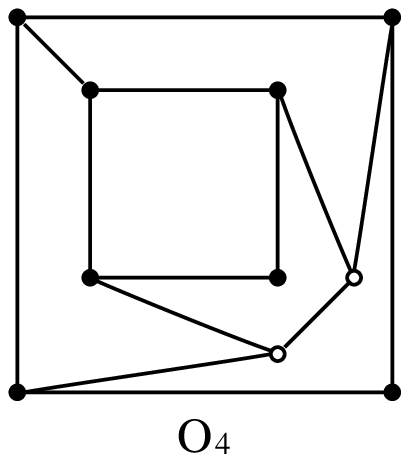}&
\includegraphics[width=\sze]{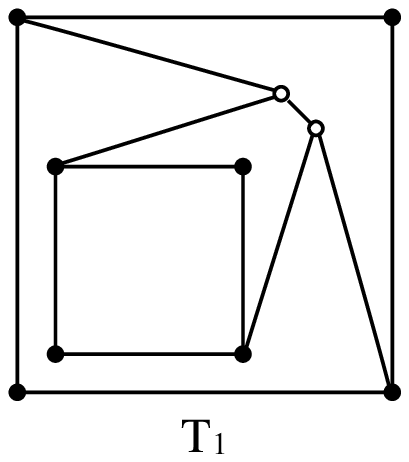}\\
\includegraphics[width=\sze]{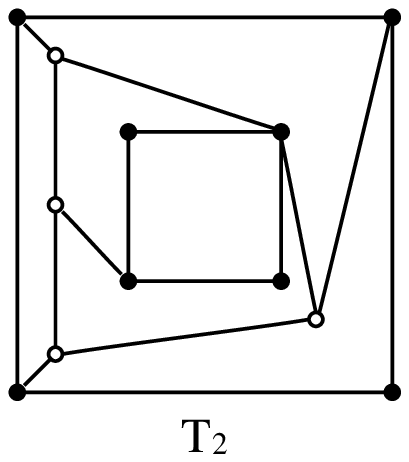}&
\includegraphics[width=\sze]{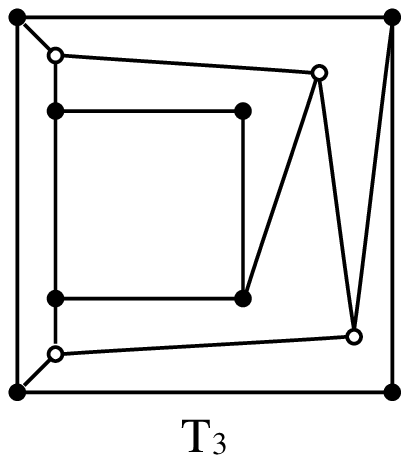}&
\includegraphics[width=\sze]{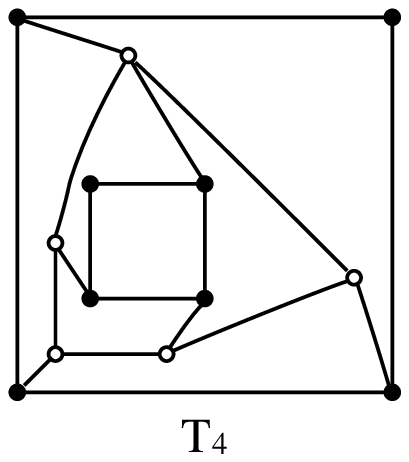}&
\includegraphics[width=\sze]{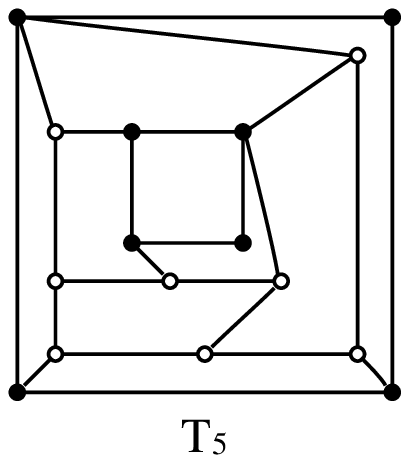}\\
\includegraphics[width=\sze]{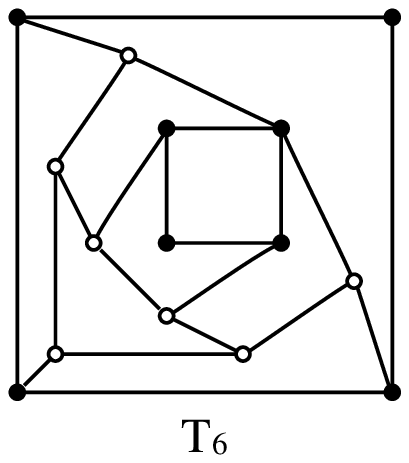}&
\includegraphics[width=\sze]{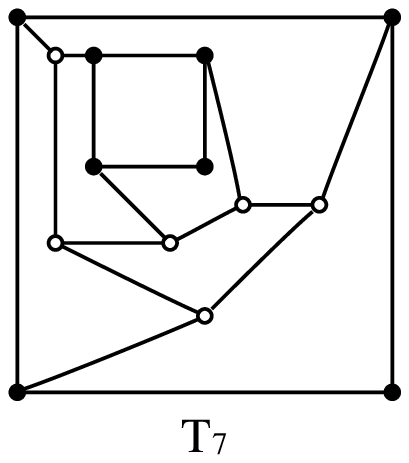}&
\includegraphics[width=\sze]{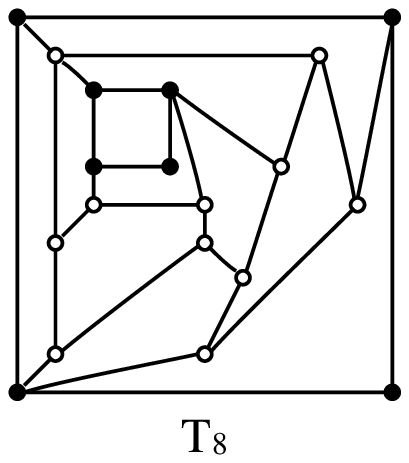}&
\includegraphics[width=\sze]{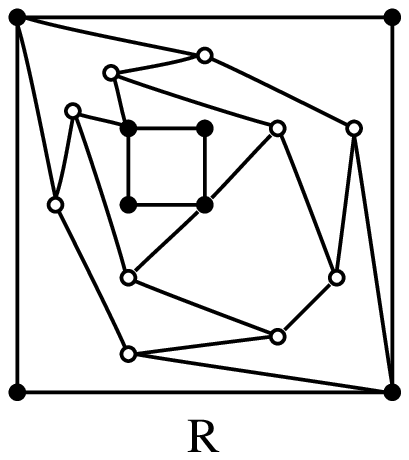}
\end{tabular}
\end{center}
\caption{Critical graphs on the cylinder, bounded by $4$-cycles.}\label{fig-44}
\end{figure}

\begin{figure}
\begin{center}
\newcommand{\sze}{35mm}
\begin{tabular}{ccc}
\includegraphics[width=\sze]{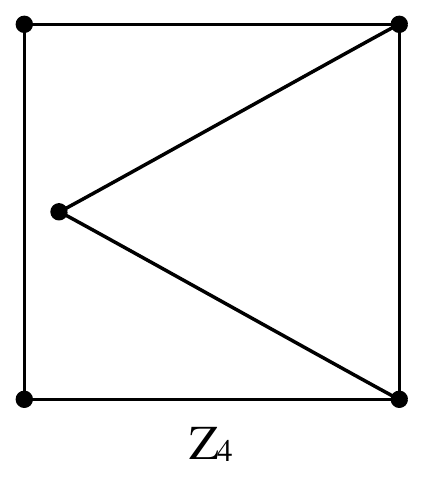}&
\includegraphics[width=\sze]{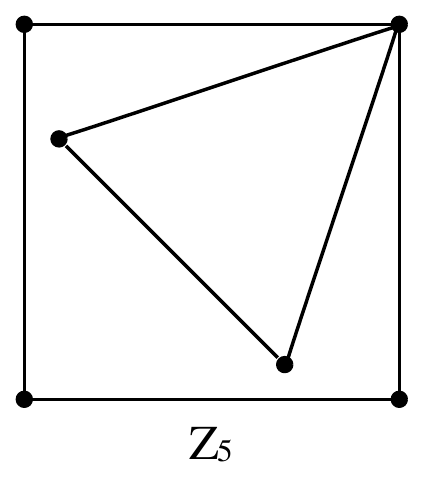}&
\includegraphics[width=\sze]{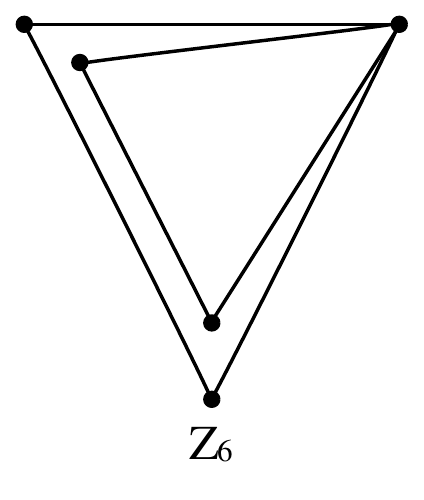}\\
\includegraphics[width=\sze]{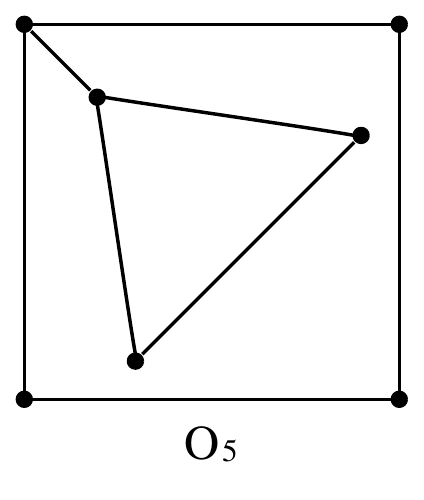}&
\includegraphics[width=\sze]{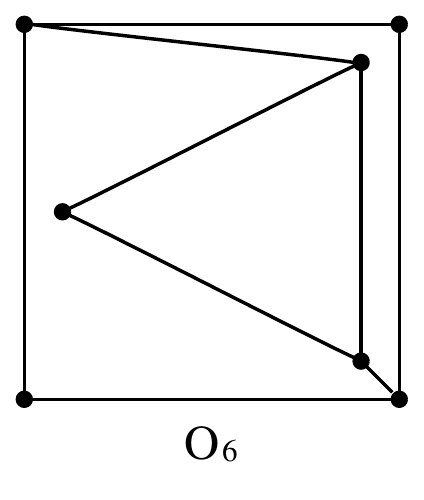}&
\includegraphics[width=\sze]{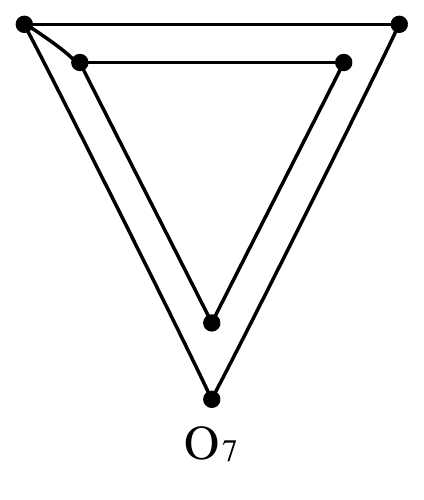}
\end{tabular}
\end{center}
\caption{Critical graphs on the cylinder, with precolored triangle.}\label{fig-34}
\end{figure}

Theorem~\ref{thm-cyl} is proved by the method of reducible configurations:
Considering a graph $G$ on the cylinder in that the distance between the
boundaries $C_1$ and $C_2$ is at least $5$, we find a {\em reducible
configuration}---a subgraph that enables us to transform $G$ to a smaller graph
$H$ that is nontrivial $(C_1\cup C_2)$-critical if and only if $G$ is
nontrivial $(C_1\cup C_2)$-critical.  If $H$ does not contain cycles of length
at most $4$ distinct from $C_1$ and $C_2$, then we argue that $H$ is not one of
the graphs enumerated in Theorem~\ref{thm-cyl}, thus showing that $G$ is not
$(C_1\cup C_2)$-critical.  Otherwise, we cut $H$ along the cycles of length at
most $4$, use Theorem~\ref{thm-cyl} to describe the resulting pieces, and
conclude that every precoloring of $C_1$ and $C_2$ extends to a coloring of
$H$, again implying that $H$ (and thus also $G$) is not $(C_1\cup C_2)$-critical.

This leaves us with the case that the distance between $C_1$ and $C_2$ in $G$
is at most $4$.  In that case, we color the shortest path between $C_1$ and
$C_2$ and cut the graph along it, obtaining a graph of girth $5$ with a
precolored face of length at most $16$.  Such critical graphs with a precolored
face of length at most $11$ were enumerated by Walls~\cite{walls-enum} and
independently by Thomassen~\cite{thom-surf}, who also gives some necessary
conditions for graphs with a precolored face of length $12$.  The exact
enumeration of graphs with a precolored face of length $12$ appears in
Dvo\v{r}\'ak and Kawarabayashi~\cite{dvkaw}.  These results can be summarized
as follows.  

Given a plane graph with the outer face $B$, a {\em chord} of $B$ is an edge in $E(G)\setminus E(B)$
incident with two vertices of $B$.  A {\em $t$-chord} of $B$ is a path $Q=q_0q_1\ldots q_t$ of length $t$ ($t\ge 2$) such
that $q_0\neq q_t$ and $V(Q)\cap V(B)=\{q_0,q_t\}$.  Sometimes, we refer to a chord as a {\em $1$-chord}.
A {\em shortcut} is a $t$-chord of $B$ such that $t$ is smaller than the distance between $u$ and $v$ in $B$.

\begin{theorem}[Dvo\v{r}\'ak and Kawarabayashi~\cite{dvkaw}]\label{thm-crit12}
Let $G$ be a plane graph of girth at least $5$ and $B$ the outer face of $G$ of length
at most $12$.  If $B$ is a cycle, $G$ contains no shortcut of length at most
two, no two vertices of $G$ of degree two are adjacent and $G$ is nontrivial $B$-critical,
then $G$ is isomorphic to one of the graphs in Figure~\ref{fig-12}.
\end{theorem}

\begin{figure}
\begin{center}
\includegraphics{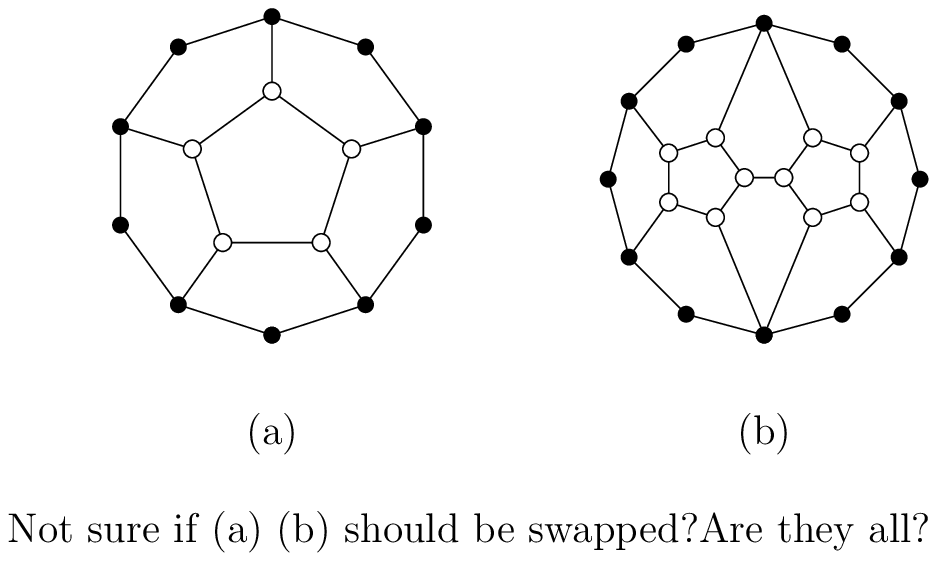}
\end{center}
\caption{Nontrivial critical graphs with precolored face of length at most
$12$.}\label{fig-12}
\end{figure}

\begin{figure}
\begin{center}
\includegraphics[width=9cm]{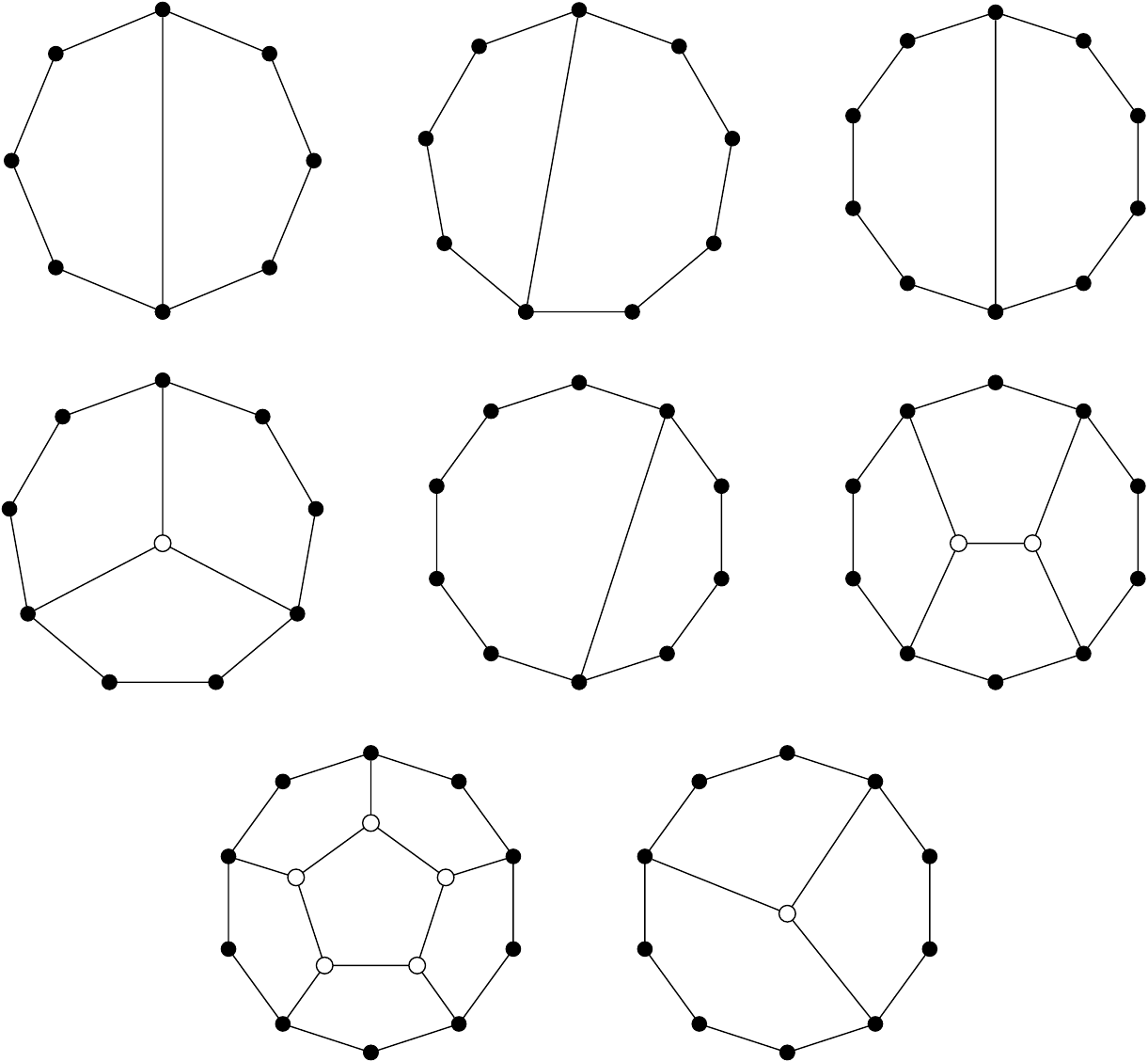}
\end{center}
\caption{Nontrivial critical graphs with precolored face of length at most $10$.}\label{fig-10}
\end{figure}

Let us note that all other critical graphs with the precolored face of length
at most $12$ can be constructed from the graphs in Figure~\ref{fig-12} and a
$5$-cycle
by a sequence of subdividing the edges of the outer face and gluing pairs of
graphs
along paths of length at most two in their outer faces.  For instance, all such
nontrivial critical graphs with $\ell(B)\le 10$ are drawn in Figure~\ref{fig-10}.

The number of critical graphs grows exponentially with the length of the
precolored face, and enumerating all the graphs becomes increasingly difficult.
We implemented an algorithm to generate such graphs based on the results of
Dvo\v{r}\'ak and Kawarabayashi~\cite{dvkaw}, and used the computer to enumerate
the graphs with the outer face of length at most $16$.  There are $108$ such
graphs with the precolored face of length $13$, $427$ for length $14$, $1746$
for length $15$ and $7969$ for length $16$, up to isomorphism (including the
case that $G=B$).  Even excluding the trivial cases that $G$ has a shortcut of
length at most $2$ or contains two adjacent vertices of degree two as in
Theorem~\ref{thm-crit12}, there still remain $8$ graphs with the precolored
face of length $14$ (there are none with a precolored face with length $13$),
$13$ with the length $15$ and $76$ with the length $16$, thus we do not include
their list in this paper.  Here, let us point out only the following claim,
which still makes it possible to enumerate all the graphs easily:

\begin{theorem}\label{thm-crit16}
Let $G$ be a plane graph of girth $5$ and $B$ the outer face of $G$ of length
at most $16$.  If $G$ has no shortcut of length at most $4$ and $G$ is
nontrivial $B$-critical, then $G$ is isomorphic to the graph in
Figure~\ref{fig-12}(a) or to the graphs in Figure~\ref{fig-16}.
\end{theorem}

\begin{figure}
\begin{center}
\includegraphics[width=9cm]{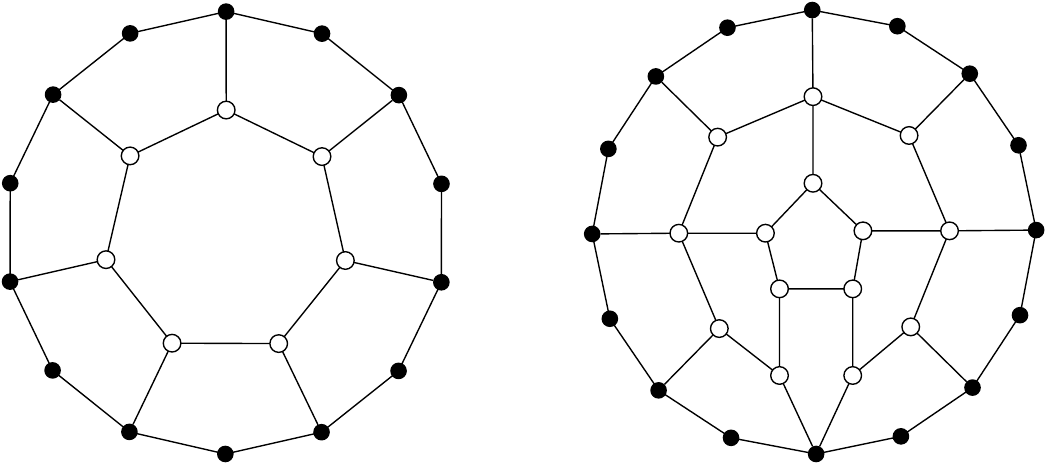}
\end{center}
\caption{Nontrivial critical graphs with precolored faces of length $14$ and $16$, respectively.}\label{fig-16}
\end{figure}

The complete list of the graphs, as well as programs used to generate them can
be found at \oururl. A description of the
programs can be found in Section~\ref{sec-prog}.

In Section~\ref{sec-disk} we give a proof of Theorem~\ref{thm-crit16}. 
Section~\ref{sec-cyl} is devoted to Theorem~\ref{thm-cyl}. Finally,
Section~\ref{sec-main} contains a proof of Theorem~\ref{thm-main}.

\section{Properties of the critical graphs}

Let $G$ be a $T$-critical graph, for some $T\subseteq G$.  For $S\subseteq G$, a
graph $G'\subseteq G$ is an {\em $S$-component} of $G$ if $S\subseteq G'$, $T\cap G'\subseteq S$ and all edges of
$G$ incident with vertices of $V(G')\setminus V(S)$ belong to $G'$. When we use $S$-components, $T$ will always be clear from the context. For example, if $G$ is a plane graph with $T$
contained in the boundary of its outer face and $S$ is a cycle in $G$, then the subgraph of $G$ consisting of the vertices and edges
drawn in the closed disk bounded by $S$ is an $S$-component of $G$.

\begin{lemma}\label{lemma-crs}
Let $G$ be a $T$-critical graph.  If $G'$ is an $S$-component of $G$, for some $S\subseteq G$,
then $G'$ is $S$-critical.
\end{lemma}
\begin{proof}
Since $G$ is $T$-critical, every isolated vertex of $G$ belongs to $T$, and thus every isolated vertex of $G'$ belongs to $S$.
Suppose for a contradiction that $G'$ is not $S$-critical.  Then, there exists an edge $e\in E(G')\setminus E(S)$ such that
every coloring of $S$ that extends to $G'-e$ also extends to $G'$.  Note that $e\not\in E(T)$.
Since $G$ is $T$-critical, there exists a coloring $\psi$ of $T$ that extends to a coloring $\varphi$ of $G-e$, but does not
extend to a coloring of $G$.  However, by the choice of $e$, the restriction of $\varphi$ to $S$ extends to a coloring
$\varphi'$ of $G'$.  Let $\varphi''$ be the coloring that matches $\varphi'$ on $V(G')$ and $\varphi$ on $V(G)\setminus V(G')$.
Observe that $\varphi''$ is a coloring of $G$ extending $\psi$, which is a contradiction.
\end{proof}

Let us remark that Lemma~\ref{lemma-crs} would not hold if we replaced ``critical'' with ``strongly critical'', see Figure~\ref{fig-strongcritical} for an example. This is the main
reason why we (unlike some previous works in the area, e.g. Thomassen~\cite{thom-surf}) consider critical rather than strongly
critical graphs.  However, since every strongly critical graph is also critical, all the characterizations and enumerations that
we provide for critical graphs apply to strongly critical graphs as well.

\begin{figure}
\begin{center}
\includegraphics{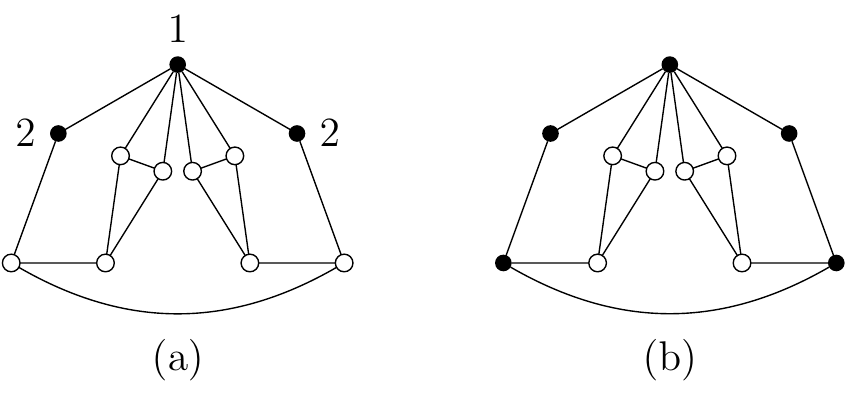}
\end{center}
\caption{(a) A strongly critical graph, with a precolored path on three vertices; (b) not a strongly critical graph with a precolored 5-cycle. }\label{fig-strongcritical}
\end{figure}

Lemma~\ref{lemma-crs} in conjunction with Theorem~\ref{thm-crit12}
describes the subgraphs drawn inside cycles in plane critical graphs.  Let us state a few useful special cases of this
claim explicitly:

\begin{corollary}\label{cor-cycles}
Let $G$ be a plane graph and $T$ a subgraph of $G$ such that $G$ is $T$-critical.  Suppose that every cycle
in $G$ that is not contained in $T$ has length at least $5$.  Let $C$ be a cycle in $G$
and $H$ the subgraph of $G$ drawn in the closed disk bounded by $C$.  Suppose that $H\cap T\subseteq C$.
If $H\neq C$, then $\ell(C)\ge 8$.  If $|V(H)\setminus V(C)|\ge 1$, then $\ell(C)\ge 9$.  Finally,
if $|V(H)\setminus V(C)|\ge 2$, then $\ell(C)\ge 10$.
\end{corollary}

\section{Graphs with one precolored face}\label{sec-disk}

In this section we describe an algorithm for enumerating all $B$-critical graphs
of girth 5 with outer face $B$. First, we describe a previously know recursive description.
Then we show that it can be turned into an algorithm for enumerating $B$-critical graphs.
We implemented the resulting algorithm and we provide its source code.

Dvo\v{r}\'ak and Kawarabayashi~\cite{dvkaw} proved the following claim (in a more general setting of list-coloring):
\begin{theorem}[Dvo\v{r}\'ak and Kawarabayashi~\cite{dvkaw}]\label{thm-numvert}
Let $G$ be a plane graph of girth at least $5$ with the outer face $B$ bounded by a cycle of length at least $10$.
If $G$ is $B$-critical, then $|E(G)|\le 18\ell(B)-160$ and $|V(G)|\le \frac{37\ell(B)-320}{3}$.
\end{theorem}

The obvious algorithm to enumerate the critical graphs by trying all the graphs of the size given by Theorem~\ref{thm-numvert}
is too slow.  However, the proof of Theorem~\ref{thm-numvert} identifies a list of configurations
such that at least one of them must appear in each plane critical graph of girth at least $5$ with the precolored outer face.
For each such configuration, a reduction is provided that makes it possible to obtain $G$ from critical graphs with a shorter precolored outer face.
This leads to a practical algorithm to generate such graphs.  For the algorithm, it turns out to be simpler to use the following
easy corollary of the structural result of Dvo\v{r}\'ak and Kawarabayashi~\cite{dvkaw}.

\begin{theorem}[Dvo\v{r}\'ak and Kawarabayashi~\cite{dvkaw}]\label{thm-struct}
Let $G$ be a plane graph of girth at least $5$ with the outer face $B$ bounded by a cycle.  If $G$ is a $B$-critical graph,
then $G$ is $2$-connected and at least one of the following holds:

\begin{itemize}
\item[(a)] $G$ has a shortcut of length at most $4$, or
\item[(b)] $G$ contains two adjacent vertices of degree two (belonging to $B$), or
\item[(c)] there exists a path $P=v_0v_1v_2v_3v_4\subseteq B$ and a $4$-chord $Q=v_0w_1w_2w_3v_4$ of $B$ such that $v_2w_2\in E(G)$, or
\item[(d)] there exists a $4$-chord $Q=w_0w_1w_2w_3w_4$ of $B$ and $5$-faces
$C_1$ and $C_2$ such that a cycle $C\subseteq B\cup Q$ distinct from $B$ bounds a face of $G$,
$|V(C_1\cap B)|=|V(C_2\cap B)|=3$, $C_1\cap C=w_0w_1$ and $C_2\cap C=w_3w_4$.
\end{itemize}
See Figure~\ref{fig-thm-struct}.
\end{theorem}

\begin{figure}
\begin{center}
\includegraphics[width=12cm]{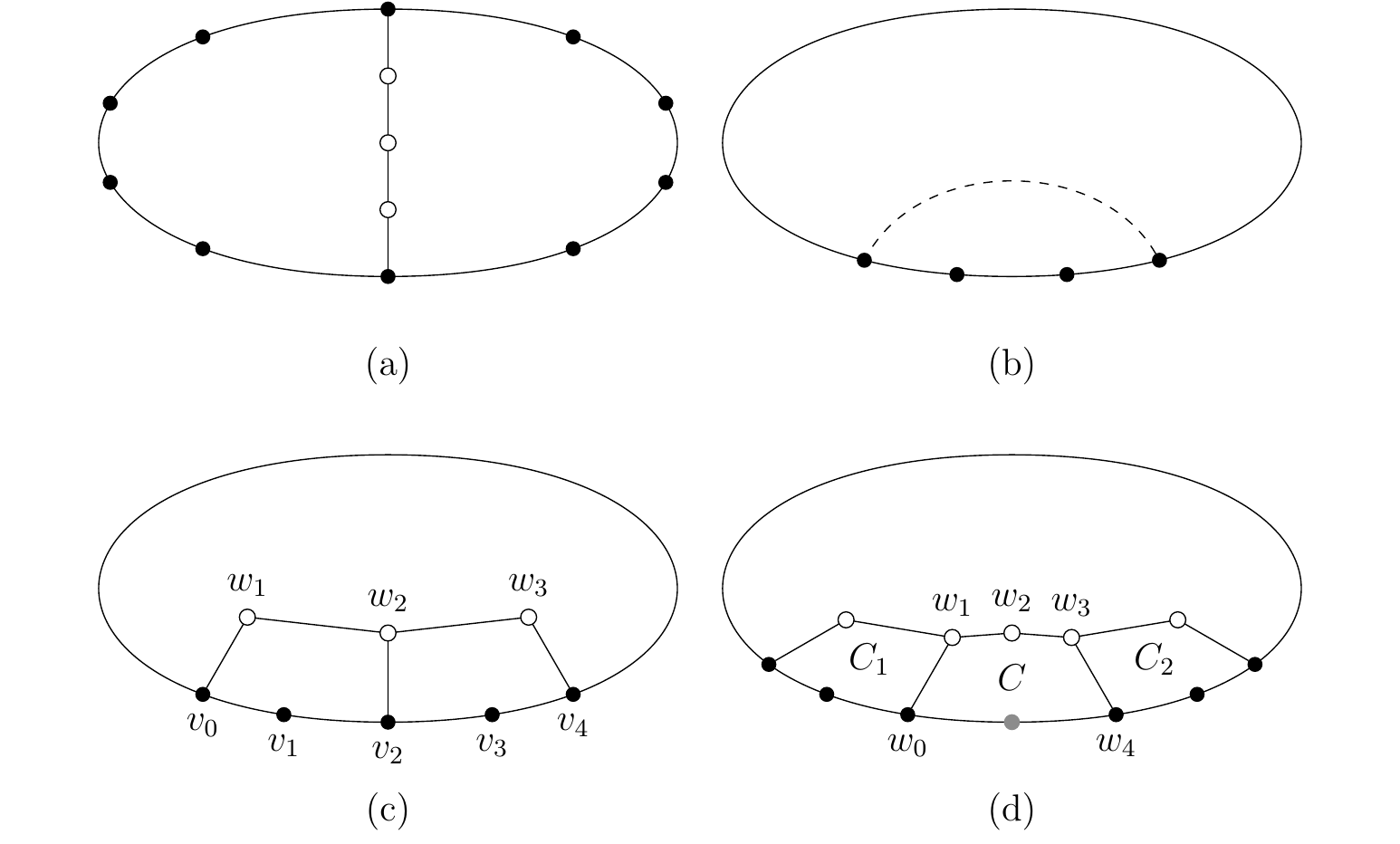}
\end{center}
\caption{Cases of Theorem~\ref{thm-struct}.}\label{fig-thm-struct}
\end{figure}

While these configurations are not sufficient to prove Theorem~\ref{thm-numvert}, each of the more complicated configurations
considered in the proof of Theorem~\ref{thm-numvert} contains one of the configurations of Theorem~\ref{thm-struct} as a subgraph.
For the reduction in case (d), we also need the following result, which is shown for strongly critical graphs in Thomassen~\cite{thom-surf},
and explicitly for critical graphs in Dvo\v{r}\'ak and Kawarabayashi~\cite{dvkaw}.  For a plane graph $G$ with the outer face $B$,
let $m(G)$ be the length of the longest face of $G$ distinct from $B$.

\begin{theorem}[\cite{dvkaw,thom-surf}]\label{thm-fsize}
Let $G$ be a plane graph of girth at least $5$ with the outer face $B$ bounded by a cycle.  If $G$ is a nontrivial $B$-critical graph, then $m(G)\le \ell(B)-3$.
\end{theorem}

We now define several graph generating operations roughly corresponding to the cases (a)--(d) of Theorem~\ref{thm-struct}.
Let $G_1$ and $G_2$ be plane graphs with outer faces $B_1$ and $B_2$, respectively.
\begin{itemize}
\item[(a)]
Let $P_i=v_0^iv_1^i\ldots v_t^i$ be paths
such that $P_i\subseteq B_i$ for $i\in \{1,2\}$ and some $t>0$.  We let $U(G_1, P_1, G_2, P_2)$ be the graph obtained from the disjoint union of
$G_1$ and $G_2$ by identifying $v_j^1$ with $v_j^2$ for $j=0,1,\ldots, t$ and suppressing the arising parallel edges.
\item[(b)] For an edge $e\in E(G_1)$, let $S(G_1, e)$ be the graph obtained from $G_1$ by subdividing the edge $e$ by one vertex.
\item[(c)] For a path $P=v_0w_1w_2w_3v_4\subseteq B_1$, let $J(G_1, P)$ be the graph obtained from $G_1$ by adding new vertices $v_1$, $v_2$ and $v_3$
and edges $v_0v_1$, $v_1v_2$, $v_2w_2$, $v_2v_3$ and $v_3v_4$.\\
\item[(d)]  Let $P=u_0u_1u_2u_3u_4\subseteq B_1$ be a path, let $y_1=u_1$, $y_2$, \ldots, $y_k=u_3$ be the
vertices adjacent to $u_2$ in the cyclic order according to their drawing around $u_2$
and let $f_i$ be the edge $u_2y_i$ for $1 \leq i \leq k$.
For $2\le i\le k$ and $0\le j\le 1$, let $X(G_1, P, f_i, j)$ be the plane graph obtained from $G_1$ by
splitting $u_2$ to two vertices $u_2'$ and $u_2''$ so that $u_2'$ is adjacent to $y_1$, $y_2$, \ldots, $y_{i-1}$
and $u_2''$ is adjacent to $y_{i}$, \ldots, $y_k$,
adding vertices $x_1$, $x_2$, \ldots, $x_{4+j}$ and edges $u_0x_1$, $x_1x_2$,
$x_2x_3$, \ldots, $x_{4+j}u_4$, $u'_2x_2$ and $u''_2x_{3+j}$.
See Figure~\ref{fig-X}. 
\end{itemize}
Note that $m(X(G_1, P, f_i, j))\le m(G_1)+j+3$ and the length $\ell$ of the
outer face of $X(G_1, P, f_i, j)$ is equal to $\ell(B_1)+j+1$.
By Theorem~\ref{thm-fsize}, if $G_1$ is a nontrivial $B_1$-critical graph, then $m(X(G_1, P, f_i, j))\le \ell(B_1)+j < \ell$.

\begin{figure}
\begin{center}
\includegraphics{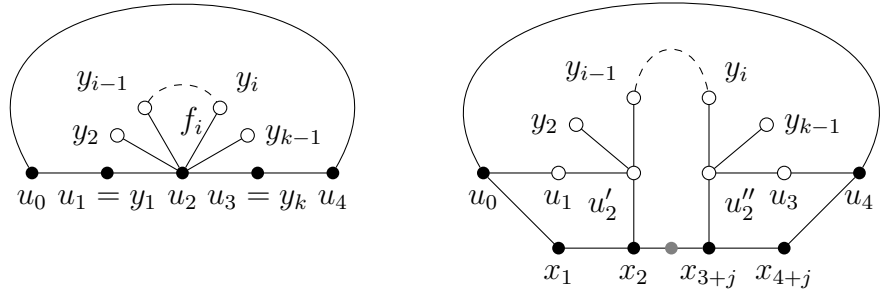}
\end{center}
\caption{$X(G_1, P, f_i, j)$.}\label{fig-X}
\end{figure}

For $i\ge 5$, let $\KK_i$ be the set of all (up to isomorphism) plane graphs $G$ of girth $5$ with the outer face $B$ bounded by a cycle,
such that $G$ is $B$-critical and $\ell(B)=i$.
By Theorem~\ref{thm-numvert} and Theorem~\ref{thm-crit12}, $\KK_i$ is finite.
For a $2$-connected plane graph $G$ with the outer face $B$, let $\KK(G)$ be the set of all graphs $H\supseteq G$ with the outer face $B$
such that for every face $C$ of $G$, the subgraph of $H$ drawn in the closed disk bounded by $C$ belongs to $\KK_{\ell(C)}$.  In other words,
$\KK(G)$ consists of graphs obtained from $G$ by pasting a critical graph to each face distinct from the outer one.
Let us remark that we do not exclude the case that the pasted graph is trivial, i.e., a face of $G$ may also be a face of some graphs in $\KK(G)$.
Note that $\KK(G)$ is finite, and can be constructed by a straightforward algorithm if the sets $\KK_i$ are provided for $5\le i\le m(G)$.

For some $\ell$, suppose that $\SSS$ is a finite set of plane graphs $G$ of girth at least $5$ with the outer face $B(G)$
bounded by  a cycle of length $\ell$.
Let $\KK(\SSS)=\bigcup_{G\in \SSS} \KK(G)$.
Let $\TT(\SSS)\subseteq \SSS$ be the set consisting of all $B(G)$-critical graphs $G$ in $\SSS$.
Let $\JJ'(\SSS)=\{J(G, P) : G\in \SSS, P\subseteq B(G), \ell(P)=4\}$.  Note that
the outer face of each graph $G$ in $\JJ'(\SSS)$ has length $\ell$.
Let $\SSS_0$, $\SSS_1$, $\SSS_2$, \ldots, be
the sequence of sets of graphs such that $\SSS_0=\TT(\SSS)$ and
$\SSS_{i+1}=\TT(\JJ'(\SSS_i))$ for $i\ge 0$.  Let $\JJ(\SSS)=\bigcup_{i\ge 0}
\SSS_i$.
Since $\JJ(\SSS)\subseteq \KK_{\ell}$, $\JJ(\SSS)$ is finite, and there exists $k$ such that $\SSS_i=\emptyset$ for each $i\ge k$.
Therefore, the set $\JJ(\SSS)$ can be constructed algorithmically by finding the sets $\SSS_0$, $\SSS_1$, \ldots, as long as they are non-empty.

Let
$$ \KK_i^{(a)} = 
\left\{\begin{array}{lc}
 & 1\le t\le 4, i+2t=i_1+i_2, 5\le i_1\le i_2\le i-1,\\
U(G_1,P_1,G_2,P_2) : & G_1\in \KK_{i_1}, G_2\in\KK_{i_2},\\
 & \hbox{$P_j$ is a path in the outer face of $G_j$ } \\ & \hbox{ with
$\ell(P_j)=t$, for $j\in\{1,2\}$}\end{array} \right\}.$$
Let $$\KK_i^{(b)} = \{S(G,e) : G\in \KK_{i-1}, \hbox{$e$ is an edge of the outer face of $G$}\}.$$
Let $$\KK_i^{(d)} = 
\left\{\begin{array}{lc}
 & 0\le j\le 1, G\in\KK_{i-j-1},\\
X(G,P,e,j) : & \hbox{$P=u_0u_1u_2u_3u_4$ is a path in the outer face of $G$},\\
&\hbox{$e\neq u_1u_2$ is incident with $u_2$}\end{array} \right\}.$$

Let $\KK''_i$ consist of all graphs in $\KK_i^{(a)}\cup \KK_i^{(b)}\cup \KK(\KK_i^{(d)})$ that have girth at least $5$.
Let $\KK'_i=\JJ(\KK''_i)$.  Note that the set $\KK'_i$ is finite and can be constructed algorithmically,
given the sets $\KK_j$ for $5\le j<i$.

\begin{theorem}\label{thm-alg}
The following holds for $i>5$: $\KK_i=\KK'_i$.
\end{theorem}
\begin{proof}
Note that every graph $G\in\KK'_i$ is a plane $B$-critical graph of girth at least $5$, where $B$ is the outer face of $G$ and $\ell(B)=i$
and thus $\KK'_i\subseteq \KK_i$. Therefore, we only need to show that $\KK_i\subseteq \KK'_i$.  By Theorem~\ref{thm-numvert}, there exists
a constant $N$ such that $|V(H)|\le N$ for every $H\in \KK_i$.

Consider a graph $G\in \KK_i$ with the outer face $B$.  If there exists a path $P=v_0v_1v_2v_3v_4\subseteq B$ and a $4$-chord $Q = v_0w_1w_2w_3v_4$
of $B$ such that $v_2w_2\in E(G)$, then let $B'$ be the cycle obtained from $B$ by replacing $P$ by $Q$.  By Corollary~\ref{cor-cycles},
$v_0v_1v_2w_2w_1$ and $v_4v_3v_2w_2w_3$ are $5$-faces, and by Lemma~\ref{lemma-crs}, $G'=G-\{v_1,v_2,v_3\}$ is $B'$-critical.
It follows that $G=J(G',Q)$.  We conclude that there exists a sequence (of length at most $N/3$) of plane graphs $G=G_0$, $G_1$, \ldots, $G_k$
of girth at least $5$ with the outer faces $B=B_0, B_1, \ldots, B_k$, respectively, and paths $P_1$, \ldots, $P_k$ such that
$G_j$ is $B_j$-critical for $0\le j\le k$, $P_j\subseteq B_j$ and $G_{j-1}=J(G_j,P_j)$ for $1\le j\le k$, and $G_k$
does not contain the configuration (c) of Theorem~\ref{thm-struct}.  In other words, as long as $G_j$ contains the configuration (c), we keep reducing the graph and when there is no configuration (c), we stop. We claim that $G_k\in \KK''_i$, implying that $G\in \KK'_i$.

Since $G_k$ is a plane $B_k$-critical graph of girth at least $5$, Theorem~\ref{thm-struct} implies that it contains one of the configurations
(a), (b) or (d).  If it contains the configuration (a) (a shortcut $Q$ of length at most $4$), then let $C_1$ and $C_2$ be
the cycles in $B_k\cup Q$ distinct from $B_k$, and let $H_j$ be the subgraph of $G_k$ drawn in the closed disk bounded by $C_j$ for $j\in\{1,2\}$.
By Lemma~\ref{lemma-crs}, $H_j$ is $C_j$-critical, which implies that $H_j\in \KK_{\ell(C_j)}$.  Since $Q$ is a shortcut,
$\ell(C_j)<i$.  We conclude that $G_k\in \KK_i^{(a)}\subseteq\KK''_i$.  Therefore, we may assume that $G_k$ has no shortcut of length at most $4$.

Suppose that $G_k$ contains two adjacent vertices $u$ and $v$ of degree two (the configuration (b)).  Since $G_k$ is $B_k$-critical,
both $u$ and $v$ belong to $V(B_k)$.  The edge $uv$ is not contained in any cycle of length $5$, since otherwise there would exist a shortcut of length at most two.
Let $H$ be the graph obtained from $G_k$ by identifying the vertices $u$ and $v$ to a new vertex $w$, with the outer face $B'$,
and note that $H$ has girth at least $5$.    Furthermore, $H$ is $B'$-critical, since each precoloring of $B$ corresponds to a precoloring of
$B'$ matching it on $V(B)\setminus\{u,v\}$.  Observe that $G=S(H, e)$, where $e$ is an edge incident with $w$, and thus $G_k\in \KK_i^{(b)}\subseteq\KK''_i$.
Thus, we may assume that no two vertices of degree two are adjacent in $G_k$.  In particular, $G_k$ is a nontrivial $B_k$-critical graph,
and there exists a precoloring $\varphi$ of $B_k$ that does not extend to a coloring of $G_k$.

Finally, consider the case that $G_k$ contains the configuration (d). That is, there exists a $4$-chord $Q=w_0w_1w_2w_3w_4$ of $B_k$ and $5$-faces
$C_1$ and $C_2$ such that a cycle $C\subseteq B_k\cup Q$ distinct from $B_k$ bounds a face of $G$,
$|V(C_1\cap B_k)|=|V(C_2\cap B_k)|=3$, $C_1\cap C=w_0w_1$ and $C_1\cap C=w_3w_4$.  Since $G_k$ does not contain adjacent vertices of degree two, we have
$\ell(C)\le 6$.  Let $j=\ell(C)-5$.
Let $H$ be the graph obtained from $G_k$ by removing $w_0,w_4$ and their neighbors in $V(B)$ and by identifying
$w_1$ with $w_3$ to a new vertex $w$, and let $B'$ be the outer face of $H$.
Since $w_2$ has degree at least three\footnote{Note that vertices of degree one or two in $C$-critical graph $G$ must be in $C$. For every vertex $v$ of degree at most two, every coloring of $G-v$ extends to a coloring of $G$, which contradicts $C$-criticality if $v \not\in V(C)$.}, $w_1w_2w_3$ is not a subpath of the boundary of a face $F\neq C$ in $G_k$;
hence, Corollary~\ref{cor-cycles} implies that the girth of $H$ is at least $5$.
Indeed, if there is a cycle $Z$ in $H$ of length at most $4$, it must contain $w$. We can replace $w$ by $w_1,w_2,w_3$ and obtain a cycle $Z'$ of length at most $6$ in $G$. Since $w$ has degree at least three, the cycle is not a face which contradicts Corollary~\ref{cor-cycles}.

Observe that the precoloring $\varphi$ of $B_k$ (which does not extend to $G_k$) extends to a coloring $\psi$ of $(B_k\cup C_1\cup C_2\cup C)-\{w_2\}$ such that
$w_1$ and $w_3$ have the same color.  Since $\varphi$ does not extend to a coloring of $G_k$, we conclude that the precoloring of $B'$ given by $\psi$
does not extend to a coloring of $H$.  Therefore, $H$ has a nontrivial $B'$-critical subgraph $H'$.  Let $P\subseteq B'$ be the path of length
$4$ such that $w$ is the middle vertex of $P$.  Lemma~\ref{lemma-crs} implies that $G_k\in\KK(X(H',P,e, j))$ for some edge $e\in E(H')$ incident with $w$.
Thus, $G_k\in \KK(\KK_i^{(d)})\subseteq \KK''_i$.

It follows that $G_k\in \KK''_i$, and thus $G\in \KK'_i$.  Since the choice of $G$ was arbitrary, this implies that $\KK_i\subseteq\KK'_i$
and hence $\KK'_i=\KK_i$.
\end{proof}

The sets $\KK_5$, \ldots, $\KK_{12}$ are given by Theorem~\ref{thm-crit12}.  Theorem~\ref{thm-alg} gives an algorithm that we used
to construct the sets $\KK_{13}$, \ldots, $\KK_{16}$ (we also used the program to generate the sets $\KK_8,\ldots, \KK_{12}$,
to give it a better testing).  Theorem~\ref{thm-crit16} follows by the
inspection of the graphs in $\KK_5\cup\ldots\cup\KK_{16}$
(which was also computer assisted).

\section{Graphs on the cylinder}\label{sec-cyl}

Let us now turn our attention to graphs drawn in the cylinder.
Our goal is to describe plane graphs that are critical for two precolored $(\le\!4)$-faces, such that all other cycles have length at least $5$.
Such graphs can be thought of as embedded in the cylinder so that the two short faces are on the top and bottom of the cylinder.

First, in Lemma~\ref{lemma-cyl-le4} we use $\KK_5\cup\ldots\cup\KK_{16}$ to generate critical graphs on cylinder
with two precolored $(\le\!4)$-cycles at distance at most $4$ from each other and all other cycles of length at least $5$. This part is computer assisted.
Next, we glue pairs of these graphs together to obtain critical graphs with one non-precolored separating $(\le\!4)$-cycle, see Lemma~\ref{lemma-join2}.
We discuss the outcomes of gluing three such graphs in Lemma~\ref{lemma-join3}.
Finally, in Lemma~\ref{lemma-join} we give a general description of the critical graphs created by from those of Lemma~\ref{lemma-cyl-le4} by gluing.
We complete the description by Lemma~\ref{lemma-cyl-ge5}, which shows that
a plane graph with two precolored $(\le\!4)$-faces at distance at least $5$ and all other cycles of length at least $5$ is never critical.

\begin{lemma}\label{lemma-cyl-le4}
Let $G$ be a connected graph embedded on the cylinder with distinct boundaries $C_1$ and $C_2$ such that $\ell(C_1), \ell(C_2)\le 4$ and every cycle in $G$ distinct
from $C_1$ and $C_2$ has length at least $5$.  If $G$ is $(C_1\cup C_2)$-critical and the distance between $C_1$ and $C_2$ is at most $4$,
then $G$ is isomorphic to one of the graphs drawn in Figures~\ref{fig-44} or \ref{fig-34}.
\end{lemma}
\begin{proof}
Let us first consider the case that $\ell(C_1)=\ell(C_2)=4$.  
If $C_1$ and $C_2$ share an edge, then $C_1 \cup C_2$ contains a $6$-cycle which 
bounds a face by Corollary~\ref{cor-cycles}. Thus $G$ is graph $Z_1$ in Figure~\ref{fig-44}.
Therefore, we may assume $C_1$ and $C_2$ share no edges.

Let $P$ be a shortest path between $C_1$ and $C_2$.  By Lemma~\ref{lemma-crs},
$G$ is $(C_1\cup C_2\cup P)$-critical.  Let $H$ be the graph obtained from $G$ by cutting along the path $P$, splitting the vertices of $P$ into
two and duplicating the edges of $P$, and let $B$ be the resulting face.
See Figure~\ref{fig-split} for the splitting of $T_7$.
Observe that $H$ is $B$-critical and $B$ is a cycle.
Furthermore, $\ell(B)=\ell(C_1)+\ell(C_2)+2\ell(P)\le 16$, thus $H$ is one of the graphs in $\KK_5\cup\ldots\cup \KK_{16}$, which we enumerated using a computer in the previous section.

\begin{figure}
\begin{center}
\includegraphics{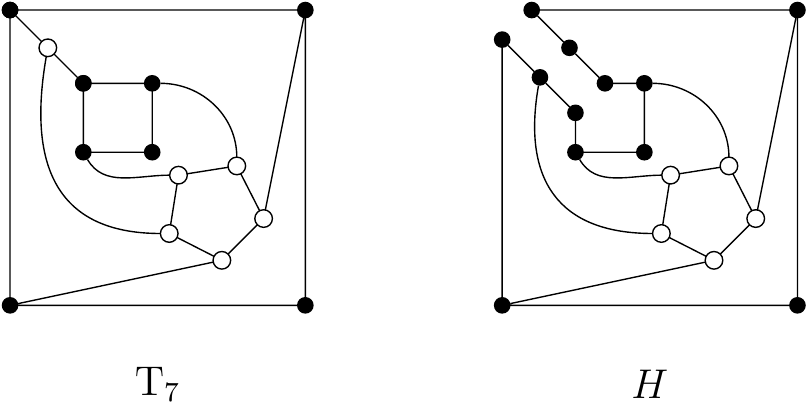}
\end{center}
\caption{Splitting of T$_7$ into $B$-critical graph $H$ where $B$ is the outer face of $H$.}\label{fig-split}
\end{figure}

Note that $G$ can be obtained from $H$ by identifying appropriate paths in the face $B$.  Using a computer, we checked all possible choices
of $H$ (as described in Section~\ref{sec-disk}) and the paths, and checked whether the resulting graph satisfies the assumptions of this lemma.  This way, we proved that
$G$ must be one of the graphs depicted in Figure~\ref{fig-44}. 

If $\ell(C_1)=3$ or $\ell(C_2)=3$, then we subdivide edges of $C_1$ or $C_2$, so that the new precolored cycles $C'_1$ and $C'_2$ have length exactly $4$.
This does not change the distance between the cycles, and the resulting graph $G'$ is $(C'_1\cup C'_2)$-critical.  Therefore, $G'$ is one of the
graphs depicted in Figure~\ref{fig-44}.  Inspection of these graphs shows that $G$ is one of the graphs in Figure~\ref{fig-34}.
\end{proof}

Let us remark that there are no graphs satisfying the assumptions of Lemma~\ref{lemma-cyl-le4} where the distance between $C_1$
and $C_2$ is exactly $4$, and only one such graph $R$ where the distance is exactly three.
This is the last computer-based result of this paper (although we used the computer to check the correctness of the case analyses
on several other places in the paper, all of them were also performed independently by hand).  In particular, if someone proved Lemma~\ref{lemma-cyl-le4} without
the use of a computer, this would give computer-free proofs of Theorems~\ref{thm-dist} and \ref{thm-main}.

For the graphs depicted in Figures~\ref{fig-44} and \ref{fig-34}, let $B$ denote the outer face and $T$ the other face of length at most $4$.
For a plane graph $G$ with faces $F_1$ and $F_2$ and a precoloring $\psi$ of $F_1$, let $c(G, F_1, \psi, F_2)$ be the number of colorings $\varphi$ of $F_2$
such that $\psi\cup \varphi$ does not extend to a coloring of $G$ (in case that $F_1$ and $F_2$ intersect,
 this includes the colorings that
assign the common vertices colors different from those given by $\psi$; say in $Z_2$ if $v$ is a common vertex of $F_1$ and $F_2$, we count colorings where $\psi(v) \neq \varphi(v)$).  Let $c(G,F_1, F_2)$ be the maximum of $c(G, F_1, \psi, F_2)$ over all precolorings $\psi$
of $F_1$.  By a straightforward inspection of the listed graphs, we find that
the values of $c(G, B, T)$ and $c(G, T, B)$ for the graphs in
Figures~\ref{fig-44} and \ref{fig-34} are as follows:
\begin{center}
\begin{tabular}{ccc|ccc}
$G$ & $c(G, B, T)$ & $c(G, T, B)$ & $G$ & $c(G, B, T)$ & $c(G, T, B)$ \\
\hline
$Z_1$ & $15$ & $15$ & $O_6$ & $4$ &  $11$ \\
$Z_2$ & $12$ & $12$ & $O_7$ &  $2$ &  $2$ \\
$Z_3$ & $16$ & $16$ & $T_1$ &  $8$ &  $8$ \\
$Z_4$ & $5$ &  $15$ & $T_2$ &  $1$ &  $2$ \\
$Z_5$ & $4$ &  $12$ & $T_3$ &  $4$ &  $4$ \\
$Z_6$ &  $4$ &  $4$ & $T_4$ &  $3$ &  $3$ \\
$O_1$ &  $6$ &  $6$ & $T_5$ &  $2$ &  $2$ \\
$O_2$ & $12$ & $11$ & $T_6$ &  $2$ &  $2$ \\
$O_3$ & $11$ & $11$ & $T_7$ &  $2$ &  $2$ \\
$O_4$ & $12$ & $12$ & $T_8$ &  $2$ &  $2$ \\
$O_5$ &  $2$ &  $6$ & $R$   &  $4$ &  $4$ \\
\end{tabular}
\end{center}

As an example, let us compute $c(T_1,B,T)$.
Let $B=b_1b_2b_3b_4$ and $T=t_1t_2t_3t_4$
Denote the two not precolored vertices by $z_1$ and $z_3$,
where $z_i$ is adjacent to $t_i$ and $b_i$ for $i \in \{1,3\}$. Observe that a
precoloring $\psi$ of $B$ and $\varphi$ of $T$ do not extend to a coloring of $T_1$
if and only if $\{\psi(b_1),\varphi(t_1)\} = \{\psi(b_3),\varphi(t_3)\}$ and
$\psi(b_1) \neq \varphi(t_1)$.
Hence we need to consider only two cases
for $\psi$: either  $\psi(t_1) = \psi(t_3)$ or $\psi(t_1) \neq \psi(t_3)$.
Assume first that $\psi(t_1) \neq \psi(t_3)$. Then $\varphi(b_1) = \psi(t_3)$
and $\varphi(b_3) = \psi(t_1)$. This leaves only one possibility for $\varphi$
of $b_2$ and $b_4$. Hence there is one coloring $\varphi$ such that
$\phi \cup \varphi$ does not extend to $T_1$.
For the second case assume that  $\psi(t_1) = \psi(t_3)$. There
are two possibilities for assigning $\varphi(b_1) = \varphi(b_3)$ such
that $\psi(t_1) \neq \varphi(b_1)$. Each of these two possibilities can be
extended to a coloring of $T$ if four ways. Hence the total number of precolorings $\varphi$ such that $\psi \cup \varphi$ does not extend is $8$.

Based on these numbers, we characterize critical graphs obtained by pasting two such cylinders together.
A cycle $C$ in a plane graph {\em separates} subgraphs $G_1$ and $G_2$ if neither of the closed regions of the plane bounded by $C$ contains both $G_1$ and $G_2$.

\begin{figure}
\begin{center}
\newcommand{\sze}{35mm}
\begin{tabular}{cccc}
\includegraphics[width=\sze]{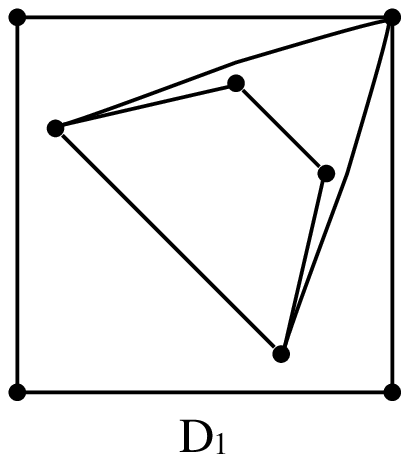}&
\includegraphics[width=\sze]{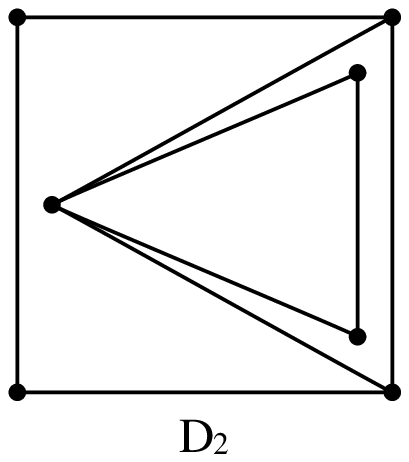}&
\includegraphics[width=\sze]{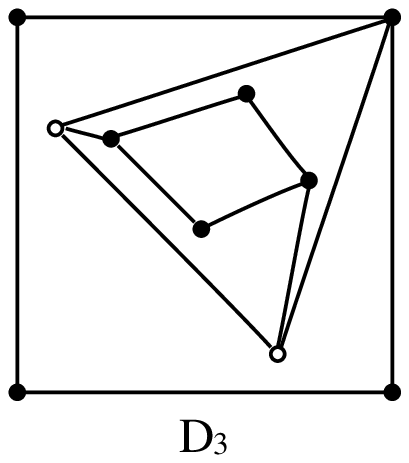}&
\includegraphics[width=\sze]{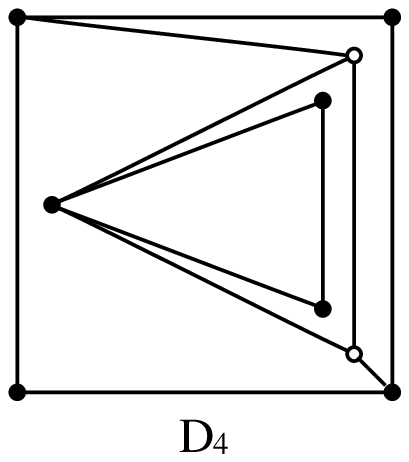}\\
\includegraphics[width=\sze]{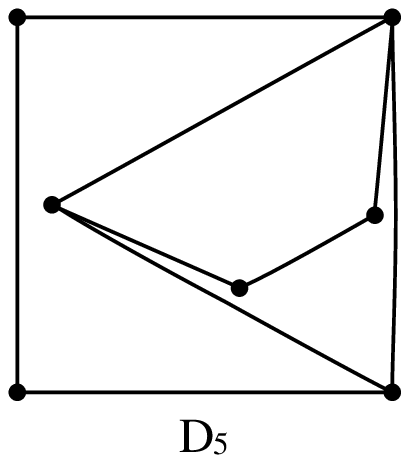}&
\includegraphics[width=\sze]{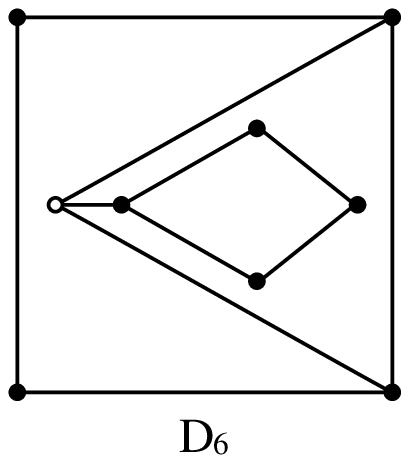}&
\includegraphics[width=\sze]{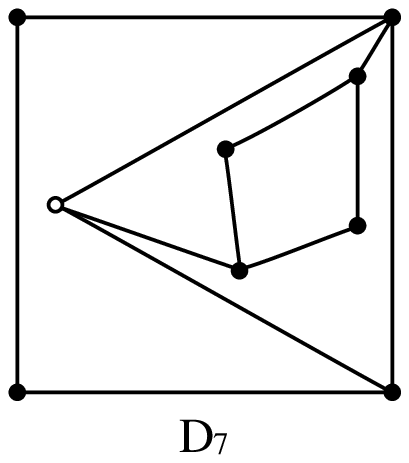}&
\includegraphics[width=\sze]{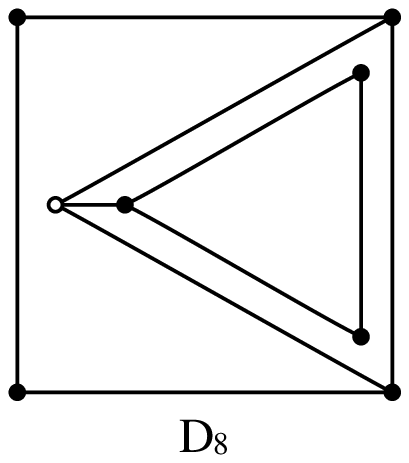}\\
\includegraphics[width=\sze]{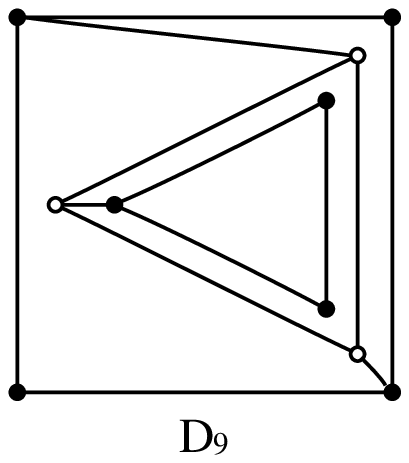}&
\includegraphics[width=\sze]{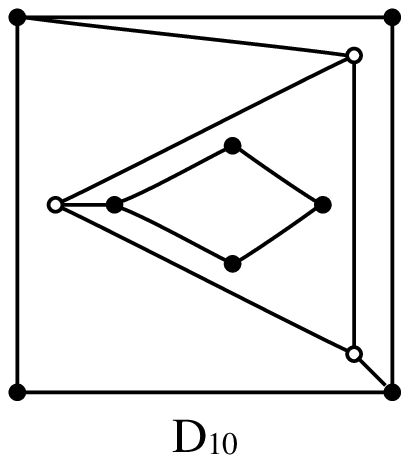}&
\includegraphics[width=\sze]{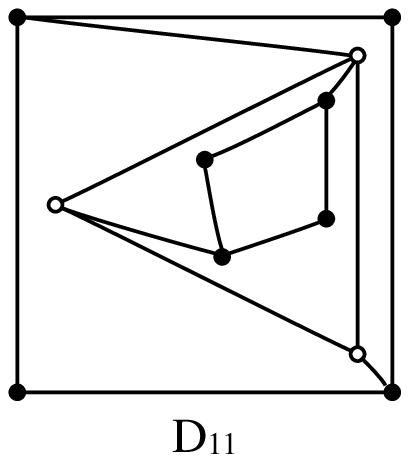}&
\end{tabular}
\end{center}
\caption{Critical graphs on the cylinder, one separating triangle.}\label{fig-join2a}
\end{figure}

\begin{figure}
\begin{center}
\newcommand{\sze}{35mm}
\begin{tabular}{cccc}
\includegraphics[width=\sze]{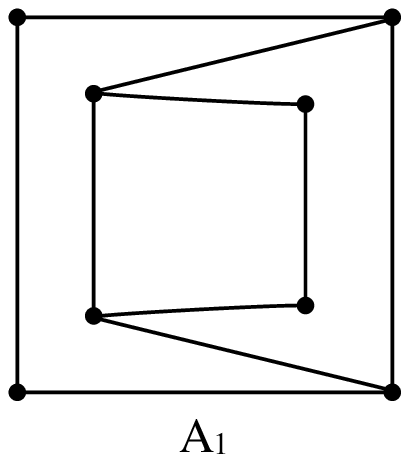}&
\includegraphics[width=\sze]{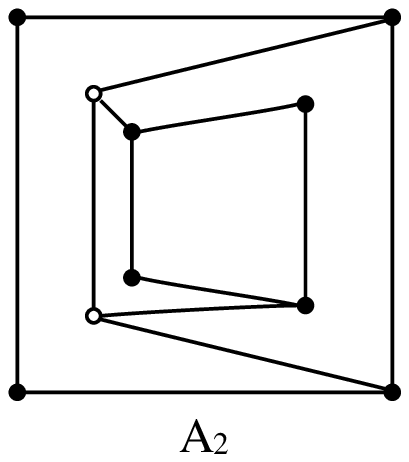}&
\includegraphics[width=\sze]{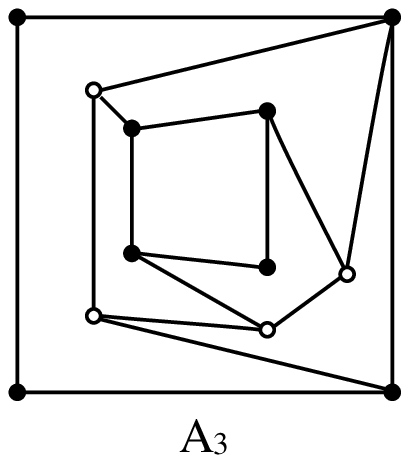}&
\includegraphics[width=\sze]{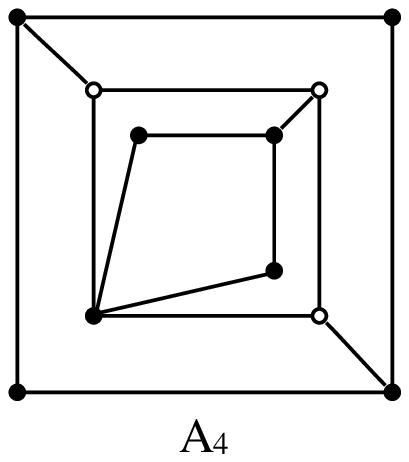}\\
\includegraphics[width=\sze]{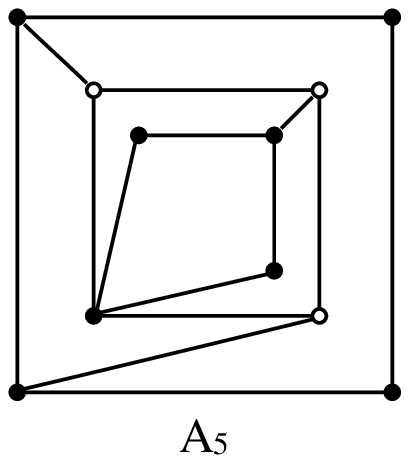}&
\includegraphics[width=\sze]{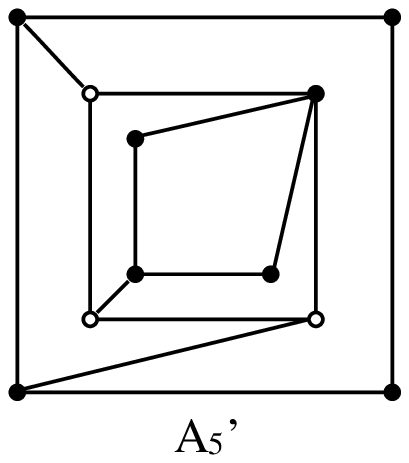}&
\includegraphics[width=\sze]{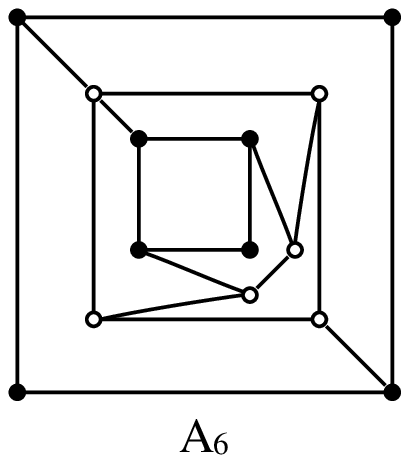}&
\includegraphics[width=\sze]{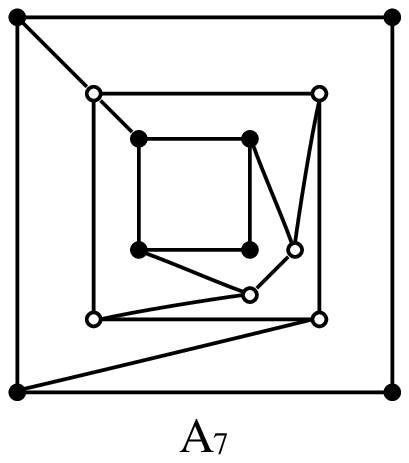}\\
\includegraphics[width=\sze]{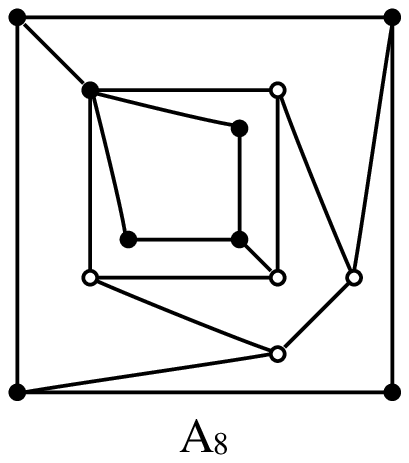}&
\includegraphics[width=\sze]{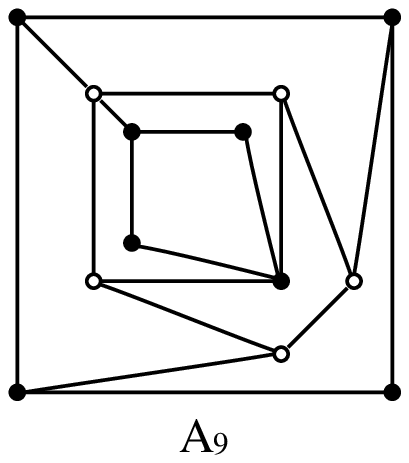}&
\includegraphics[width=\sze]{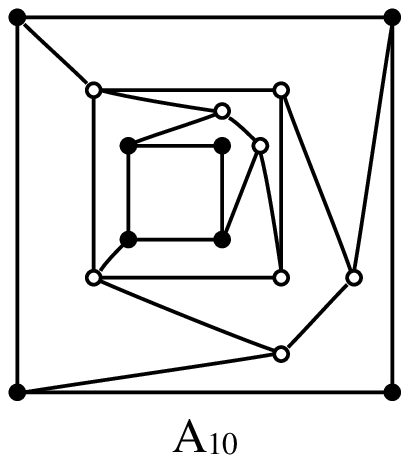}&
\includegraphics[width=\sze]{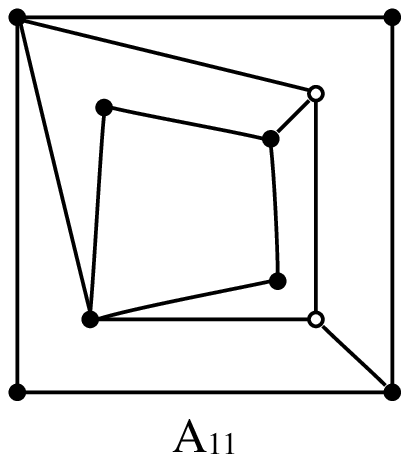}\\
\includegraphics[width=\sze]{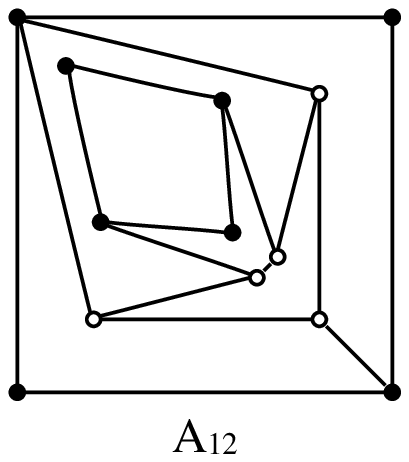}&
\includegraphics[width=\sze]{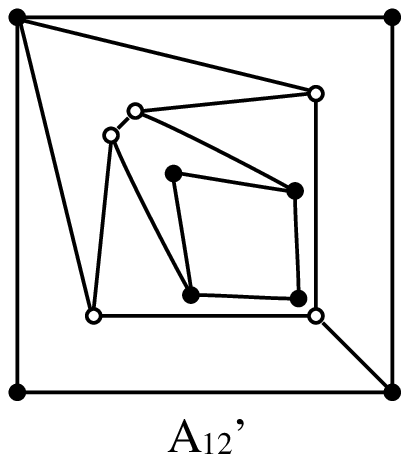}&
\includegraphics[width=\sze]{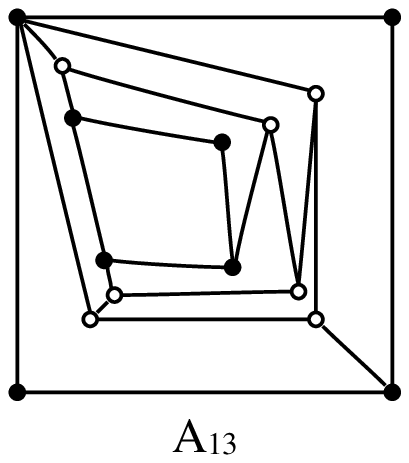}&
\end{tabular}
\end{center}
\caption{Critical graphs on the cylinder, one separating $4$-cycle. Note that A$_5$ and A$_5$' is the same graph but different embedding. We use primes to distinguish different embeddings.}\label{fig-join2b}
\end{figure}

\begin{figure}
\begin{center}
\newcommand{\sze}{35mm}
\begin{tabular}{cccc}
\includegraphics[width=\sze]{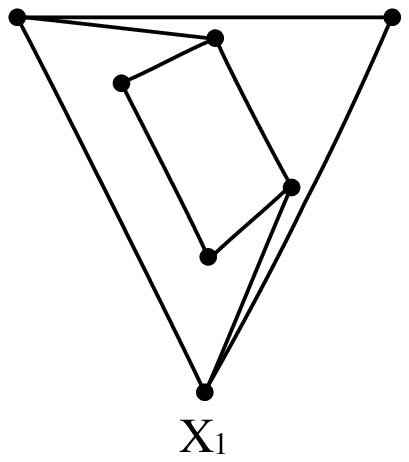}&
\includegraphics[width=\sze]{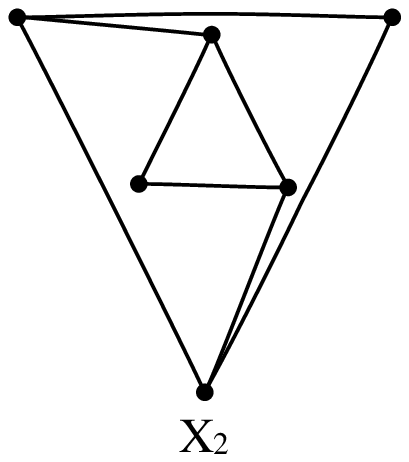}&
\includegraphics[width=\sze]{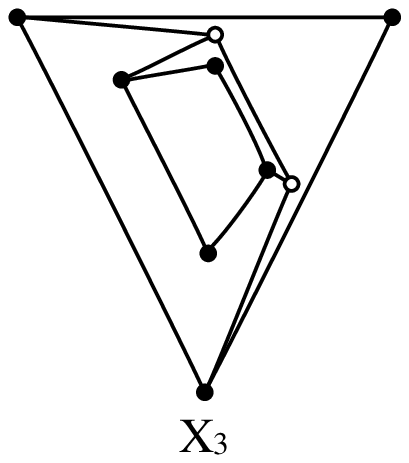}&
\includegraphics[width=\sze]{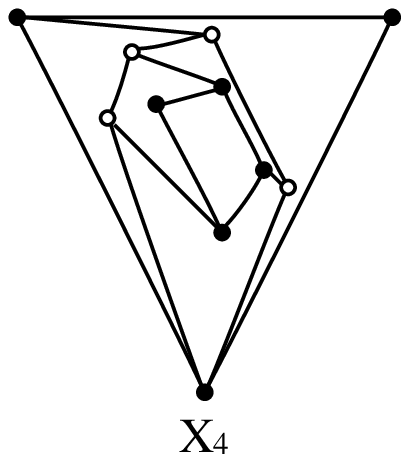}\\
\includegraphics[width=\sze]{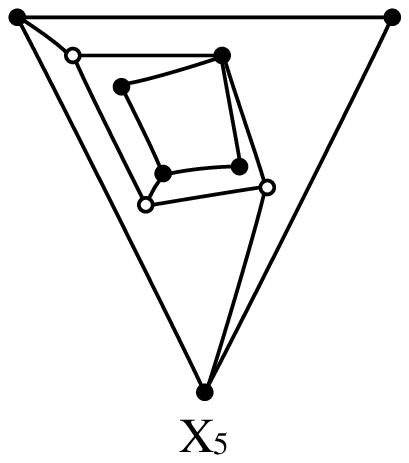}&
\includegraphics[width=\sze]{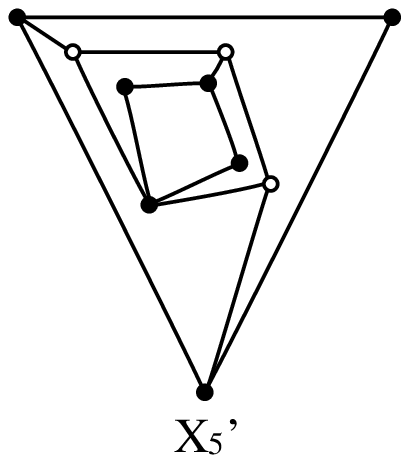}&
\includegraphics[width=\sze]{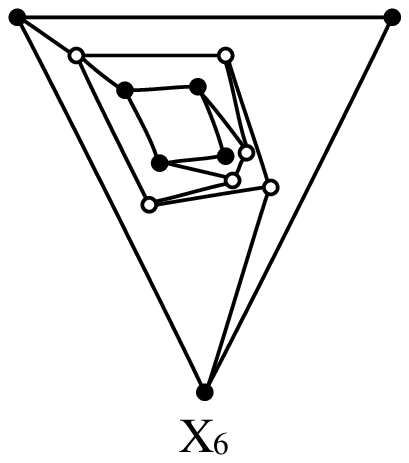}&
\end{tabular}
\end{center}
\caption{Critical graphs on the cylinder, one separating $4$-cycle, precolored triangle.}\label{fig-join2c}
\end{figure}

Let $G$ be a plane graph and $C_1$ and $C_2$ be two cycles
in $G$ such that $C_2$ is drawn in the closed interior of $C_1$. 
A graph $H$ \emph{drawn between} $C_1$ and $C_2$ is a graph
obtained from  $G$ by removing open exterior of $C_1$ and open interior of $C_2$. In particular, $C_i$ bounds a face in $H$ for $i \in \{1,2\}$.

\begin{lemma}\label{lemma-join2}
Let $G$ be a connected graph embedded on the cylinder with distinct boundaries $C_1$ and $C_2$ such that $\ell(C_1), \ell(C_2)\le 4$.  Let $C\subseteq G$ be a cycle of length
at most $4$ separating $C_1$ from $C_2$.  Assume that every cycle in $G$ distinct from $C$, $C_1$ and $C_2$ has length at least $5$, and that the distance
between $C_1$ and $C$, as well as the distance between $C$ and $C_2$, is at most $4$.
If $G$ is $(C_1\cup C_2)$-critical, then $G$ is isomorphic to one of the graphs drawn in Figures~\ref{fig-join2a}, \ref{fig-join2b} and \ref{fig-join2c}.
\end{lemma}
\begin{proof}
Let $G_i$ be the subgraph of $G$ drawn between $C_i$ and $C$, for $i\in\{1,2\}$. 
By Lemma~\ref{lemma-crs} and Lemma~\ref{lemma-cyl-le4}, $G_i$ is equal to one of the
graphs drawn in Figures~\ref{fig-44} and \ref{fig-34}.  Let $\varphi$ be a precoloring of $C_1\cup C_2$ that does not extend to a coloring of $G$.
Suppose first that $\ell(C)=3$, i.e., the situation depicted in Figure~\ref{fig-join2a}.
There exist $6$ colorings of $C$ by three colors.  Observe that for every coloring $\psi$ of $C$ there exists $i\in\{1,2\}$
such that the precoloring of $C_i$ and $C$ given by $\varphi\cup\psi$ does not extend to $G_i$, and thus $c(G_1, C_1, C)+c(G_2,C_2,C)\ge 6$.
By symmetry, we may assume that $c(G_1, C_1, C)\geq 3$ and hence $G_1$ is one of $Z_4$, $Z_5$, $Z_6$ and $O_6$.  If $G_1\in\{Z_5,Z_6\}$, then $C$ contains two vertices that have degree two
in $G_1$, and since $G$ is critical, they must either belong to $C_2$ or have degree at least three, implying that $G_2\in\{Z_4,O_6\}$.  Hence, $G$ is
one of the graphs $D_1$, $D_2$, $D_3$ or $D_4$.  If $G_1=Z_4$, then we conclude similarly that $G$ is one of $D_1$, $D_2$, $D_5$, $D_6$, $D_7$ or $D_8$,
and if $G_1=O_6$, then $G$ is one of $D_3$, $D_4$, $D_7$, $D_9$, $D_{10}$ or $D_{11}$.

Let us now consider the case that $\ell(C_1)=\ell(C_2)=\ell(C)=4$, as depicted in Figure~\ref{fig-join2b}.  Since $C$ has $18$ colorings, we have
$c(G_1, C_1, C)+c(G_2,C_2,C)\ge 18$.  We may assume that $c(G_1,C_1,C) \geq 9$, i.e., $G_1\in \{Z_1, Z_2, Z_3, O_2, O_3, O_4\}$.  If $G_1\in \{Z_1, Z_2\}$ or $G_1=O_2$ with $C_1$ being
the outer face of $O_2$, then $C$ contains two adjacent vertices whose degree is two in $G_1$.  These vertices must either belong to $C_2$, or their
degree must be at least three in $G_2$.  We conclude (also taking into account that $c(G_2,C_2,C) \geq 3$) that $G_2\in \{Z_1, O_2, O_4, T_3, T_4\}$, and (excluding the combinations that do not
result in a critical graph), $G$ is one of the graphs $A_1$, $A_2$ or $A_3$.  From now on, assume that $G_1,G_2\not\in\{Z_1,Z_2\}$. If $G_1=O_3$ or $G_1=O_2$ with $C$ being
the outer face of $O_2$, then we similarly conclude that $G_2\in\{Z_3, O_2, O_3, O_4, T_1\}$ and $G$ is one of the graphs $A_4$, $A_5$, $A_5'$, $A_6$ or $A_7$.
We may assume that $G_1,G_2\not\in\{O_2,O_3\}$.  If $G_1=O_4$, then $G_2\in \{Z_3, O_1, O_4, T_1\}$, and $G$ is $A_8$, $A_9$ or $A_{10}$.
Finally, if $G_1=Z_3$, then $G$ is $A_{11}$, $A_{12}$, $A'_{12}$ or $A_{13}$.

If $\ell(C)=4$ and $\ell(C_1)=3$ or $\ell(C_2)=3$, then $G$ is one of the graphs in Figure~\ref{fig-join2c}, obtained from those in
Figure~\ref{fig-join2b} by suppressing vertices of degree two.
\end{proof}

Again, let us summarize the values of $c(G, B, T)$ and $c(G, T, B)$ for these
graphs:

\begin{center}
\begin{tabular}{ccc|ccc}
$G$ & $c(G, B, T)$ & $c(G, T, B)$ & $G$ & $c(G, B, T)$ & $c(G, T, B)$ \\
\hline
$D_1$    & $12$ & $12$ & $A_5$, $A'_5$       &  $4$ &  $6$ \\
$D_2$    &  $4$ & $12$ & $A_6$               &  $2$ &  $1$ \\
$D_3$    &  $8$ & $12$ & $A_7$               &  $2$ &  $3$ \\
$D_4$    &  $4$ &  $8$ & $A_8$               & $10$ & $10$ \\
$D_5$    & $15$ & $15$ & $A_9$               &  $9$ &  $8$ \\
$D_6$    &  $6$ &  $6$ & $A_{10}$            &  $2$ &  $2$ \\
$D_7$    & $11$ & $12$ & $A_{11}$            & $14$ & $14$ \\
$D_8$    &  $2$ &  $6$ & $A_{12}$, $A'_{12}$ &  $4$ &  $4$ \\
$D_9$    &  $2$ &  $4$ & $A_{13}$            &  $4$ &  $4$ \\
$D_{10}$ &  $6$ &  $4$ & $X_1$               &  $9$ &  $3$ \\
$D_{11}$ &  $6$ &  $6$ & $X_2$               &  $3$ &  $3$ \\
$A_1$    &  $9$ &  $9$ & $X_3$               &  $1$ &  $1$ \\
$A_2$    &  $1$ &  $3$ & $X_4$               &  $2$ &  $1$ \\
$A_3$    &  $2$ &  $3$ & $X_5$               &  $4$ &  $2$ \\
$A_4$    &  $4$ &  $2$ & $X_6$               &  $2$ &  $1$ \\
\end{tabular}
\end{center}

We proceed by listing the graphs with two separating cycles of length at most $4$.

\begin{figure}
\begin{center}
\newcommand{\sze}{35mm}
\begin{tabular}{ccc}
\includegraphics[width=\sze]{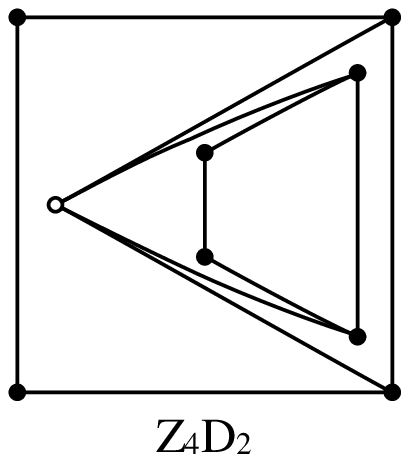}&
\includegraphics[width=\sze]{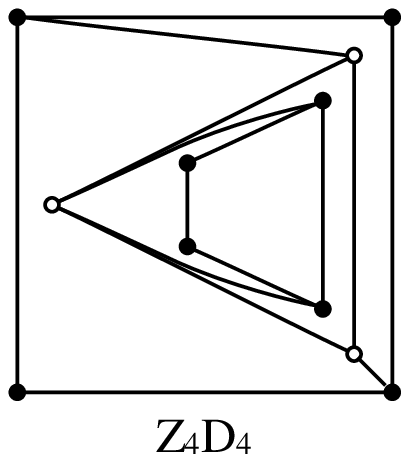}&
\includegraphics[width=\sze]{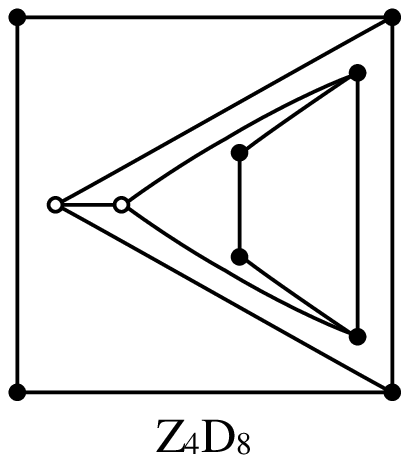}\\
\includegraphics[width=\sze]{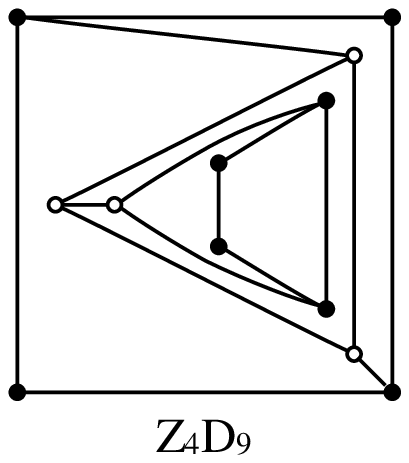}&
\includegraphics[width=\sze]{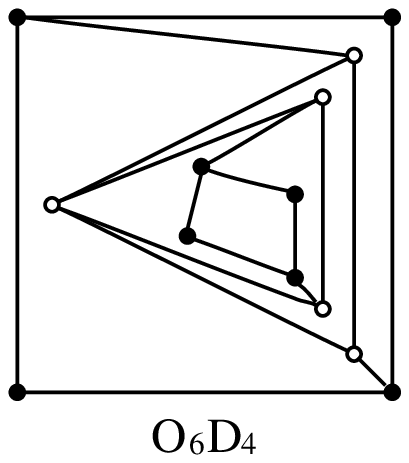}&
\includegraphics[width=\sze]{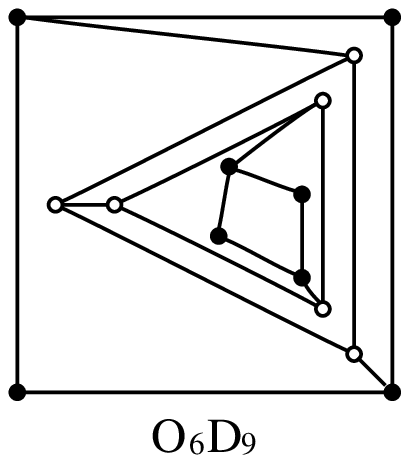}
\end{tabular}
\end{center}
\caption{Critical graphs on the cylinder, two separating triangles.}\label{fig-join3a}
\end{figure}

\begin{figure}
\begin{center}
\newcommand{\sze}{35mm}
\begin{tabular}{cccc}
\includegraphics[width=\sze]{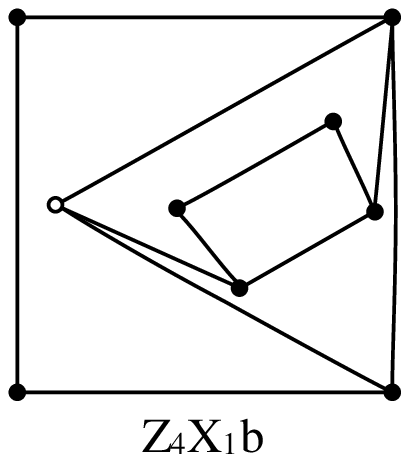}&
\includegraphics[width=\sze]{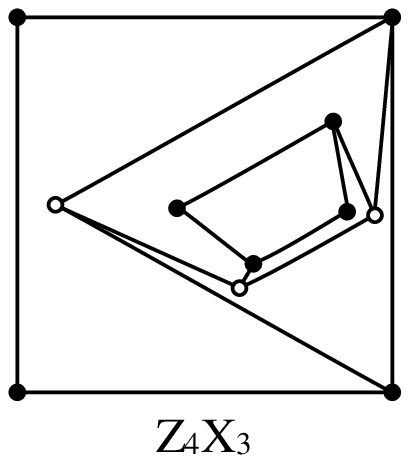}&
\includegraphics[width=\sze]{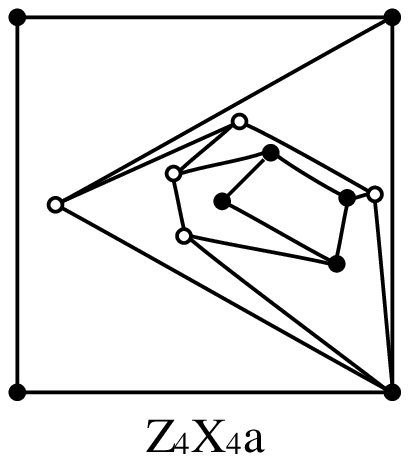}&
\includegraphics[width=\sze]{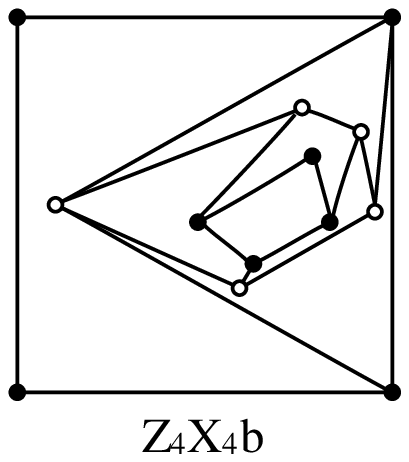}\\
\includegraphics[width=\sze]{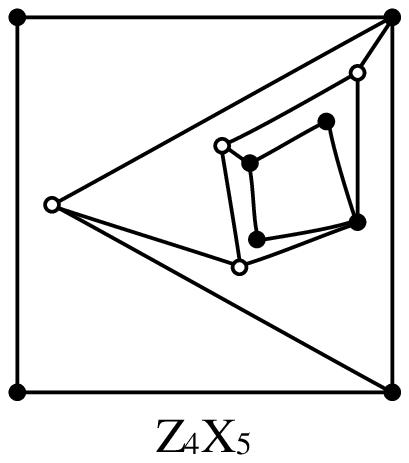}&
\includegraphics[width=\sze]{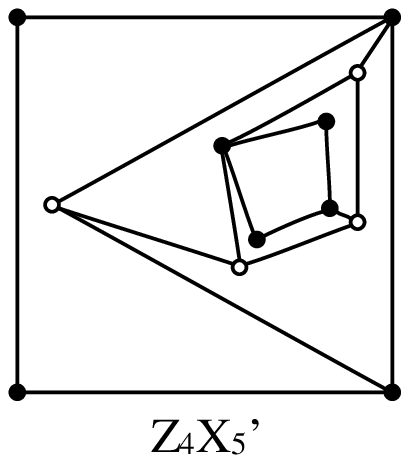}&
\includegraphics[width=\sze]{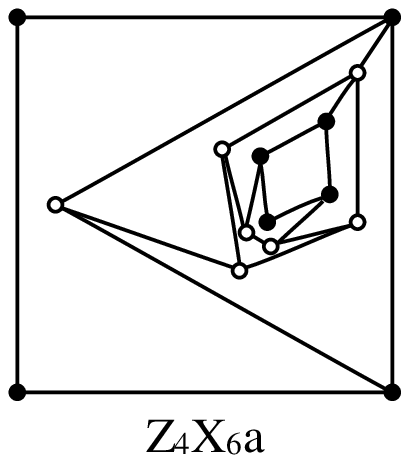}&
\includegraphics[width=\sze]{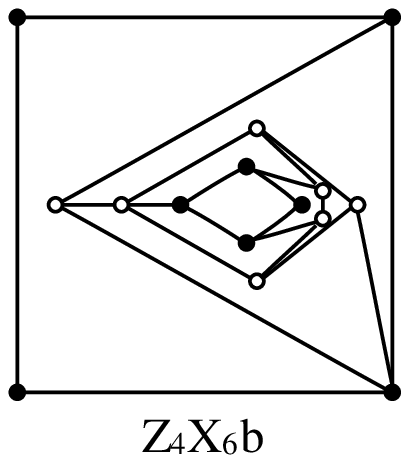}\\
\includegraphics[width=\sze]{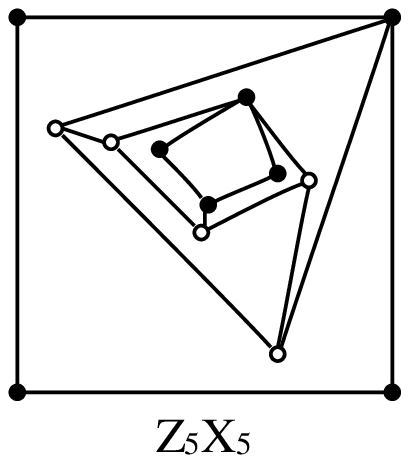}&
\includegraphics[width=\sze]{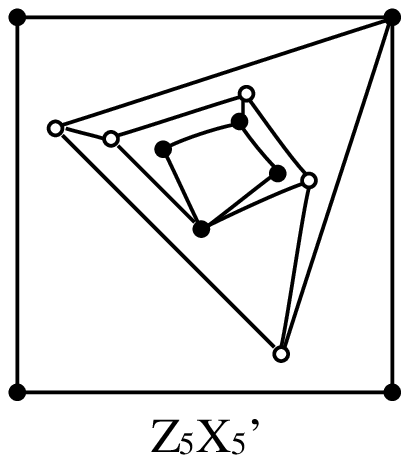}&
\includegraphics[width=\sze]{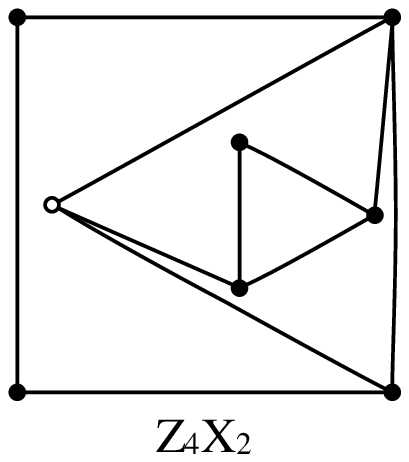}&
\includegraphics[width=\sze]{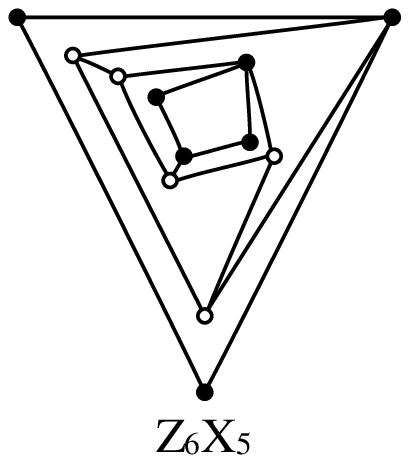}\\
\includegraphics[width=\sze]{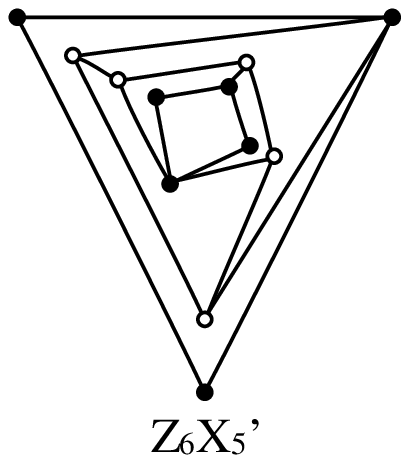}&
\end{tabular}
\end{center}
\caption{Critical graphs on the cylinder, separating triangle and a $4$-cycle.}\label{fig-join3b}
\end{figure}

\begin{figure}
\begin{center}
\newcommand{\sze}{35mm}
\begin{tabular}{cccc}
\includegraphics[width=\sze]{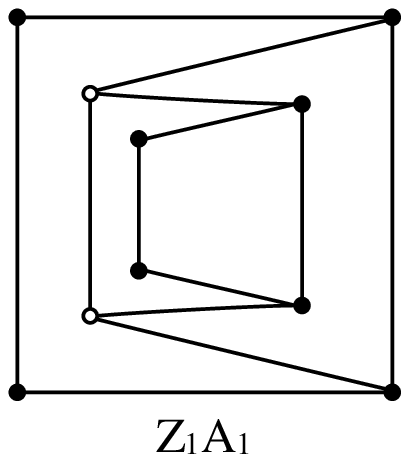}&
\includegraphics[width=\sze]{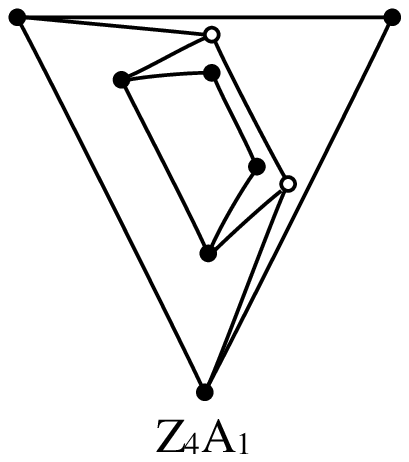}&
\includegraphics[width=\sze]{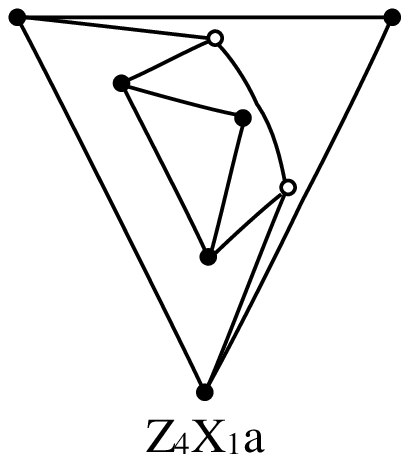}&
\includegraphics[width=\sze]{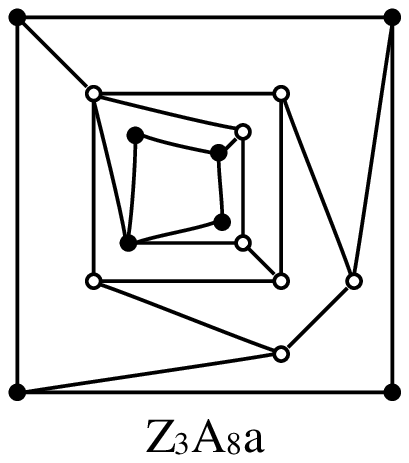}\\
\includegraphics[width=\sze]{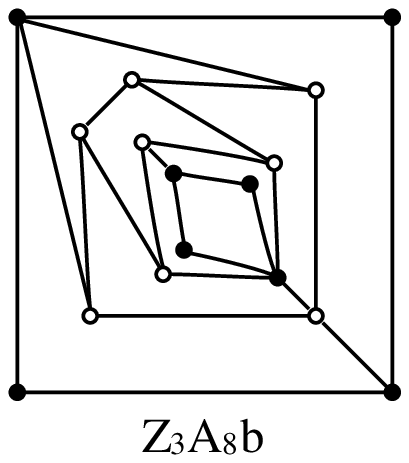}&
\includegraphics[width=\sze]{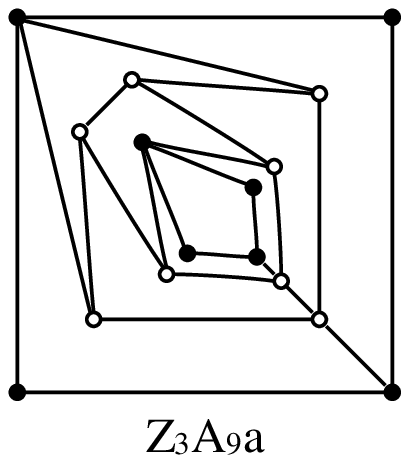}&
\includegraphics[width=\sze]{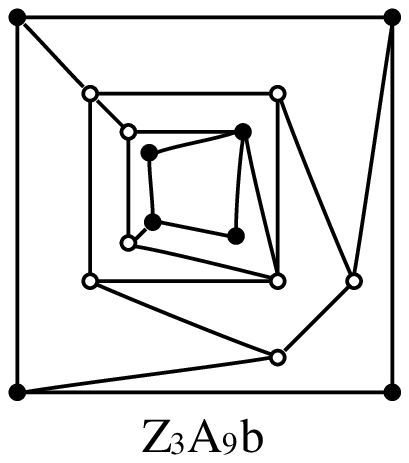}&
\includegraphics[width=\sze]{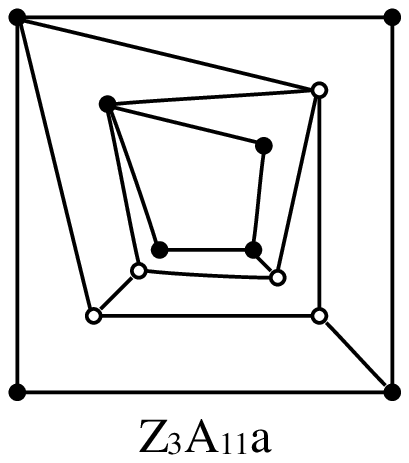}\\
\includegraphics[width=\sze]{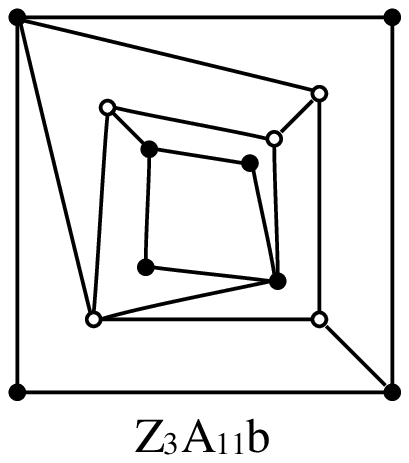}&
\includegraphics[width=\sze]{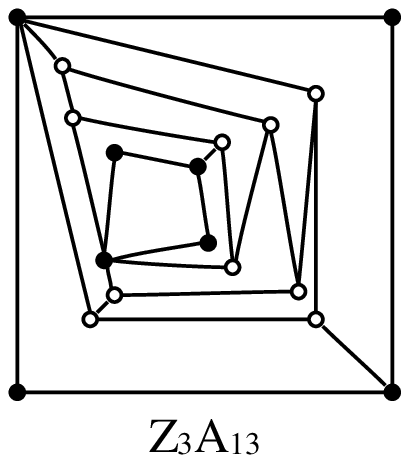}&
\includegraphics[width=\sze]{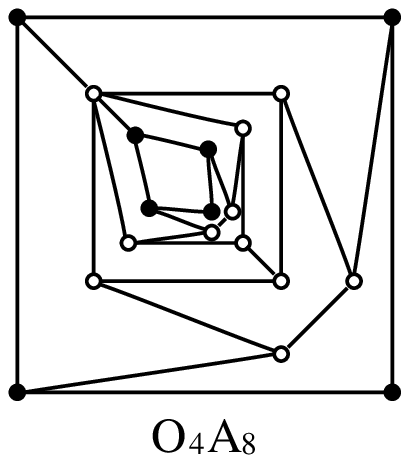}&
\includegraphics[width=\sze]{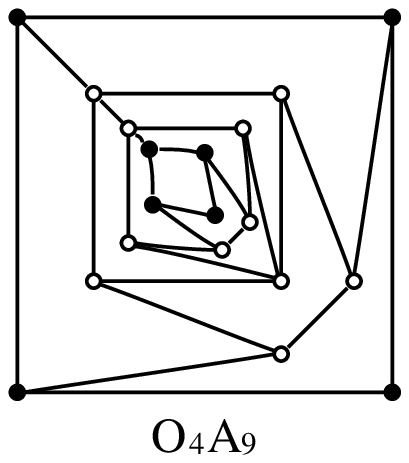}
\end{tabular}
\end{center}
\caption{Critical graphs on the cylinder, two separating $4$-cycles.}\label{fig-join3c}
\end{figure}

\begin{lemma}\label{lemma-join3}
Let $G$ be a connected graph embedded on the cylinder with distinct boundaries $C_1$ and $C_2$ such that $\ell(C_1), \ell(C_2)\le 4$.
Let $C,C'\subseteq G$ be distinct cycles of length
at most $4$ separating $C_1$ from $C_2$, such that $C$ separates $C_1$ from $C'$.  Assume that every cycle in $G$ distinct from $C$, $C'$, $C_1$ and $C_2$ has length at least $5$, and
that the distances between $C_1$ and $C$, between $C$ and $C'$, and between $C'$ and $C_2$ are at most $4$.
If $G$ is $(C_1\cup C_2)$-critical, then $G$ is isomorphic to one of the graphs drawn in
Figures~\ref{fig-join3a}, \ref{fig-join3b} and \ref{fig-join3c}.
\end{lemma}
\begin{proof}
By symmetry between $C_1,C$ and $C_2,C$, assume that $\ell(C)\le \ell(C')$.  Also, assume that $\ell(C_1)=\ell(C_2)=4$---the graphs
bounded by triangles follow by suppressing the precolored vertices of degree two.
Let $G_i$ be the subgraph of $G$ drawn between $C_i$ and $C$, for $i\in\{1,2\}$.  By Lemmas~\ref{lemma-crs}, \ref{lemma-cyl-le4} and \ref{lemma-join2},
$G_1$ is equal to one of the graphs drawn in Figures~\ref{fig-44} and \ref{fig-34}
and $G_2$ is equal to one of the graphs in Figures~\ref{fig-join2a}, \ref{fig-join2b} and \ref{fig-join2c}.

Suppose first that $\ell(C)=\ell(C')=3$.
It suffices to consider the graphs $G_1$ and $G_2$ such that $G_1$ is one of the graphs in Figure~\ref{fig-34}
and $G_2$ is one of the graphs in Figure~\ref{fig-join2a}, that is, $G_1\in\{Z_4,Z_5,O_5, O_6\}$ and $G_2\in \{D_2,D_4,D_8, D_9\}$.
Furthermore, it suffices to consider the pairs satisfying $c(G_1, C_1, C)+c(G_2,C_2,C)\ge 6$.
All critical graphs arising from these combinations are depicted in Figure~\ref{fig-join3a}.
Let us remark that combination $G_1=O_6$ and $G_2=D_2$ is the same as $Z_4D_4$ and combination $G_1=O_6$ and $G_2=D_8$ is the same as $Z_4D_9$. 

If $\ell(C)=3$ and $\ell(C')=4$, then we combine graphs $G_1$ from Figure~\ref{fig-34} with graphs $G_2$ from Figure~\ref{fig-join2c}
such that $c(G_1, C_1, C)+c(G_2,C_2,C)\ge 6$, i.e., $G_1=Z_4$ and $G_2\in\{X_1,X_3,X_4,X_5,X'_5, X_6\}$, or
$G_1\in\{Z_5, O_6\}$ and $G_2\in \{X_1,X_5, X'_5\}$.  All critical graphs arising from these combinations are depicted in Figure~\ref{fig-join3b}
(let us remark that the combination $G_1=O_6$ and $G_2=X_1$ is not critical, since the set of precolorings that extend to it is equal
to that of $D_{10}$, which is its subgraph).

Finally, if $\ell(C)=\ell(C')=4$, then we combine graphs $G_1$ from Figure~\ref{fig-44} with graphs $G_2$ from Figure~\ref{fig-join2b}
such that $c(G_1, C_1, C)+c(G_2,C_2,C)\ge 18$ and $C$ does not contain a non-precolored vertex of degree two.  Furthermore, if $G_1=Z_1$,
we can exclude from consideration the graphs such that $C$ is a cycle of non-precolored vertices of degree three, as an even cycle
of vertices of degree three cannot appear in any critical graph.  That is, for $G_1=Z_1$ we need to consider $G_2\in \{A_1, A_3, A_8, A_9, A_{13}\}$
(only $G_2=A_1$ results in a critical graph).  Almost all combinations need to be considered for $G_1=Z_3$, where
$G_2\in\{A_8, A_9,A_{11},A_{13}\}$ result in a critical graph.  Once these combinations are considered, we may assume that $G_2\not\in \{A_1,A_{11}\}$
by symmetry, since in these graphs the subgraph drawn between $C'$ and $C_2$ would be $Z_1$ or $Z_3$.
Finally, we need to consider the combinations $G_1\in\{O_2,O_3, O_4\}$ and $G_2\in\{A_5,A'_5,A_8, A_9\}$ or $G_1=T_1$ and $G_2=A_8$.
All the critical graphs obtained by these combinations are in Figure~\ref{fig-join3c}.
\end{proof}

The numbers of non-extending colorings for these graphs are

\begin{center}
\begin{tabular}{ccc|ccc}
$G$ & $c(G, B, T)$ & $c(G, T, B)$ & $G$ & $c(G, B, T)$ & $c(G, T, B)$ \\
\hline
$Z_4D_2$               & $12$ & $12$ & $Z_4X_2$            &  $3$ &  $9$ \\
$Z_4Z_4$               & $12$ &  $8$ & $Z_6X_5$, $Z_6X'_5$ &  $4$ &  $2$ \\
$Z_4D_8$               &  $6$ &  $6$ & $Z_1A_1$            &  $3$ &  $3$ \\
$Z_4Z_9$               &  $6$ &  $4$ & $Z_4A_1$            &  $3$ &  $1$ \\
$O_6D_4$               &  $8$ &  $8$ & $Z_4X_1a$           &  $1$ &  $1$ \\
$O_6D_9$               &  $4$ &  $4$ & $Z_3A_8a$           &  $4$ &  $4$ \\
$Z_4X_1b$              &  $9$ &  $9$ & $Z_3A_8b$           &  $8$ &  $8$ \\
$Z_4X_3$               &  $1$ &  $3$ & $Z_3A_9a$           &  $4$ &  $4$ \\
$Z_4X_4a$              &  $2$ &  $3$ & $Z_3A_9b$           &  $4$ &  $4$ \\
$Z_4X_4b$              &  $2$ &  $3$ & $Z_3A_{11}a$        & $12$ & $12$ \\
$Z_4X_5$, $Z_4X'_5$    &  $4$ &  $6$ & $Z_3A_{11}b$        & $12$ & $12$ \\
$Z_4X_6a$              &  $2$ &  $3$ & $Z_3A_{13}$         &  $4$ &  $4$ \\
$Z_4X_6b$              &  $2$ &  $3$ & $O_4A_8$            &  $2$ &  $2$ \\
$Z_5X_5$, $Z_5X'_5$    &  $4$ &  $6$ & $O_4A_9$            &  $2$ &  $2$ \\
\end{tabular}
\end{center}

This rather tedious case analysis concludes with the next lemma.

\begin{figure}
\begin{center}
\newcommand{\sze}{35mm}
\begin{tabular}{cccc}
\includegraphics[width=\sze]{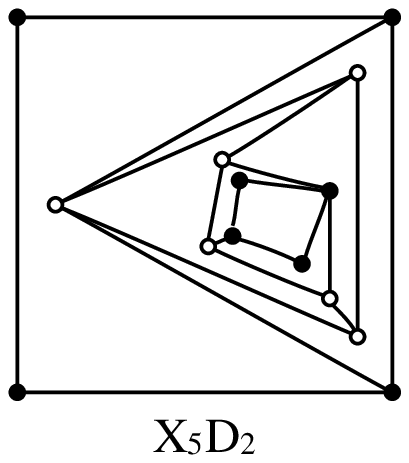}&
\includegraphics[width=\sze]{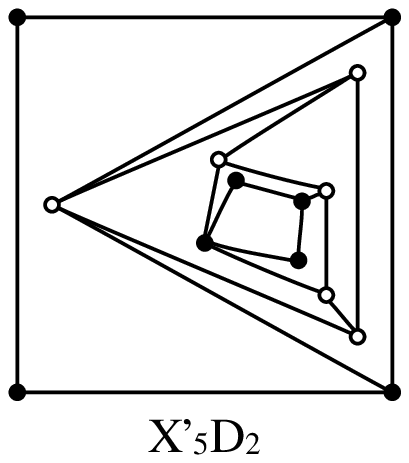}&
\includegraphics[width=\sze]{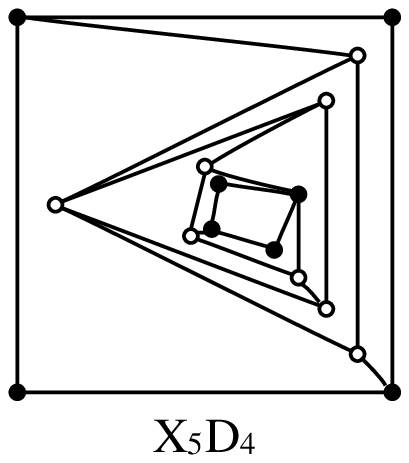}&
\includegraphics[width=\sze]{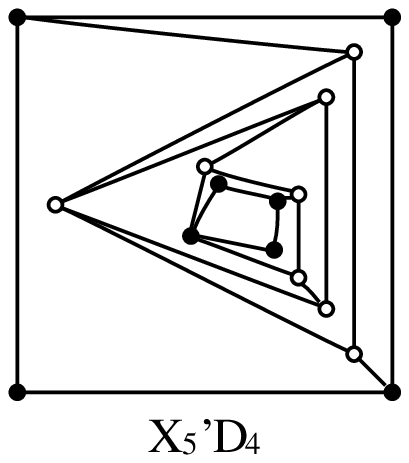}\\
\includegraphics[width=\sze]{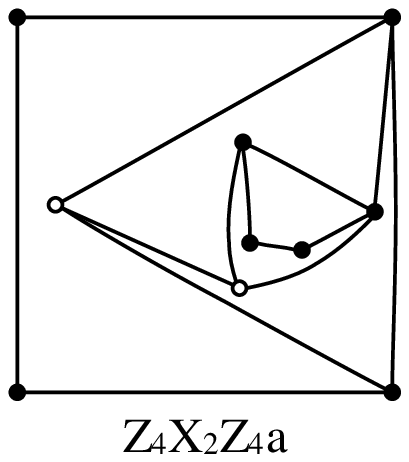}&
\includegraphics[width=\sze]{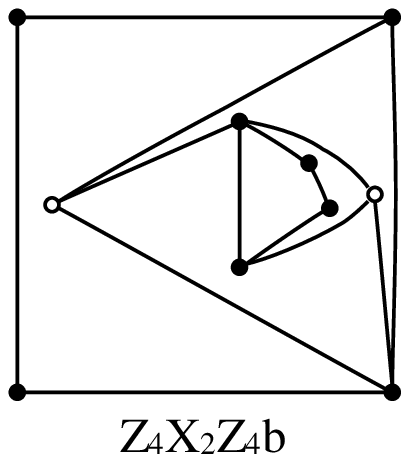}&
\includegraphics[width=\sze]{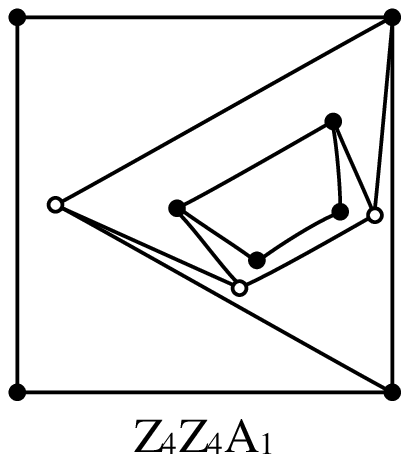}&
\includegraphics[width=\sze]{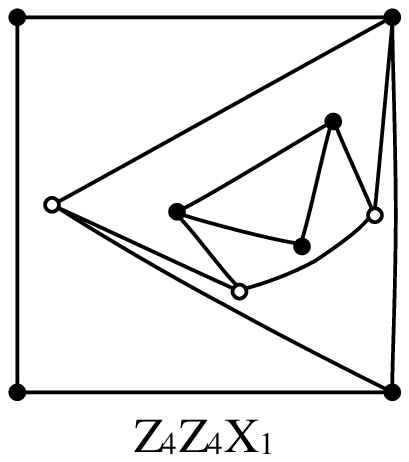}\\
\includegraphics[width=\sze]{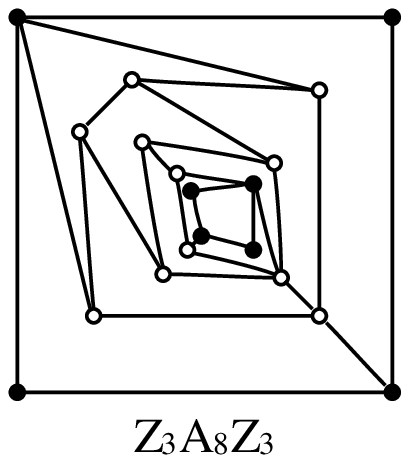}&
\includegraphics[width=\sze]{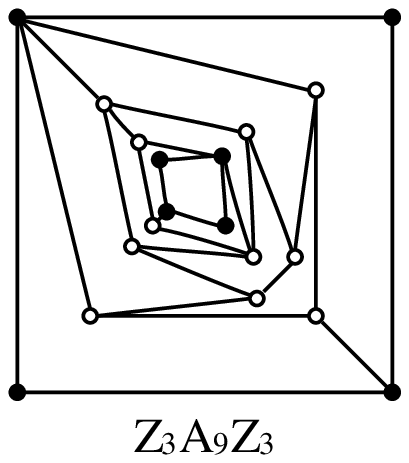}&
\includegraphics[width=\sze]{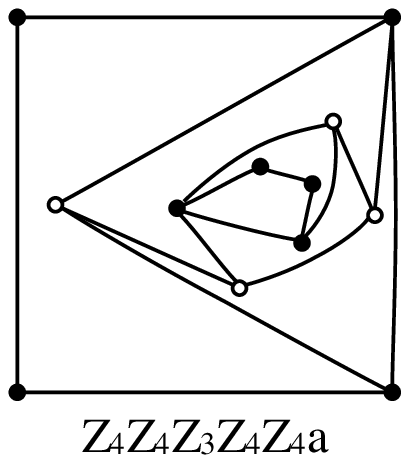}&
\includegraphics[width=\sze]{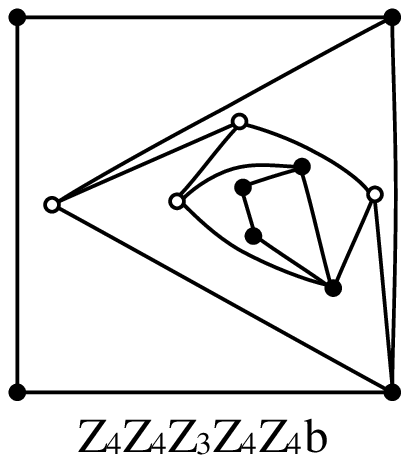}
\end{tabular}
\end{center}
\caption{Other critical graphs on the cylinder.}\label{fig-joinmany}
\end{figure}

\begin{lemma}\label{lemma-join}
Let $G$ be a connected graph embedded in the cylinder with distinct boundaries $C_1$ and $C_2$ such that $\ell(C_1), \ell(C_2)\le 4$.
Let $C_1=K_0$, $K_1$, \ldots, $K_k=C_2$ be a sequence of distinct cycles of
length at most $4$ in $G$ such that $K_i$ separates
$K_{i-1}$ from $K_{i+1}$ for $1\le i\le k-1$ and the distance between $K_i$ and $K_{i+1}$ is at most $4$ for $0\le i\le k-1$.
Assume that every cycle of length at most $4$ in $G$ is equal to $K_i$ for some $i\in\{0,\ldots, k\}$.
If $G$ is $(C_1\cup C_2)$-critical, then one of the following holds:
\begin{itemize}
\item $G$ is one of the graphs described by Lemmas~\ref{lemma-cyl-le4}, \ref{lemma-join2} or \ref{lemma-join3}, or
\item $G\in \CC$, or
\item $G$ is one of the graphs drawn in Figure~\ref{fig-joinmany}.
\end{itemize}
\end{lemma}
\begin{proof}
By Lemmas~\ref{lemma-cyl-le4}, \ref{lemma-join2} or \ref{lemma-join3}, we may assume that $k\ge 4$.  The graphs described by Lemma~\ref{lemma-join2}
satisfy that if $\ell(C)=\ell(C')=3$, then $\ell(C_1)=\ell(C_2)=4$.  Therefore, Lemma~\ref{lemma-crs} implies that at least one of $K_i$, $K_{i+1}$ and $K_{i+2}$ has length $4$,
for $0\le i\le k-2$.  For $1\le i\le k-1$, let $P_i$ (respectively $N_i$) be the subgraphs of $G$ drawn between $K_i$ and $C_1$
(respectively $C_2$).

Suppose first that $k=4$, and assume that $\ell(C_1)=\ell(C_2)=4$.
If $\ell(K_2)=\ell(K_3)=3$, then $P_2\in \{X_1,X_3,X_4,X_5,X'_5, X_6\}$ and $N_2\in \{D_2,D_4,D_8,D_9\}$.
Furthermore, $c(P_2,C_1,K_2)+c(N_2,C_2,K_2)\ge 6$, implying that $P_2\in\{X_1,X_5, X'_5\}$ and $N_2\in\{D_2,D_4\}$.  The critical graphs arising this way are
$X_5D_2$, $X'_5D_2$, $X_5D_4$ and $X'_5D_4$.
The case that $\ell(K_1)=\ell(K_2)=3$ is symmetric.  If $\ell(K_1)=\ell(K_3)=3$, then $P_3=Z_4X_2$ and $N_1=Z_4X_2$, and thus $N_3=Z_4$.  It follows that $G\in \{Z_4X_4Z_4a,Z_4X_4Z_4b\}$.
If $\ell(K_2)=3$ and $\ell(K_1)=\ell(K_3)=4$, then $P_2,N_2\in \{X_1,X_3,X_4,X_5,X'_5, X_6\}$, and since $c(P_2,C_1,K_2)+c(N_2,C_2,K_2)\ge 6$, we conclude that $P_2=N_2=X_1$.
However, the graph obtained by combining $X_1$ with itself is not critical.
If $\ell(K_1)=3$ and $\ell(K_2)=\ell(K_3)=4$, then $N_1=Z_4A_1$ and $P_1\in\{Z_4,Z_5,O_5,O_6\}$. Since $c(P_1,C_1,K_1)+c(N_1,C_2,K_1)\ge 6$, it follows that $P_1=Z_4$
and $G=Z_4Z_4A_1$.  The case that $\ell(K_3)=3$ and $\ell(K_1)=\ell(K_2)=4$ is symmetric.

Finally, consider the case that $\ell(K_1)=\ell(K_2)=\ell(K_3)=4$.
Then $P_3$ is one of the graphs in Figure~\ref{fig-join3c}, implying that $P_1\in\{Z_1, Z_3, O_4\}$, and by symmetry, $N_3\in \{Z_1,Z_3,O_4\}$.
If $N_3\neq Z_3$, we have $c(P_3, C_1, K_3)\ge 6$, and thus $P_3\in \{Z_3A_8b, Z_3A_{11}a, Z_3A_{11}b\}$.  For all these choices of $P_3$, we
have $P_1=Z_3$.  Therefore, by symmetry we may assume $N_3=Z_3$.  The combinations of $Z_3$ with the graphs in Figure~\ref{fig-join3c}
that result in a critical graph are $Z_3A_8Z_3$, $Z_3A_9Z_3$ and the graphs belonging to $\CC$.

The only graph with $k=4$ and $\ell(C_1)\le 3$ or $\ell(C_2)\le 3$ is $Z_4Z_4X_1$, obtained by suppressing a vertex of degree two in $Z_4Z_4A_1$.
Thus, all the graphs with $k=4$ satisfy the conclusion of this lemma.

Suppose now that $k=5$.  The graphs $P_4$ and $N_1$ are among the graphs described by this lemma for $k=4$.
This implies that $P_1\in\{Z_1,Z_3, Z_4\}$.  If $\ell(K_1)=3$, then $P_1=Z_4$ and $N_1=Z_4Z_4X_1$ and
$G=Z_4Z_4Z_1Z_4Z_4a$ or $G=Z_4Z_4Z_1Z_4Z_4b$.  The case that $\ell(K_4)=3$ is symmetric.
Therefore, assume that $\ell(K_1)=\ell(K_4)=4$.  This implies that $N_1\not\in \{Z_4X_2Z_4a, Z_4X_2Z_4b\}$.
Neither $Z_1$ nor $Z_4$ can be combined with a graph from $\CC$ to form a critical graph, as the resulting graph
would contain a non-precolored vertex of degree two.  The same argument shows that if $P_1\in \{Z_1,Z_4\}$, then
$N_1\not\in \{X_5D_4,X'_5D_4,Z_3A_8Z_3, Z_3A_9Z_3\}$.  The combinations of $Z_1$ or $Z_4$ with $X_5D_2$, $X'_5D_2$. $Z_4Z_4A_1$
or $Z_4Z_4X_1$
are not critical.  We conclude that $P_1=Z_3$, and by symmetry, $N_4=Z_3$.  Since $G$ does not contain non-precolored vertices
of degree two, we have $N_1\not\in\{X_5D_2,X'_5D_2, Z_4Z_4A_1,Z_4Z_4X_1\}$.  If $N_1\in \CC$, then $G\in \CC$.
Otherwise, $N_1\in \{X_5D_4,X'_5D_4,Z_3A_8Z_3, Z_3A_9Z_3\}$.  However, the combinations of $Z_3$ with these graphs are not critical.

Therefore, we may assume that $k\ge 6$.  Let $G_i$ be the subgraph of $G$ drawn between $K_i$ and $K_{i+5}$, for $0\le i\le k-5$.
By the previous paragraph, $G_i\in \{Z_4Z_4Z_1Z_4Z_4a,Z_4Z_4Z_1Z_4Z_4b\}\cup \CC$, hence $\ell(K_i)=\ell(K_{i+5})=4$.
Furthermore, considering $G_{i-1}$ (if $i>0$) or $G_{i+1}$ (if $i<k-5$), we conclude that $\ell(K_{i+1})=4$ or $\ell(K_{i+4})=4$,
implying that $G_i\in \CC$.  This implies that $G\in \CC$.
\end{proof}

Let us remark that if $G$ is a graph in $\CC$ with $4$-faces $C_1$ and $C_2$, then $G$ is $(C_1\cup C_2)$-critical---to see this observe that that the precolorings
of $C_1$ and $C_2$ in that the vertices of $C_i$ of degree two have different colors for each $i\in\{1,2\}$ do not extend to a coloring of $G$.
Let us now point out some consequences of Lemma~\ref{lemma-join} that are useful in the proof of Theorem~\ref{thm-cyl}.

\begin{figure}
\begin{center}
\newcommand{\sze}{35mm}
\begin{tabular}{ccc}
\includegraphics[width=\sze]{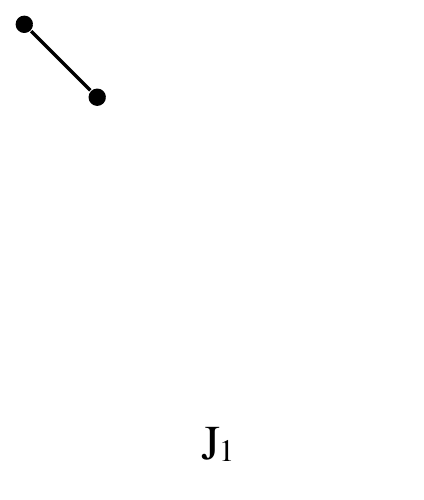}
&
\includegraphics[width=\sze]{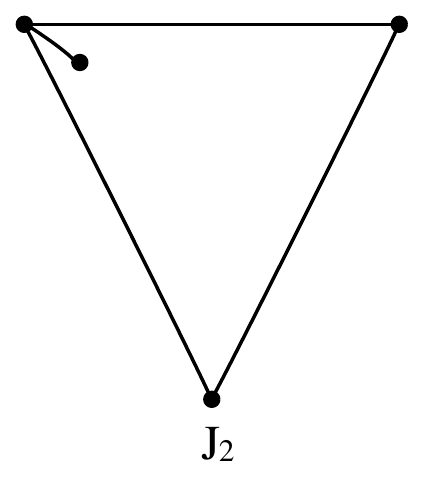}&
\includegraphics[width=\sze]{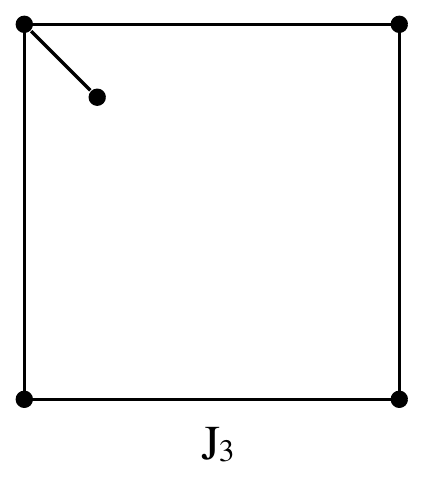}\\
\includegraphics[width=\sze]{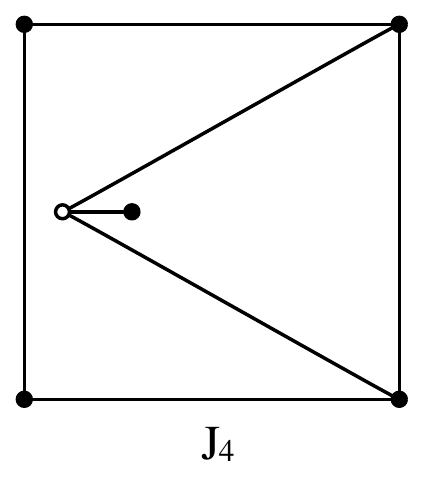}&
\includegraphics[width=\sze]{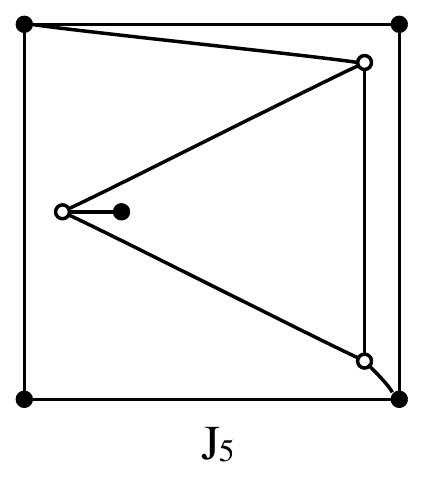}&
\end{tabular}
\end{center}
\caption{Critical graphs with a precolored vertex.}\label{fig-to1}
\end{figure}

\begin{lemma}\label{lemma-cyl-le41}
Let $G$ be a connected plane graph, $v$ a vertex of $G$ and $C\subseteq G$ either a vertex of $G$, or a cycle bounding a face of length at most $4$.
Assume that every cycle of length at most $4$ distinct from $C$ separates $v$ from $C$.  Furthermore, assume that for every two subgraphs $K_1,K_2\subseteq G$
such that $K_i\in\{v,C\}$ or $K_i$ is a cycle of length at most $4$ for $i\in\{1,2\}$, either the distance between $K_1$ and $K_2$ is at most $4$, or there
exists a cycle of length at most $4$ separating $K_1$ from $K_2$.
If $G$ is nontrivial $(v\cup C)$-critical, then $G$ is one of the graphs $J_1$, $J_2$, \ldots, $J_5$ drawn in Figure~\ref{fig-to1}.
\end{lemma}
\begin{proof}
Let $G'$ be the graph obtained from $G$ in the following way: Add new vertices $v'$ and $v''$ and edges of the triangle $C_1=vv'v''$.
If $C$ is a single vertex, add also new vertices $c'$ and $c''$ and edges of the triangle $C_2=Cc'c''$, otherwise set $C_2=C$.
Observe that $G'$ is $(C_1\cup C_2)$-critical and satisfies assumptions of Lemma~\ref{lemma-join}.  The claim follows by the
inspection of the graphs enumerated by Lemma~\ref{lemma-join}, using the fact that $C_1$ and $C_2$ are disjoint,
$\ell(C_1)=3$ and $v'$ and $v''$ have degree two.
\end{proof}

The following claims follow by a straightforward inspection of the graphs listed in Lemmas~\ref{lemma-join}
and \ref{lemma-cyl-le41}:

\begin{corollary}\label{cor-cyl-special}
Let $G$ be a connected plane graph and $C_1$ and $C_2$ distinct subgraphs of
$G$ such that $C_i$ is either a single vertex or a cycle of length at most $4$
bounding a face, for $i\in\{1,2\}$.  Assume that $G$ is nontrivial $(C_1\cup C_2)$-critical
and that every cycle of length at most
$4$ distinct from $C$ separates $C_1$ from $C_2$.  Furthermore, assume that for
every two subgraphs $K_1,K_2\subseteq G$ such that $K_i\in\{C_1,C_2\}$ or $K_i$
is a cycle of length at most $4$ for $i\in\{1,2\}$, either the distance between
$K_1$ and $K_2$ is at most $4$, or there exists a cycle of length at most $4$
separating $K_1$ from $K_2$.
\begin{itemize}
\item[(a)] If the distance between $C_1$ and $C_2$ is at least three and $G$ has a face of length at least $7$, then
\\ $G\in \{D_9,D_{10},A_{12}, A'_{12}, Z_4D_8, Z_4D_9, O_6D_9, Z_5X_5, Z_5X'_5\}$.
\item[(b)] If the distance between $C_1$ and $C_2$ is at least three and all cycles of length at most $4$ in $G$
distinct from $C_1$ and $C_2$ intersect in a non-precolored vertex, then\\
$G\in \{R, J_5, D_9, D_{10}, A_{10}, A_{12}, A'_{12}, Z_4D_4, O_6D_4, Z_4X_4b, Z_3A_9b, O_4A_9\}$.
\item[(c)] If the distance between $C_1$ and $C_2$ is at least three and all cycles of length at most $4$ in $G$
distinct from $C_1$ and $C_2$ intersect in a precolored vertex, then
$G\in \{R, A_{12}, A'_{12}, Z_4X_4b\}$.
\item[(d)] If the distance between $C_1$ and $C_2$ is at least two and $G$ has a face of length at least $9$,
then $G\in\{D_6,D_{10}\}$.
\item[(e)] If the distance between $C_1$ and $C_2$ is at least two and $G$ is not $2$-edge-connected,
then $G\in\{J_4, J_5, D_6,D_8,D_9,D_{10}, Z_4D_8, Z_4D_9, O_6D_9\}$.
\item[(f)] If the distance between $C_1$ and $C_2$ is at least two, $G$ has a face $M'$ of length at least $7$,
and there exists an edge $e$, a vertex $x$ and a face $M\neq M'$ distinct from $C_1$ and $C_2$ such that
\begin{itemize}
\item $M$ and $M'$ share the edge $e$, and $x$ is incident with $M$,
\item every path of length two between $C_1$ and $C_2$ contains the edge $e$, and
\item every path of length at most $4$ between $C_1$ and $C_2$ contains $e$ or $x$ or both, and
\item every cycle of length at most $4$ distinct from $C_1$ and $C_2$ contains $e$ or $x$ or both,
\end{itemize}
then $G\in\{D_9,D_{10},A_7, A_{12}, A'_{12}\}$.
\end{itemize}
\end{corollary}

Aksenov~\cite{aksen} proved that every planar graph with at most three triangles is $3$-colorable.
Let us note that the result was recently reproved with a significantly simpler proof~\cite{aksenNew}
and the description of planar graphs with $4$ triangles that are not $3$-colorable is known~\cite{fourTriangles}.
In the origianl proof, Aksenov showed
that for any plane graph $G$, if a face $B$ of length at most $4$ is precolored and $G$ contains at most one triangle distinct from $F$, then
the precoloring of $F$ extends.
\begin{theorem}[Aksenov~\cite{aksen}]\label{thm-aksen}
Let $G$ be a plane graph with the outer face $B$ of length at most $4$.  If $G$ is nontrivial $B$-critical, then $G$ contains at least two triangles
distinct from $B$.
\end{theorem}

The next lemma (following from Theorem~\ref{thm-aksen}) enables us to consider only connected graphs in the proof of Theorem~\ref{thm-cyl}:
\begin{lemma}\label{lemma-conn}
Let $G$ be a plane graph and $C_1$ and $C_2$ distinct subgraphs of $G$ such that $C_i$ is either a single vertex or a cycle of length at most $4$
bounding a face, for $i\in\{1,2\}$.  Assume that every cycle in $G$ of length at most $4$ separates $C_1$ from $C_2$.
If $G$ is nontrivial $(C_1\cup C_2)$-critical, then $G$ is connected.
\end{lemma}
\begin{proof}
We may assume that $C_1$ and $C_2$ are faces, since otherwise we can add a new
cycle of length three to $G$ to replace $C_1$ or $C_2$ if they were single
vertices.
If $G$ were not connected, then there would exist a cycle $K$ of length at most
$4$ and a nontrivial $K$-component $G'$ of $G$
such that $G'$ contains at most one triangle distinct from $K$.  By Lemma~\ref{lemma-crs}, $G'$ is $K$-critical, contradicting
Theorem~\ref{thm-aksen}.
\end{proof}

The following lemma finishes the proof of Theorem~\ref{thm-cyl}:
\begin{lemma}\label{lemma-cyl-ge5}
Let $G$ be a connected plane graph and $C_1$ and $C_2$ distinct subgraphs of $G$ such that $C_i$ is either a single vertex or a cycle of length at most $4$
bounding a face, for $i\in\{1,2\}$. Assume that every cycle in $G$ distinct from $C_1$ and $C_2$ has length at least $5$.
If $G$ is $(C_1\cup C_2)$-critical, then the distance between $C_1$ and $C_2$ is at most $4$.
\end{lemma}
\begin{proof}
Suppose for a contradiction that $G$ is a smallest counterexample to this claim, i.e., the distance between $C_1$ and $C_2$ is at least $5$,
and if $H$ is a graph satisfying the assumptions of Lemma~\ref{lemma-cyl-ge5} with $|E(H)|<|E(G)|$, then the distance between its precolored
cycles is at most $4$.  For future references, let us note that
\claim{cl-far}{the distance between $C_1$ and $C_2$ in $G$ is at least $5$,}
and
\claim{cl-cyc}{all cycles distinct from $C_1$ and $C_2$ in $G$ have length at least $5$.}
Let us now show some properties of $G$.

\claim{cl-2conn}{$G$ is $2$-connected.}
\begin{proof}
Suppose that $v$ is a cut-vertex in $G$.  Since $G$ is $(C_1\cup C_2)$-critical, Gr\"otzsch's theorem implies that $v$ separates
$C_1$ from $C_2$.  Let $G_1$ and $G_2$ be induced subgraphs of $G$ such that $G=G_1\cup G_2$, $V(G_1)\cap V(G_2)=\{v\}$,
$C_1\subseteq G_1$ and $C_2\subseteq G_2$.  By Lemma~\ref{lemma-crs}, $G_i$ is $(C_i\cup v)$-critical, for $i\in\{1,2\}$.
Furthermore, $|E(G_i)|<|E(G)|$, thus the distance between $C_i$ and $v$ is at most $4$.  By (\ref{cl-far}) and symmetry,
we may assume that the distance between $C_1$ and $v$ is at least three.
However, since $G_1$ does not contain a cycle of length at most $4$ distinct from $C_1$,
this contradicts Lemma~\ref{lemma-cyl-le41}.
\end{proof}

\claim{cl-no22}{No two vertices of degree two in $G$ are adjacent.}
\begin{proof}
Suppose that vertices $v_1$ and $v_2$ of degree two in $G$ are adjacent.  Since $G$ is critical, both $v_1$ and $v_2$ are precolored,
and by symmetry, we may assume that they belong to $C_1$.  Since $G$ is $2$-connected, it follows that $\ell(C_1)=4$.
Let $C_1=v_1v_2v_3v_4$.  By (\ref{cl-cyc}), $v_3v_4$ and $v_3v_2v_1v_4$
are the only paths of length at most three between $v_3$ and $v_4$.  It follows that the graph $G'$ obtained from
$G$ by identifying $v_1$ with $v_2$ to a new vertex $v$ does not contain a cycle of length at most $4$ distinct
from $vv_3v_4$ and $C_2$.  Furthermore, observe that $G'$ is $(vv_3v_4\cup C_2)$-critical,
$|E(G')|<|E(G)|$ and the distance between $vv_3v_4$ and $C_2$ is at least $5$.  This contradicts the minimality of $G$.
\end{proof}

Let us fix a precoloring $\varphi$ of $C_1\cup C_2$ that does not extend to a coloring of $G$.  By the minimality of $G$,
$\varphi$ extends to every proper subgraph of $G$ that contains $C_1\cup C_2$.

\claim{cl-samecol}{Let $v_1v_2v_3$ be a path in $C_1\cup C_2$ such that $v_2$ has degree two and is incident with
a face of length $5$.  Then $\varphi(v_1)=\varphi(v_3)$.}
\begin{proof}
By symmetry, assume that $v_1v_2v_3\subseteq C_1$.  Since $v_2$ is incident with a $5$-face and no cycle in $G$ distinct
from $C_1$ and $C_2$ has length $4$, we conclude that $\ell(C_1)=4$.  Let $C_1=v_1v_2v_3v_4$.
Suppose for a contradiction that $\varphi(v_1)\neq \varphi(v_3)$.
Let $v_1v_2v_3xy$ be a $5$-face, and let $G'$ be the graph obtained from $G-v_2$ by identifying $v_1$ and $x$ to a new vertex $z$.
Let $C'_1=v_3zv_4$.  Note that the precoloring of $C'_1\cup C_2$ given by
$\varphi$ does not extend to a coloring of $G'$, thus $G'$ contains
a nontrivial $(C'_1\cup C_2)$-critical subgraph $G''$ such that $\varphi$ does not extend to a coloring of $G''$.  The distance between $C'_1$ and $C_2$ is at least $4$.
Observe that $G''$ contains a cycle $C$ of length at most $4$ distinct from $C'_1$ and $C_2$, as otherwise we would obtain a contradiction
with Lemmas~\ref{lemma-cyl-le4} or \ref{lemma-cyl-le41} or with the minimality of $G$.

Since $C$ does not exist in $G$, we have $z\in V(C)$.  Let $K$ be a cycle in $G$ induced by $(V(C) \setminus \{z\}) \cup \{v_1yx\}$.
Note that $V(K)$ indeed induces a cycle since it cannot have any chords.
Moreover, $K$ does not bound a face, since $y$ has degree at least three.  
Suppose that the exterior of $K$ contains both $C_1$ and $C_2$.
Corollary~\ref{cor-cycles} applied on $K$ and its nonempty interior 
contradicts the criticality of $G$.
Hence $K$ separates
$C_1$ from $C_2$, and thus $C$ separates $C'_1$ from $C_2$ in $G''$.  Choose $C$ among the cycles of length at most $4$ in $G''$
distinct from $C'_1$ and $C_2$ so that the subgraph $G''_2\subseteq G''$ drawn between $C$ and $C_2$ is as small as possible.
This implies that all cycles in $G''_2$ distinct from $C$ and $C_2$ have length at least $5$.  By Lemma~\ref{lemma-crs}, $G''_2$
is $(C\cup C_2)$-critical, and by the minimality of $G$, the distance between $C$ and $C_2$ is at most $4$.
Lemmas~\ref{lemma-cyl-le4} and \ref{lemma-cyl-le41} imply that the distance between $C$ and $C_2$ is at most three.

On the other hand, since $z\in V(C)$, by (\ref{cl-far}) the distance
between $C$ and $C_2$ is at least three.  Therefore, the distance between $C$ and $C_2$ is exactly three.
By Lemma~\ref{lemma-cyl-le4}, $\ell(C)=\ell(C_2)=4$ and $G''_2=R$.  The graph $G''$ contradicts Lemma~\ref{lemma-join}.
\end{proof}

We call a vertex $v$ {\em light} if $v$ is not precolored and the degree of $v$ is exactly three.

\claim{cl-no53}{The graph $G$ does not contain the following configuration: A $5$-face $F=v_1v_2v_3v_4v_5$ such that $v_1$, $v_3$, $v_4$ and $v_5$ are light,
and either $v_2$ is light, or both $v_4$ and $v_5$ have a precolored neighbor.}
\begin{proof}
If $v_i$ is a light vertex of $F$, then let $x_i$ be the neighbor of $v_i$ that is not incident with $F$, for $1\le i\le 5$.
By (\ref{cl-cyc}), the vertices $x_1$, \ldots, $x_5$ are distinct.  By (\ref{cl-far}), we may assume that all precolored neighbors of the vertices of $F$ belong to $C_1$.

If both $x_4$ and $x_1$ are precolored, then there exists a path $P\subseteq C_1$ joining $x_1$ and $x_4$ and a closed
region $\Delta$ of the plane bounded by the cycle $K$ formed by $P$ and $x_1v_1v_5v_4x_4$ such that $\Delta$ contains neither $C_1$ nor $C_2$.
Since $\ell(K)\le 7$, Corollary~\ref{cor-cycles} implies that the open interior of $\Delta$ is a face.  However, $\Delta\neq F$, which implies
that $v_5$ has degree two.  This is a contradiction, thus at most one of $x_1$ and $x_4$ is precolored.  Similarly, at most one of $x_3$ and $x_5$
is precolored.

If $v_2$ is light, then by the symmetry of $F$ we may assume that either no vertex of $F$ has a precolored neighbor, or
that $x_4$ is precolored and $x_3$ is not.  By the previous paragraph, this also implies that $x_1$ is not precolored.
If $v_2$ is not light, then both $x_4$ and $x_5$ are precolored, and thus neither $x_1$ nor $x_3$ are precolored.

Let $G'$ be the graph obtained from $G$ by removing the light vertices of $F$ and adding the edge $x_1x_3$.
Since $x_1\neq x_3$, $G'$ has no loops.  Suppose that $\varphi$ extends to a coloring $\psi$ of $G'$.
If $v_2$ is light, then each vertex of $F$ has one precolored neighbor, thus it has two available colors.
Furthermore, the lists of colors available at $v_1$ and $v_3$ are not the same, thus $\psi$ extends to
a coloring of $F$, giving a coloring of $G$ that extends $\varphi$.
Suppose that $v_2$ is not light.
If $\psi(x_1)=\psi(v_2)$, then we can color vertices of $F$ in order $v_3$, $v_4$, $v_5$, $v_1$.
Similarly, $\psi$ extends to $F$ if $\psi(x_3)=\psi(v_2)$.  Therefore, assume that $\psi(x_1)=1$,
$\psi(v_2)=2$ and $\psi(x_3)=3$.  Then, color $v_1$ by $3$, $v_3$ by $1$ and extend the coloring to $v_4$
and $v_5$.  This is possible, since by (\ref{cl-samecol}), $\varphi(x_4)=\varphi(x_5)$.
We conclude that $\varphi$ extends to a coloring of $G$, which is a contradiction.

Therefore, $\varphi$ does not extend to a coloring of $G'$, and $G'$ has a nontrivial $(C_1\cup C_2)$-critical
subgraph $G''$.  By the minimality of $G$, we have $x_1x_3\in E(G'')$.  Note that the distance between $C_1$ and
$C_2$ in $G''$ is at least three, since neither $x_1$ nor $x_3$ is precolored.  Also, every cycle in $G''$
of length at most $4$ distinct from $C_1$ and $C_3$ contains the edge $x_1x_3$.  If $x_1x_3$ were an edge-cut,
then $G''$ contains no such cycle, and thus $G''$ would contradict Lemmas~\ref{lemma-cyl-le4} and \ref{lemma-cyl-le41}
or the minimality of $G$.  It follows that $x_1x_3$ is incident with two distinct faces in $G''$.
For a cycle $M$ in $G''$ containing $x_1x_3$,
let $\overline{M}$ be the closed walk in $G$ obtained from $M$ by replacing $x_1x_3$ by the path $x_1v_1v_2v_3x_3$.
Let $F'$ be the face of $G''$ incident with $x_1x_3$ such that the interior of the corresponding region bounded by
$\overline{F'}$ in $G$ contains $v_4$ and $v_5$.  By Corrolary~\ref{cor-cycles}, $\ell(\overline{F'})\ge 10$,
and thus $\ell(F')\ge 7$.

Consider a cycle $C$ of length at most $4$ in $G''$ distinct from $C_1$ and $C_2$.  If the cycle $\overline{C}$ of length at most $7$ does not separate
$C_1$ from $C_2$, then by Corollary~\ref{cor-cycles} it bounds a face.  Since $\overline{C}\neq F$,
we conclude that $v_2$ has degree two.  This is a contradiction,
since $v_2$ is not precolored.  We conclude that $C$ separates $C_1$ from $C_2$.
By Lemma~\ref{lemma-conn}, $G''$ is connected.  Furthermore, by the minimality of $G$, the graph $G''$ satisfies the
assumptions of Corollary~\ref{cor-cyl-special}.  Since the distance between $C_1$ and $C_2$ in $G''$ is at least three
and $G''$ has a face $F'$ of length at least $7$, Corollary~\ref{cor-cyl-special}(a) implies that
$G''\in \{D_9,D_{10},A_{12}, A'_{12}, Z_4D_8, Z_4D_9, O_6D_9, Z_5X_5, Z_5X'_5\}$.  Furthermore,
since all cycles of length at most $4$ distinct from $C_1$ and $C_2$ in $G''$ contain a common edge $x_1x_3$,
we conclude that $G''\in \{D_9,D_{10}, A_{12}, A'_{12}\}$.  Since neither $x_1$ nor $x_3$ is precolored,
the inspection of the possible choices for $G''$ shows that $G''$ contains a path $Q$ of length at most $3$
joining a vertex of $C_1$ with a vertex of $C_2$,
such that $x_1x_3\not\in E(Q)$.  However, $Q$ is a subgraph of $G$, contradicting (\ref{cl-far}).
\end{proof}

\claim{cl-no7}{The graph $G$ does not have any face $F$ of length at least $7$.}
\begin{proof}
Suppose for a contradiction that $F=v_1v_2\ldots v_k$ is a face of length $k\ge 7$ in $G$.
Since the distance between $C_1$ and $C_2$ is at least $5$, we may assume that $v_1$, $v_2$ and
$v_3$ are not precolored.  Let $G'$ be the graph obtained from $G$ by identifying $v_1$ with $v_3$
to a new vertex $v$.  Observe that $\varphi$ does not extend to a coloring of $G'$, thus $G'$
has a nontrivial $(C_1\cup C_2)$-critical subgraph $G''$.  The distance between $C_1$ and $C_2$
in $G''$ is at least three.  Furthermore, since $v_2$ has degree at least three and every cycle $C$
in $G''$ of length at most $4$ distinct from $C_1$ and $C_2$ corresponds to a cycle of length at
most $6$ in $G$ containing the path $v_1v_2v_3$, Corollary~\ref{cor-cycles} implies that each such cycle $C$ separates
$C_1$ from $C_2$ and satisfies $v\in V(C)$.  By the minimality of $G$, the graph $G''$ satisfies the
assumptions of Corollary~\ref{cor-cyl-special}(b).  By (\ref{cl-far}),
all paths of length at most $4$ in $G''$ between $C_1$ and $C_2$ contain
$v$.  This implies that $G''\in\{J_5, D_9, D_{10}, Z_4D_4, O_6D_4\}$.  However, in these graphs, it is
not possible to split $v$ to two vertices ($v_1$ and $v_3$) in such a way that the resulting graph contains neither
a path of length at most $4$ between $C_1$ and $C_2$ nor a cycle of length at most $4$ distinct from $C_1$ and $C_2$,
which is a contradiction.
\end{proof}

\claim{cl-no6}{All faces of $G$ distinct from $C_1$ and $C_2$ have length $5$.}
\begin{proof}
Let $F=v_1v_2\ldots v_k$ be a face of $G$ distinct from $C_1$ and $C_2$.  By (\ref{cl-no7}),
$k\le 6$.  Suppose for a contradiction that $k=6$.  Let us first consider the case that $v_1$, $v_3$ and $v_5$
are not precolored.
Then consider the graph $G'$ obtained from $G$ by identifying $v_1$, $v_3$ and $v_5$
to a single vertex $v$.  Note that $\varphi$ does not extend to a coloring of $G'$, thus $G'$ has a non-trivial
$(C_1\cup C_2)$-critical subgraph $G''$.

Consider now a cycle $C$ of length at most $4$ in $G''$ distinct from $C_1$ and $C_2$.
Note that $v\in V(C)$, and by symmetry, we may assume that a cycle $K\subseteq G$ can be obtained from $C$ by replacing
$v$ by $v_1v_2v_3$.
Suppose that $C$ does not separate $C_1$ from $C_2$.  Then $K$ does not separate $C_1$ from $C_2$, and by
Corollary~\ref{cor-cycles}, $K$ bounds a face distinct from $F$, hence $v_2$ has degree two.  Since neither $v_1$ nor $v_3$ is
precolored, we conclude that $C_1=v_2$ or $C_2=v_2$.  But that implies that $C$ separates $C_1$ from $C_2$ in $G''$, which is a
contradiction.  Therefore, every cycle of length at most $4$ distinct from $C_1$ and $C_2$ separates $C_1$ from $C_2$ in $G''$.

As in the proof of (\ref{cl-no7}), we conclude that $G''\in\{J_5, D_9, D_{10}, Z_4D_4, O_6D_4\}$.  Furthermore,
since it is possible to split $v$ to three vertices $v_1$, $v_3$ and $v_5$ so that the resulting graph contains neither
a cycle of length at most $4$ distinct from $C_1$ and $C_2$ nor a path between $C_1$ and $C_2$ of length at most $4$,
we have $G''\in\{J_5, D_9, D_{10}\}$.  Furthermore, we may assume that $C_1=c_1c_2c_3c_4$ has length $4$,
there exists a path $c_1w_1w_2c_3$, $v_1$ is adjacent to $w_1$, $v_3$ is adjacent to $w_2$ and $v_5$ is adjacent to a vertex of $C_2$
in $G$.  We choose the labels of $c_2$ and $c_4$ so that the $8$-cycle $c_1w_1v_1v_2v_3w_2c_3c_4$ does not separate
$C_1$ from $C_2$.  Since $v_2$ cannot be a non-precolored vertex of degree two, Corollary~\ref{cor-cycles} implies that $v_2$ is
adjacent to $c_4$, and it is not precolored.  Corollary~\ref{cor-cycles} also implies that $c_1c_2c_3w_2w_1$ is a face.
By (\ref{cl-far}), $v_4$ and $v_6$ are not precolored.  Therefore, we may
identify $v_2$, $v_4$ and $v_6$ instead, and by a symmetric argument, we conclude that $\ell(C_2)=4$ and $G$ is the graph
depicted in Figure~\ref{fig-6face}.  However, this graph is not $(C_1\cup C_2)$-critical.

\begin{figure}
\begin{center}
\includegraphics{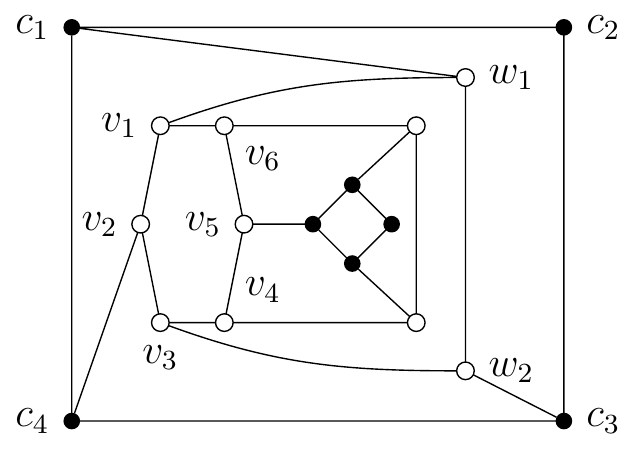}
\end{center}
\caption{A graph with a $6$-face.}\label{fig-6face}
\end{figure}

It follows that at least one of $v_1$, $v_3$ and $v_5$ is precolored, and by symmetry, at least one of $v_2$, $v_4$ and $v_6$ is
precolored.  If $v_1$ and $v_4$ were precolored and the rest of the vertices of $F$ were internal, then Corollary~\ref{cor-cycles}
implies that $v_1v_2v_3v_4$ or $v_1v_6v_5v_4$ together with a path in $C_1\cup C_2$ bounds a face, implying that $v_2$ or $v_6$
have degree two.  This is a contradiction, thus by symmetry, we may assume that $v_1,v_2\in V(C_1)$.  Since
$G$ does not contain a cycle of length at most $4$ distinct from $C_1$ and $C_2$, at least one of $v_3$ and $v_6$, say $v_3$, is not precolored.
Also, Corollary~\ref{cor-cycles} implies that $v_4$ and $v_5$ are not precolored.  Let us consider the graph $G'$ obtained by identifying
$v_1$, $v_3$ and $v_5$ to a single vertex $v$ and its $(C_1\cup C_2)$-critical subgraph $G''$.  By Corollary~\ref{cor-cyl-special}(c),
we have $G''\in \{R, A_{12}, A'_{12}, Z_4X_4b\}$.  However, all these graphs contain a path of length at most $4$ joining $C_1$ and $C_2$
that does not contain $v$, contradicting (\ref{cl-far}).
\end{proof}

\claim{cl-no51}{No face of $G$ is incident with $4$ light vertices.}
\begin{proof}
Suppose for a contradiction that $F=v_1v_2v_3v_4v_5$ is a face of $G$ such that $v_1$, $v_3$, $v_4$ and $v_5$
are light.  For $i\in\{1,3,4,5\}$, let $x_i$ be the neighbor of $v_i$ that is not incident with $F$.
By (\ref{cl-cyc}), the vertices $x_i$ are
distinct and $x_4$ is not adjacent to $x_5$.  Also, by (\ref{cl-no53}), we may assume that $x_4$ is not
precolored and that $v_2$ is not light. See Figure~\ref{fig-4light}.

\begin{figure}
\begin{center}
\includegraphics{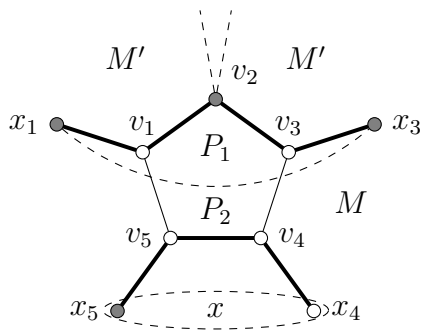}
\end{center}
\caption{Configuration from (\ref{cl-no51}).}\label{fig-4light}
\end{figure}

Let $G'$ be the graph obtained from $G$ by removing $v_1$, $v_3$, $v_4$ and $v_5$, identifying $x_4$ with $x_5$
to a new vertex $x$, and adding the edge $x_1x_3$.  Consider a coloring $\psi$ of $G'$.  We show that $\psi$
extends to a coloring of $G$:  Color both $x_4$ and $x_5$ by $\psi(x)$.  If $\psi(v_2)=\psi(x_1)$, then color $v_3$, $v_4$, $v_5$ and $v_1$ in this
order; each vertex has at least one available color.  The case that $\psi(v_2)=\psi(x_3)$ is symmetric.
Finally, if $\psi(x_1)$, $\psi(x_3)$ and $\psi(v_2)$ are pairwise different, then color
$v_1$ by $\psi(x_3)$ and $v_3$ by $\psi(x_1)$, and extend this coloring to $v_4$ and $v_5$ (this is possible,
since $x_4$ and $x_5$ are both colored by the same color $\psi(x)$).  We conclude that $\varphi$ does not extend to a coloring of $G'$,
and thus $G'$ has a nontrivial $(C_1\cup C_2)$-critical subgraph $G''$.

Let $C$ be a cycle of length at most $4$ in $G''$ distinct from $C_1$ and $C_2$, and let $K$ be the corresponding cycle in $G$, obtained by replacing
the edge $x_1x_3$ by the path $P_1=x_1v_1v_2v_3x_3$ or the vertex $x$ by the path $P_2=x_4v_4v_5x_5$ (or both).
Suppose that $C$ does not separate $C_1$ from $C_2$.  If $\ell(K)\le 7$, then Corollary~\ref{cor-cycles} implies that $K$ bounds a face,
and by (\ref{cl-no6}), $\ell(K)=5$.  However, that implies $\ell(C)\le\ell(K)-3\le 2$, which is a contradiction.  Therefore, $\ell(K)\ge 8$,
and thus $P_1, P_2\subseteq K$.  By planarity, $K-(P_1\cup P_2)$ consists of paths $Q_1$ between $x_1$ and $x_5$ and $Q_2$ between $x_3$ and $x_4$.
However, since $\ell(C)\le 4$, at least one of $Q_1$ and $Q_2$ has length one, contradicting (\ref{cl-cyc}).
We conclude that $C$ separates $C_1$ from $C_2$.  By Lemma~\ref{lemma-conn}, $G''$ is connected.

By (\ref{cl-far}), if $x_1\in V(C_i)$, then $x_3\not\in V(C_{3-i})$ for $i\in \{1,2\}$.
Also, if $x_5\in V(C_i)$, then $x_4$ has no neighbor in $V(C_{3-i})$.  It follows that the distance between $C_1$ and $C_2$ in $G''$ is at least two.
Let us also note that by (\ref{cl-no6}), the distance between $x_4$ and $x_5$ in $G$ is two, thus if $x_1x_3\not\in E(G'')$, then the distance between
$C_1$ and $C_2$ in $G''$ is at least three.

Suppose now that $v_2$ has degree at most three in $G$.  Since $v_2$ is not light and $v_1$ and $v_3$ are light, this implies that
$v_2=C_1$ or $v_2=C_2$. We assume the former.  In $G''$, $v_2$ has degree at most one.  By Lemma~\ref{lemma-cyl-le41}, $G''\in \{J_4,J_5\}$.
Let $x_2$ be the neighbor of $v_2$ in $G''$.  Note that $x_2\not\in \{x_1,x_3, x\}$ by (\ref{cl-cyc}).
Let $x'_2$ and $x''_2$ be the neighbors of $x_2$ distinct from $v_2$.  Similarly, we conclude that $\{x'_2,x''_2\}\cap \{x_1,x_3\}=\emptyset$.
Since $x_2x'_2x''_2$ is a triangle, we have $x\in \{x'_2,x''_2\}$, say $x=x''_2$.  Then a path starting with $v_2x_2x'_2$ shows that the distance between $C_1$ and $C_2$
in $G$ is at most three, which is a contradiction.  We conclude that $v_2$ has degree at least $4$.

For a face $M$ of $G''$, let $G_M$ be the subgraph of $G$ drawn in the region of the plane corresponding to $M$, bounded by the closed walk $\overline{M}$ obtained from the
boundary walk of $M$ by replacing $x_1x_3$ by $P_1$ or $x$ by $P_2$ (or both).  Let us note that the open interior of this region is either an open disk, or
a union of two open disks (the latter is the case when both $x_1x_3$ and $v_2$ are incident with $M$).

Suppose first that $x_1x_3\not\in E(G'')$.  Let us recall that in this case, the distance between $C_1$ and $C_2$ in $G''$ is at least three.
Since $\varphi$ extends to every proper subgraph of $G$, we have $x\in V(G'')$.  Let $M$ be the face of
$G''$ such that $v_1v_5, v_3v_4\in E(G_M)$.  If $\ell(\overline{M})=\ell(M)+6$, then $x$ forms a cut in $G''$ and $G''$ contains no cycle of length
at most $4$ distinct from $C_1$ and $C_2$.  We conclude that $G''=R$ because the distance between $C_1$ and $C_2$ in $G''$ is at least three.
However, $R$ is $2$-connected,
which is a contradiction.  Therefore, $\ell(\overline{M})=\ell(M)+3$.  If $v_2\in V(G'')$, then $G_M$ contains two $2$-chords $v_5v_1v_2$ and $v_4v_3v_2$.  If
$v_2\not\in V(G'')$, then $G_M$ contains a vertex $v_2$ of degree at least $4$ not contained in $\overline{M}$.  In both cases, Theorem~\ref{thm-crit12}
implies that $\ell(\overline{M})\ge 12$, and thus $\ell(M)\ge 9$.  By Corollary~\ref{cor-cyl-special}(d), $G''\in\{D_6,D_{10}\}$.
The former is not possible, since the distance between $C_1$ and $C_2$ in $G''$ is at least three, thus $G''=D_{10}$.  Note that $x$ must lie in
the triangle in $D_{10}$.  Observe that it is not possible to split such a vertex to two vertices $x_4$ and $x_5$ so that the distance between $C_1$ and $C_2$ is more than three,
contradicting (\ref{cl-far}).  This implies that $x_1x_3\in E(G'')$.

Let $M$ be the face of $G''$ incident with $x_1x_3$ such that $v_1v_5, v_3v_4\in E(G_M)$, and $M'$ the other face incident with $x_1x_3$.
Suppose first that $M=M'$.  Then by Corollary~\ref{cor-cyl-special}(e), $G''\in\{J_4, J_5, D_6,D_8,D_9,D_{10}, Z_4D_8, Z_4D_9, O_6D_9\}$.
Since $x_1x_3$ is not contained in any cycle, any cycle of length at most $4$ in $G''$ distinct from $C_1$ and $C_2$ contains $x$.
This implies that $G''\not\in \{Z_4D_8, Z_4D_9, O_6D_9\}$.  By inspection of the remaining choices for $G''$ we conclude that $x_1$ or $x_3$ is precolored, and belongs to say $C_2$.
Since the distance between $C_1$ and $C_2$ in $G$ is at least $5$, neither $x_4$ nor $x_5$ belongs to $C_1$, thus 
$G''\not\in \{J_4, D_6, D_8\}$ and $x$ is not precolored.  By symmetry, assume that $x_3\in V(C_2)$.  Then $x_1$ and $x$ both belong to a triangle in $G''$.
Note that $x_1$ and $x_5$ are not adjacent, thus $x_1$ is adjacent to $x_4$.  By (\ref{cl-far}),
$x_4$ is not adjacent to a vertex in $C_1$.  The $5$-cycle $x_1v_1v_5v_4x_4$ in $G$ does not
bound a face, thus by Corollary~\ref{cor-cycles} it separates $C_1$ from $C_2$.  Let $y$ be the common neighbor of $x_1$ and $x_5$ in $G$.
Since $G''\in \{J_5, D_9,D_{10}\}$, $x_5$ and $y$ have neighbors in $C_1$.  By (\ref{cl-no6}), $x_4$ and $x_5$ have a unique common neighbor $z$
in $G$; also, the distance between $x_3$ and $x_4$ is two.  By (\ref{cl-far}), $z\not\in V(C_1)$.  Note that $z\neq x_1$, since otherwise
the $4$-cycle $x_1yx_5z$ would contradict (\ref{cl-cyc}).  Observe that $G$ contains an $8$-cycle $K$ consisting of $x_5zx_4x_1y$, edges between $y$ and
$C_1$ and between $x_5$ and $C_1$ and a path in $C_1$ such that $K$ does not separate $C_1$ from $C_2$.  Since $z$ has degree at least three,
Corollary~\ref{cor-cycles} implies that $z$ is adjacent to a vertex of $C_1$.  However, this implies that the distance of $x_4$ is two both
from $x_3\in V(C_2)$ and from $C_1$, contradicting (\ref{cl-far}).

Therefore, $M\neq M'$.  Let us note that every path of length two between $C_1$ and $C_2$ in $G''$ contains
$x_1x_3$, and every path of length at most $4$ contains $x_1x_3$ or $x$, and both $x_1x_3$ and $x$ are incident with $M$.
Suppose now that $\ell(M')\ge 7$.  Corollary~\ref{cor-cyl-special}(f) implies that $G''\in \{D_9,D_{10},A_7, A_{12}, A'_{12}\}$.
If $G''\in \{D_9,D_{10}\}$, then $x$ is adjacent both to $x_1$ and $x_3$, contradicting (\ref{cl-cyc}) or planarity.
Otherwise, let $C\subseteq G''$ be the cycle of length $4$ distinct from $C_1$ and $C_2$. Note that $x_1x_3\not\in E(C)$,
thus $x\in V(C)$.  But if $x$ is split to two vertices ($x_4$ and $x_5$) so that $C$ is not a cycle on the resulting graph, 
then the resulting graph contains
a path of length at most $4$ between $C_1$ and $C_2$ that does not contain $x_1x_3$, contradicting (\ref{cl-far}).
It follows that $\ell(M')\le 6$.

By (\ref{cl-cyc}), every path between $x_1$ and $v_2$ other than $x_1v_1v_2$ and every path between $x_3$ and $v_2$ other
than $x_3v_3v_2$ has length at least three.  Since $\ell(M')\le 6$, we conclude that $v_2$ is not incident with $M'$,
and thus $v_2\not\in V(G'')$.  Therefore, $G_{M'}$ is not a union of two cycles intersecting in $v_2$.
Since $v_2$ has degree at least $4$, Theorem~\ref{thm-crit12} implies that $\ell(\overline{M'})\ge 10$.
It follows that $\ell(M')<\ell(\overline{M'})-3$, and thus $P_2\subseteq \overline{M'}$.

Let us again consider a cycle $C$ of length at most $4$ in $G''$ distinct from $C_1$ and $C_2$, and let $K$ be the corresponding cycle in $G$, obtained by replacing
$x_1x_3$ by $P_1$ or $x$ by $P_2$ or both.  Since $M$ and $M'$ are both incident with both $x$ and $x_1x_3$, there is a cut in $G''$ formed by $x$ and $x_1x_3$. Thus $P_1\cup P_2\subseteq K$.  However, this contradicts
(\ref{cl-cyc}) or planarity.  It follows that $G''$ does not contain a cycle of length at most $4$ distinct from $C_1$ and $C_2$.
Since $M$ and $M'$ are incident with both $x_1x_3$ and $x$, the minimality of $G$ and Lemma~\ref{lemma-cyl-le4} imply that $G''=T_1$.  But,
$T_1$ contains two edge-disjoint paths of length two between $C_1$ and $C_2$, and at most one of them contains $x_1x_3$.  It follows
that the distance between $C_1$ and $C_2$ in $G$ is at most $4$, contradicting (\ref{cl-far}).
\end{proof}

Let us assign the {\em initial charge} $c_0(v)=\deg(v)-4$ to each vertex and $c_0(F)=\ell(F)-4$ to each face of $G$
(including $C_1$ and $C_2$).  By Euler's formula, the sum of these charges is $-8$.  Now, each face of $G$ distinct
from $C_1$ and $C_2$ sends a charge of $1/3$ to each incident light vertex.  This way we obtain the final charge $c$.
Clearly, $c(v)\ge 0$ for each non-precolored vertex $v$, and $c(v)>0$ if $\deg(v)>4$.

\claim{cl-inner}{The final charge of each face $F$ of $G$ distinct from $C_1$ and $C_2$ is non-negative.  Furthermore,
if $F$ is incident with less than three light vertices, then $c(F)>0$.}
\begin{proof}
By (\ref{cl-no6}), $\ell(F)=5$, and thus $c_0(F)=1$.  If $F$ is incident with $k$ light vertices, then $c(F)=1-k/3$.
Furthermore, (\ref{cl-no53}) and (\ref{cl-no51}) imply that $k\le 3$, hence $c(F)\ge 0$, and if $k<3$, then $c(F)>0$.
\end{proof}

A face $F$ distinct from $C_1$ and $C_2$ is {\em $C_i$-close} (for $i\in\{1,2\}$) if $F$ shares an edge with $C_i$.
By (\ref{cl-far}) and (\ref{cl-no6}), a $C_1$-close face cannot share a vertex with a $C_2$-close face.

\claim{cl-outer}{For $i\in\{1,2\}$, the sum $S_i$ of the final charges of $C_i$ (if $C_i$ bounds a face), the vertices of $C_i$
and the $C_i$-close faces is at least $-4$, and if it is equal to $-4$, then $V(G)\setminus V(C_1 \cup C_2)$ contains a vertex of degree at least $5$.}
\begin{proof}
If $C_i$ is equal to a single vertex $v$, then by (\ref{cl-2conn}) its degree is at least $2$, thus $c(v)=-2>-4$.

Assume now that $C_i$ is a triangle $v_1v_2v_3$.  Then $c(C_i)=-1$. For $1\le j<k\le 3$, let $F_{jk}$ be the $C_i$-close face that
shares the edge $v_jv_k$ with $C_i$.  If all vertices of $C_i$ have degree at least three, then the final charge of each of them
is at least $-1$, and by (\ref{cl-inner}), $S_i\ge -4$.  Furthermore, if $S_i=-4$, then all vertices of $C_i$ have degree
exactly three, and all non-precolored vertices of $C_i$-near faces are light.  However, this implies that
$V(F_{12}\cup F_{23}\cup F_{13})\setminus V(C_i)$ induces a cycle $C$ of length $6$ consisting of light vertices.
Observe that every coloring of $G-V(C)$ extends to a coloring of $G$, contradicting the criticality of $G$.

Let us consider the case that say $v_1$ has degree two.
By (\ref{cl-no6}), $F_{12}=F_{13}$ is a $5$-face. Replacing the path $v_2v_1v_3$ in $F_{12}$ by $v_2v_3$ results
in a 4-cycle, contradicting (\ref{cl-cyc}).

Finally, assume that $C_i=v_1v_2v_3v_4$ has length $4$, and thus $c(C_i)=0$.  For $1\le j\le 4$, let
$F_j$ be the $C_i$-close face that shares the edge $v_jv_{j+1}$ with $C_i$ (where $v_5=v_1$).
If all vertices of $C_i$ have degree at least three, then $c(v_j)\ge -1$ for $1\le j\le 4$, and
$S_i\ge -4$.  Furthermore, $S_i=-4$ only if $\deg(v_j)=3$ for $1\le j\le 4$ and all non-precolored
vertices of $C_i$-close faces are light.  However, then $(F_1\cup F_2\cup F_3\cup F_4)-V(C_i)$
is a cycle $C$ of $8$ light vertices.  Observe that any coloring of $G-V(C)$ extends to a coloring of $G$,
contradicting the criticality of $G$.

Therefore, we may assume that $\deg(v_1)=2$.  By (\ref{cl-no22}), we have $\deg(v_2),\deg(v_4)\ge 3$.
Suppose now that $\deg(v_3)\ge 3$.  If at least one vertex of $C_i$ has degree greater than $3$,
then $S_i\ge c(v_1)+c(v_2)+c(v_3)+c(v_4)+c(F_1)>-4$.  Let us consider the case that $\deg(v_2)=\deg(v_3)=\deg(v_4)=3$.
By (\ref{cl-no6}), $(F_1\cup F_2\cup F_3)-V(C_i)$ is a $5$-cycle $w_1w_2w_3w_4w_5$,
where $w_1$ is adjacent to $v_2$, $w_3$ is adjacent to $v_3$ and $w_5$ is adjacent to $v_4$.
If $w_1$ and $w_5$ are light, then (\ref{cl-no6}) implies that $w_2$ and $w_4$ have a common
neighbor $x$ such that $w_4w_5w_1w_2x$ is a $5$-face.  Since $w_2$ has degree at least three, $x\neq w_3$,
and the $4$-cycle $w_2w_3w_4x$ contradicts (\ref{cl-cyc}).  Therefore, assume that say $\deg(w_1)\ge 4$.
Then $S_i\ge -5+c(F_1)+c(F_2)\ge -4$, and $S_i=-4$ only if $w_2$, $w_3$, $w_4$
and $w_5$ are light.
If that were the case and all vertices of $V(G)\setminus V(C_1 \cup C_2)$ had degree at most $4$, then $\deg(w_1)=4$.
Let $x$ be the neighbor of $w_1$ distinct from $w_2$, $w_5$ and $v_2$.  Let $G'=G-V(C_i)-\{w_1,w_2,w_3,w_4,w_5\}$,
and let $G''$ be a $(C_{3-i}\cup x)$-critical subgraph of $G'$ such that every precoloring of $C_{3-i}\cup x$ that
extends to $G''$ also extends to $G'$.  Note that the distance between $x$ and $C_{3-i}$ is at least three
and $G''$ does not contain a cycle of length at most $4$ distinct from $C_{3-i}$,
and thus by the minimality of $G$ and Lemma~\ref{lemma-cyl-le41}, $G''$ is trivial.
It follows that every precoloring of $x$ and $C_{3-i}$ extends to $G'$.  Let $\varphi'$ be a coloring
of $G'$ that matches $\varphi$ on $C_{3-i}$, such that $\varphi'(x)=\varphi(v_2)$.
Then $\varphi'\cup \varphi$ extends to a coloring of $G$, since every vertex of the $5$-cycle
$w_1w_2w_3w_4w_5$ has two available colors, and the lists of colors available at $w_3$ and $w_5$
are not the same.  This is a contradiction.

Finally, consider the case that $\deg(v_3)=2$.  If $\deg(v_2)=3$, then by (\ref{cl-no6}),
$(F_1\cup F_3)-\{v_1,v_2,v_3\}$ is a $4$-cycle, contradicting (\ref{cl-cyc}).  We conclude
that $\deg(v_2)\ge 4$, and by symmetry, $\deg(v_4)\ge 4$.  It follows that
$S_i\ge c(v_1)+c(v_2)+c(v_3)+c(v_4)+c(F_1)+c(F_2)>-4$.
\end{proof}

By (\ref{cl-inner}) and (\ref{cl-outer}), we have
$$-8=\sum_{v\in V(G)} c(v)+\sum_{F\in F(G)} c(F)\ge c(w)+S_1+S_2 > -8,$$
where $w$ is the vertex of $V(G)\setminus V(C_1 \cup C_2)$ of maximum degree.  This is a contradiction.
\end{proof}

Theorem~\ref{thm-cyl} follows from Lemmas~\ref{lemma-cyl-le4} and \ref{lemma-cyl-ge5}. 
Together with Lemmas~\ref{lemma-join} and \ref{lemma-conn} have the following corollary:

\begin{corollary}\label{cor-cyl}
Let $G$ be a graph embedded in the cylinder with boundaries $C_1$ and $C_2$ such that $\ell(C_1), \ell(C_2)\le 4$.
If $G$ is nontrivial $(C_1\cup C_2)$-critical and every cycle of length at most $4$ distinct from $C_1$ and $C_2$ separates $C_1$ from $C_2$,
then $G\in \CC$ or $G$ is one of the graphs drawn in Figures~\ref{fig-44}, \ref{fig-34},
\ref{fig-join2a}, \ref{fig-join2b}, \ref{fig-join2c}, \ref{fig-join3a}, \ref{fig-join3b}, \ref{fig-join3c}
or \ref{fig-joinmany}.
\end{corollary}

\section{The main result}\label{sec-main}

Theorem~\ref{thm-main} follows easily from Corollary~\ref{cor-cyl} and Theorem~\ref{thm-dkt}.
Let $G$ be a graph embedded in a surface $\Sigma$,
and let $\FF=\{F_1, F_2, \ldots, F_k\}$ be a subset of faces of $G$.  
We say that a subgraph $H$ of $G$  is {\em $\FF$-contractible}
if $H\not\in \FF$ and there exists a closed disk $\Delta\subseteq \Sigma$ such that $\Delta$ contains $H$, but $\Delta$ does not contain
any face of $\FF$.  For $F\in \FF$, we say that $H$ {\em surrounds $F$} if $H$ is not $\FF$-contractible and
there exists a closed disk $\Delta\subseteq \Sigma$ such that $\Delta$ contains $H$ and $F$, but no other face of $\FF$.
We say that a subgraph $H\subseteq G$ is {\em $\FF$-good} if $F_1\cup \ldots\cup F_k\subseteq H$ and
if $F$ is a face of $H$ that is not equal to a face of $G$, then $F$ has exactly two boundary walks, each of the walks has length $4$,
and the subgraph of $G$ drawn in the closed region corresponding to $F$ belongs to $\CC$.

Let $K$ be the constant from Theorem~\ref{thm-dkt}.  Let us note that $K > 8$.
Theorem~\ref{thm-main} follows trivially from Gr\"otzsch's theorem if $g=0$.
The following holds for graphs embedded in the cylinder:

\begin{lemma}\label{lemma-cylmain}
Let $G$ be a plane graph and $F_1$ and $F_2$ faces of $G$.  If $G$ is $(F_1\cup F_2)$-critical and
every cycle of length at most $4$ separates $F_1$ from $F_2$, then $G$ contains an $\{F_1,F_2\}$-good
subgraph with at most $\ell(F_1)+\ell(F_2)+4K+20$ vertices.
\end{lemma}
\begin{proof}
If $G$ does not contain a cycle of length at most $4$ distinct from $F_1$ and $F_2$, then by Theorem~\ref{thm-dkt}
we have $|V(G)|\le \ell(F_1)+\ell(F_2)+2K$, and we may set $H=G$. Otherwise, let $C_i$ be the cycle of length at most $4$ in $G$
such that the subgraph $G_i\subseteq G$ drawn between $F_i$ and $C_i$ is as small as possible, for $i\in \{1,2\}$.
By Theorem~\ref{thm-dkt}, $|V(G_i)|\le \ell(F_i)+\ell(C_i)+2K$.  Let $M$ be the subgraph of $G$ drawn between $C_1$ and $C_2$.
If $|V(M)|\le 20$, then $|V(G)|\le \ell(F_1)+\ell(F_2)+4K+20$, and we set $H=G$.  Suppose that $|V(M)|>20$.
If $C_1$ and $C_2$ are not vertex-disjoint, then there exists a subset $\Delta$ of the plane, disjoint with $F_1$ and $F_2$ and
homeomorphic to an open disk, such that the boundary of $\Delta$ is formed by a closed walk (of length at most $8$) in $C_1\cup C_2$ 
and all vertices of $M$ are contained in the closure of $\Delta$.  By a variant of Corollary~\ref{cor-cycles},
we would conclude that $V(M)=V(C_1\cup C_2)$, contrary to the assumption that $|V(M)|>20$. If $C_1$ and $C_2$ are vertex-disjoint,
then Corollary~\ref{cor-cyl} implies that $M\in\CC$, and we set $H=G_1\cup G_2$.
\end{proof}

Let $\alpha=21K+104$ and $\beta = 15K + 76$. 
For other surfaces, we prove the following generalization of Theorem~\ref{thm-main}:

\begin{theorem}\label{thm-maingen}
Let $G$ be a graph embedded in a surface $\Sigma$ of genus $g$ and let $\FF=\{F_1, F_2, \ldots, F_k\}$ be
a set of faces of $G$ such that the open region corresponding to $F_i$ is homeomorphic to the open disk for $1\le i\le k$.
Assume that $g\ge 1$ or $k\ge 3$.  If $G$ is $(F_1\cup F_2\ldots\cup F_k)$-critical and every $\FF$-contractible cycle has length at least $5$,
then $G$ has an $\FF$-good subgraph $H$ with at most $\ell(F_1)+\ldots+\ell(F_k) + \alpha g + \beta (k - 2) - 4$ vertices.
\end{theorem}
\begin{proof}
Let us prove the claim by the induction.  Let us assume that the claim is true for all graphs embedded in
surfaces of genus smaller than $g$, or embedded in $\Sigma$ with fewer than $k$ precolored faces.
Let $\ell=\ell(F_1)+\ldots+\ell(F_k)$.

Suppose first that $G$ contains a cycle $C\not\in \FF$ of length at most $4$ that does not surround any face in $\FF$.
Cut $\Sigma$ along $C$ and cap the resulting hole(s) by disk(s); the vertices and edges of $C$ are duplicated, resulting
in a graph $G'$.  Let us now discuss several cases:
\begin{itemize}
\item If the curve given by the drawing of $C$ in $\Sigma$ is one-sided, then $G'$ is embedded in a surface $\Sigma'$
of genus $g-1$.  Let $C_1$ be the face of $G'$ corresponding to $C$; note that $\ell(C_1)=2\ell(C)$.
Observe that $G'$ is $(\FF\cup \{C_1\})$-critical.  Thomassen~\cite{thom-torus} proved that every graph embedded in the projective plane without contractible cycles of length at most $4$
is $3$-colorable, and thus if $g=1$, then $k\ge 1$.  We conclude that if $g(\Sigma')=0$, then $|\FF\cup \{C_1\}|\ge 2$.

If $g(\Sigma')=0$ and $|\FF\cup \{C_1\}|=2$, then by Lemma~\ref{lemma-cylmain}, $G'$ has an $(\FF\cup \{C_1\})$-good subgraph $H'$ with
at most $\ell+\ell(C_1)+4K+20\le \ell + \alpha g + \beta(k-2)-4$ vertices.

Otherwise, we may apply the induction hypothesis, hence $G'$ has an $(\FF\cup \{C_1\})$-good subgraph $H'$
with at most $\ell + \ell(C_1) + \alpha (g-1) + \beta(k-1)-4 \le \ell+\alpha g+\beta (k-2) -4$ vertices.

In both cases, the graph $H\subseteq G$ obtained from $H'$ by identifying the corresponding vertices of $C_1$
is $\FF$-good, and has at most $\ell+\alpha g+\beta (k-2)-4$ vertices.

\item If $C$ is two-sided, then let $C_1$ and $C_2$ be the faces of $G'$ corresponding to $C$.  If $C$ is not separating,
then $G'$ is embedded in a surface of genus $g-2$.
If $g=2$ and $k=0$, then by Lemma~\ref{lemma-cylmain}, $G'$ has a $(\{C_1,C_2\})$-good subgraph $H'$ with
at most $\ell(C_1)+\ell(C_2)+4K+20\le \ell + \alpha g + \beta(k-2)-4$ vertices.  Otherwise, we can apply induction
hypothesis to $G'$ and conclude that it has a
$(\FF\cup \{C_1,C_2\})$-good subgraph $H'$ with at most $\ell + \ell(C_1) + \ell(C_2) + \alpha (g-2) + \beta k - 4\le \ell+\alpha g+\beta (k-2) -4$ vertices.
The graph $H\subseteq G$ obtained from $H'$ by identifying the corresponding vertices of $C_1$ and $C_2$ is
$\FF$-good and has at most $\ell+\alpha g+\beta (k-2)-4$ vertices.

\item Finally, if $C$ is two-sided and separating, then $G'$ consists of subgraphs $G_1$ and $G_2$ embedded in surfaces $\Sigma_1$ and $\Sigma_2$,
respectively, such that $g=g(\Sigma_1)+g(\Sigma_2)$.  Let $\FF_i$ be the subset of $\FF$ contained in $\Sigma_i$ and $k_i=|\FF_i|$, for $i\in \{1,2\}$.
Let $\ell_i=\sum_{F\in \FF_i} \ell(F)$.
Since $C$ is not $\FF$-contractible and does not surround a face of $\FF$, we have either $g(\Sigma_i)<g$, or $|\FF_i\cup \{C_i\}|<k$ for $i\in \{1,2\}$,
and furthermore, if $g(\Sigma_i)=0$, then $|\FF_i\cup \{C_i\}|\ge 3$.
By the induction hypothesis, $G_i$ has an $(\FF_i\cup \{C_i\})$-good
subgraph $H_i$ with at most $\ell_i+\ell(C_i)+\alpha g(\Sigma_i) + \beta(k_i-1) - 4$ vertices.
The graph $H\subseteq G$ obtained from $H_1$ and $H_2$ by identifying the corresponding vertices of $C_1$ and $C_2$
has at most $\ell+\ell(C_1)+\ell(C_2)-\ell(C)+\alpha g + \beta (k-2)-8\le \ell+\alpha g + \beta (k-2)-4$ vertices.
\end{itemize}

Therefore, we may assume that every cycle of length at most $4$ in $G$ surrounds a face in $\FF$.  Then, there exist cycles $C_1,\ldots, C_k\subseteq G$
such that
\begin{itemize}
\item for $1\le i\le k$, either $C_i=F_i$, or $\ell(C_i)\le 4$ and $C_i$ surrounds $F_i$,
\item if $\Delta_i$ is the open disk bounded by $C_i$ that contains the face $F_i$, then 
$\Delta_i\cap \Delta_j=\emptyset$ for $1\le i<j\le k$, and
\item the graph $G'$ obtained from $G$ by removing all vertices and edges contained in $\Delta_1\cup \ldots\cup \Delta_k$
contains no cycle of length at most $4$ distinct from $C_1$, $C_2$, \ldots, $C_k$.
\end{itemize}

Let $G_i$ be the subgraph of $G$ drawn in the closure of $\Delta_i$, for $1\le i\le k$.  Note that $G_i$ is $(F_i\cup C_i)$-critical,
and by Lemma~\ref{lemma-cylmain}, $G_i$ has an $(F_i\cup C_i)$-good subgraph $H_i$ with at most $\ell(F_i)+\ell(C_i)+4K+20$ vertices.
By Theorem~\ref{thm-dkt}, $|V(G')|\le \sum_{i=1}^k \ell(C_i) + K(g+k)$.  Note that $H=G'\cup H_1\cup \ldots\cup H_k$ is
$\FF$-good, and it has at most $\ell+ (4K+24)k + K(g+k)\le \ell + \alpha g + \beta(k-2) - 4$ vertices.
The previous inequality does not hold for $k=0$ and $g=1$. However, in this
case $G$ is a projective planar graph without contractible cycles of length at most 4 and hence $G$ is 3-colorable
by a result of Thomassen~\cite{thom-torus}.
\end{proof}

\begin{proof}[Proof of Theorem~\ref{thm-main}]
Follows from Gr\"otzsch's theorem and Theorem~\ref{thm-maingen}, with $f(g)=\alpha g$.
\end{proof}

\section{Programs}\label{sec-prog}
Both authors of the paper wrote independent programs implementing the algorithm following from Theorem~\ref{thm-alg},
as well as the programs to verify the claims of Theorem~\ref{thm-crit16} and Lemma~\ref{lemma-cyl-le4}.
The complete lists of the graphs, as well as programs used to
generate them can be found at
\oururl.
For the technical details describing the programs and their usage, see \texttt{README} files in
the subdirectories.
The subdirectory \texttt{dvorak} also contains the programs used to verify the claims of Section~\ref{sec-cyl}, which were
first derived manually without computer.

The most time-consuming part of the graph generation is criticality testing.  We applied the straightforward algorithm
following from the definition of the critical graph: given a planar graph $G$ with the outer face $B$, for each edge $e$
not incident with $B$ we tested whether there exists a precoloring of $B$ that does not extend to $G$, but extends to $G-e$.
We augmented this algorithm with a few simple heuristics to speed it up (e.g., all vertices in $V(G)\setminus V(B)$ must have degree
at least three).  Generating the set $\KK_{16}$ took about $10$ minutes on a 2.67GHz machine.  We believe that
by parallelization and possibly using a more clever criticality testing algorithm, it would be possible to generate the
graphs at least up to $\KK_{20}$, if someone would need them.

\end{document}